\documentclass[a4paper,12pt,reqno]{amsart}
\usepackage{amssymb}
\usepackage[margin=1in,includehead,includefoot]{geometry}
\usepackage{enumitem}
\usepackage{graphicx,color}
\usepackage[bookmarksdepth=2]{hyperref}

\numberwithin{equation}{section}
\DeclareRobustCommand{\SkipTocEntry}[5]{}

\theoremstyle{plain}
\newtheorem{theorem}{Theorem}[section]
\newtheorem*{theorem*}{Theorem}
\newtheorem{lemma}[theorem]{Lemma}
\newtheorem{corollary}[theorem]{Corollary}
\newtheorem{proposition}[theorem]{Proposition}

\theoremstyle{definition}
\newtheorem{definition}[theorem]{Definition}
\newtheorem{example}[theorem]{Example}
\newtheorem{remark}[theorem]{Remark}

\newtheorem{conjecture}[theorem]{Conjecture}

\theoremstyle{remark}

\DeclareMathOperator{\re}{Re}
\DeclareMathOperator{\im}{Im}
\DeclareMathOperator{\Li}{Li_2}
\DeclareMathOperator{\Arg}{Arg}
\DeclareMathOperator{\Res}{Res}
\DeclareMathOperator{\sign}{sign}
\DeclareMathOperator{\esssupp}{ess\,supp}

\newcommand{\D}{d}

\newcommand{\dom}{\mathcal{D}}

\newcommand{\pr}{\mathbf{P}}
\newcommand{\ex}{\mathbf{E}}
\newcommand{\C}{\mathbf{C}}
\newcommand{\R}{\mathbf{R}}

\newcommand{\fourier}{\mathcal{F}}
\newcommand{\laplace}{\mathcal{L}}

\newcommand{\ind}{\mathbf{1}}
\newcommand{\sub}{\subseteq}
\newcommand{\ph}{\varphi}
\newcommand{\ro}{\varrho}
\newcommand{\eps}{\varepsilon}

\newcommand{\tscalar}[1]{\langle #1 \rangle}
\newcommand{\set}[1]{\left\{ #1 \right\}}
\newcommand{\tset}[1]{\{ #1 \}}
\newcommand{\abs}[1]{\left| #1 \right|}
\newcommand{\expr}[1]{\left( #1 \right)}

\newcommand{\itref}[1]{\ref{#1}}
\newcommand{\hl}{{(0, \infty)}}

\newcommand{\rogers}{{\mathcal{R}}}
\newcommand{\cbf}{{\mathcal{CBF}}}
\newcommand{\thet}{\vartheta}
\newcommand{\ua}{\uparrow}
\newcommand{\da}{\downarrow}
\newcommand{\ra}{\rightarrow}
\newcommand{\la}{\leftarrow}
\newcommand{\pvint}{\mathrm{PV}\!\!\int}
\newcommand{\ignore}[1]{}

\usepackage{ifthen}
\newcommand{\formula}[2][nolabel]%
{%
 \ifthenelse{\equal{#1}{nolabel}}%
 {\begin{align*} #2 \end{align*}}%
 {%
  \ifthenelse{\equal{#1}{}}%
  {\begin{align} #2 \end{align}}%
  {\begin{align} \label{#1} \begin{aligned} #2 \end{aligned} \end{align}}%
 }%
}

\begin{document}

%
%

\title[Rogers functions]{Rogers functions and fluctuation theory}
\author{Mateusz Kwa{\'s}nicki}
\thanks{Work supported by NCN grant no. 2011/03/D/ST1/00311 and the Foundation for Polish Science}
\address{Mateusz Kwa{\'s}nicki \\ Institute of Mathematics and Computer Science \\ Wroc{\l}aw University of Technology \\ ul. Wybrze{\.z}e Wyspia{\'n}\-skiego 27 \\ 50-370 Wroc{\l}aw, Poland}
\email{mateusz.kwasnicki@pwr.wroc.pl}
\date{\today}
\begin{abstract}
Extending earlier work by Rogers, Wiener--Hopf factorisation is studied for a class of functions closely related to Nevanlinna--Pick functions and complete Bernstein functions. The name \emph{Rogers functions} is proposed for this class. Under mild additional condition, for a Rogers function $f$, the Wiener--Hopf factors of $f(\xi) + \tau$, as well as their ratios, are proved to be complete Bernstein functions in both $\xi$ and $\tau$. This result has a natural interpretation in fluctuation theory of L\'evy processes: for a L\'evy process $X_t$ with completely monotone jumps, under mild additional condition, the Laplace exponents $\kappa^\ua(\tau; \xi)$, $\kappa^\da(\tau; \xi)$ of ladder processes are complete Bernstein functions of both $\xi$ and $\tau$. Integral representation for these Wiener--Hopf factors is studied, and a semi-explicit expression for the space-only Laplace transform of the supremum and the infimum of $X_t$ follows.
\end{abstract}
\maketitle


%
%

\section{Introduction and statement of the results}
\label{sec:intro}

The main subject of this paper is the Wiener--Hopf factorisation a class of Nevanlinna--Pick-type functions which, to the author's knowledge, first appeared in the Rogers's article~\cite{bib:r83}. Functions of this type will be called \emph{Rogers functions}, and the set of Rogers functions will be denoted by $\rogers$. In~\cite{bib:jf12} a very closely related notion is called \emph{asymmetric complete Bernstein function}, because if $g$ is a complete Bernstein function ($g \in \cbf$ in short) then $g(\xi^2)$ is a \emph{symmetric} Rogers function.

Importance of complete Bernstein functions (known also as operator monotone functions) in various areas of mathematics (see~\cite{bib:ssv10}) suggests that also Rogers functions can find numerous applications; for example, one can expect that some properties of the operators \smash{$g(-\tfrac{\D^2}{\D x^2})$} for $g \in \cbf$ can be extended to asymmetric operators \smash{$f(-i\tfrac{\D}{\D x})$} for $f \in \rogers$. Our primary motivation for their study originates in the fluctuation theory of L\'evy processes, and two possible applications are given in Theorems~\ref{th:kappa:cbf} and~\ref{th:kappa:inv}. Nevertheless, some knowledge of L\'evy processes and their fluctuation theory is required only in the last section of this article.

A function $f$ holomorphic in $\C_\ra = \{ \xi \in \C : \re \xi > 0 \}$ is a Rogers function if $\re (f(\xi) / \xi) \ge 0$ for all $\xi \in \C_\ra$. Equivalently, $-f(-i \xi) / \xi$ is a Nevanlinna--Pick function, that is, it maps $\C_\ua = \{ \xi \in \C : \im \xi > 0 \}$ into the closure of this set.

For readers familiar with L\'evy processes, suppose that $X_t$ is a L\'evy process and $\Psi$ is the L\'evy--Khintchine exponent of $X_t$. Then $\Psi(\xi) = f(\xi)$ for $\xi > 0$ for some $f \in \rogers$ if and only if the L\'evy measure of $X_t$ is absolutely continuous with respect to the Lebesgue measure, and its density function $\nu(x)$, as well as $\nu(-x)$, are completely monotone on $(0, \infty)$. Conversely, every Rogers function corresponds, up to an additive constant, to such a L\'evy process.

The class of Rogers functions is clearly isomorphic to the class of Nevanlinna--Pick functions, and many properties of the former are direct analogues of known results for the latter. Nevertheless, the Wiener--Hopf theory developed in~\cite{bib:r83} and extended below appears to be specific to Rogers functions, and has no clearly available counterpart for Nevanlinna--Pick functions. For this reason, it seems appropriate to coin a separate name, \emph{Rogers functions}, for the class considered herein.

The main goal of this article is to develop the theory of Rogers functions, and in particular their Wiener--Hopf factorisation. An application to fluctuation theory of L\'evy processes presented in the following two theorems is a concise summary of main results. The description of intermediate results is given later in the introduction.

A Rogers function $f$ is said to be \emph{balanced} if the set $\{\xi \in \C_\ra : f(\xi) \in [0,  \infty)\}$ is contained in $\{\xi : |\Arg \xi| < \tfrac{\pi}{2} - \eps\}$ for some $\eps > 0$ (see Definition~\ref{def:rogers:reg} and Lemma~\ref{lem:rogers:reg}). For example, the L\'evy--Khintchine exponents of all strictly stable L\'evy processes $X_t$ which are not monotone in $t$ (that is, neither $X_t$ nor $-X_t$ is a subordinator) are balanced Rogers functions. A more general condition that $f \in \rogers$ is \emph{nearly balanced} is discussed in detail in Section~\ref{sec:wh} (Definition~\ref{def:rogers:nearly}).

\begin{theorem}
\label{th:kappa:cbf}
If $X_t$ is a L\'evy process whose L\'evy--Khintchine exponent $\Psi$ is a \emph{nearly balanced} Rogers function, then the Laplace exponents $\kappa^\ua(\tau; \xi)$, $\kappa^\da(\tau; \xi)$ of the ascending and descending ladder processes for $X_t$ are complete Bernstein functions in both $\xi$ and $\tau$. Furthermore, if $0 \le \tau_1 \le \tau_2$ and $0 \le \xi_1 \le \xi_2$, then
\formula{
 \frac{\kappa^\ua(\tau_1; \xi)}{\kappa^\ua(\tau_2; \xi)} \, , && \frac{\kappa^\da(\tau_1; \xi)}{\kappa^\da(\tau_2; \xi)} \, , && \frac{\kappa^\ua(\tau; \xi_1)}{\kappa^\ua(\tau; \xi_2)} \, , && \frac{\kappa^\da(\tau; \xi_1)}{\kappa^\da(\tau; \xi_2)}
}
are complete Bernstein functions of $\xi$ and $\tau$, respectively.
\end{theorem}

\begin{remark}
For a detailed discussion of the Wiener--Hopf factors $\kappa^\ua(\tau; \xi)$ and $\kappa^\da(\tau; \xi)$ (often denoted $\kappa(\tau; \xi)$ and $\hat{\kappa}(\tau; \xi)$ in literature) in the present context, see Section~\ref{sec:ft}. General information on Wiener--Hopf factorisation for L\'evy processes can be found in~\cite{bib:bnmr01, bib:b96, bib:d07, bib:f74, bib:k06, bib:s99}. In analytical terms, $\kappa^\ua(\tau; \xi)$ and $\kappa^\da(\tau; \xi)$ are holomorphic functions of $\xi$ in the region $\C_\ra = \{ \xi \in \C : \re \xi > 0\}$ and continuous on the boundary, which satisfy $\kappa^\ua(\tau; -i \xi) \kappa^\da(\tau; i \xi) = \Psi(\xi) + \tau$ and are normalised so that $\kappa^\ua(\tau_1, \xi) / \kappa^\ua(\tau_2; \xi)$ and $\kappa^\da(\tau_1, \xi) / \kappa^\da(\tau_2; \xi)$ converge to $1$ as $\xi \nearrow \infty$ for all $\tau_1, \tau_2 \ge 0$.

The functions $\kappa^\ua(\tau; \xi)$ and $\kappa^\da(\tau; \xi)$ are Laplace exponents of ladder processes, and are closely related to the distribution of the supremum and infimum functionals for $X_t$. Theorem~\ref{th:kappa:cbf} and its extensions can therefore yield detailed description of these important objects. The first step in this direction is contained in Theorem~\ref{th:kappa:inv}. Related problems were recently studied by many authors, mostly for strictly stable L\'evy processes; see~\cite{bib:bdp08, bib:dr11, bib:ds10, bib:gj10, bib:gj12, bib:hk11, bib:ku10, bib:ku11, bib:m13} and the references therein.
\end{remark}

\begin{remark}
It is relatively simple to show that if $\Psi$ is a Rogers function and $X_t$ is monotone in $t$, then the assertion of Theorem~\ref{th:kappa:cbf} is also true. In a sense, in this case $\Psi$ is a completely imbalanced Rogers function. Therefore, it is natural to conjecture that Theorem~\ref{th:kappa:cbf} holds for general Rogers functions $\Psi$, without assuming that $\Psi$ is nearly balanced. For $X_t$ with one-sided jumps, this conjecture is shown to be true in~\cite{bib:bkkk14}.

The class of functions defined by the assertion of Theorem~\ref{th:kappa:cbf} could be called \emph{bivariate} complete Bernstein functions. The representation and properties of such functions may be of independent interest.
\end{remark}

\begin{remark}
Parts of Theorem~\ref{th:kappa:cbf} seem to be new even for symmetric processes $X_t$, studied in~\cite{bib:kmr12, bib:kmr13}. These articles suggest possible applications of Theorem~\ref{th:kappa:cbf}. A closely related concept is the generalised eigenfunction expansion of the generator and transition operators of $X_t$ killed upon leaving a half-line, developed in~\cite{bib:k11} in the symmetric setting. See Section~\ref{sec:ft} for further discussion.
\end{remark}

\begin{theorem}
\label{th:kappa:inv}
If $X_t$ is a L\'evy process whose L\'evy--Khintchine exponent $\Psi$ is a \emph{balanced} Rogers function, then the Laplace transforms of the supremum and infimum functionals
\formula{
 X^\ua_t & = \sup \tset{X_s : s \in [0, t]} , & X^\da_t & = \inf \tset{X_s : s \in [0, t]}
}
have Laplace transforms
\formula[eq:kappa:inv]{
 \ex \exp(-\xi X^\ua_t) & = \frac{1}{\pi} \int_0^\infty \Psi_{r\ua}(\xi) \, \frac{\xi \re \zeta(r)}{\xi^2 - 2 \xi \im \zeta(r) + r^2} \, \frac{\lambda'(r)}{\lambda(r)} \, e^{-t \lambda(r)} \D r \\
 \ex \exp(\xi X^\da_t) & = \frac{1}{\pi} \int_0^\infty \Psi_{r\da}(\xi) \, \frac{\xi \re \zeta(r)}{\xi^2 + 2 \xi \im \zeta(r) + r^2} \, \frac{\lambda'(r)}{\lambda(r)} \, e^{-t \lambda(r)} \D r .
}
for $t, \xi > 0$. In the above formulae, $\zeta(r) \in \C_\ra$ is the unique point such that $|\zeta(r)| = r$ and $\Psi(\zeta(r))$ is a real number (where $\Psi$ is the holomorphic extension to $\C_\ra$ of the L\'evy--Khintchine exponent), $\lambda(r) = \Psi(\zeta(r))$, and
\formula[eq:kappa:inv:ratio]{
 \Psi_{r\ua}(\xi) & = \lim_{\eps \searrow 0} \exp \Biggl( \frac{1}{\pi} \int_0^\infty \re \Biggl( \expr{\frac{\zeta'(s)}{\xi + i \zeta(s)} - \frac{\zeta'(s)}{\eps + i \zeta(s)}} \times \\
 & \hspace*{9em} \times \log \frac{(\zeta(s) - \zeta(r)) (\zeta(s) + \overline{\zeta(r)})}{\lambda(s) - \lambda(r)} \Biggr) \D s \Biggr) , \\
 \Psi_{r\da}(\xi) & = \lim_{\eps \searrow 0} \exp \Biggl( \frac{1}{\pi} \int_0^\infty \re \Biggl( \expr{\frac{\zeta'(s)}{\xi - i \zeta(s)} - \frac{\zeta'(s)}{\eps - i \zeta(s)}} \times \\
 & \hspace*{9em} \times \log \frac{(\zeta(s) - \zeta(r)) (\zeta(s) + \overline{\zeta(r)})}{\lambda(s) - \lambda(r)} \Biggr) \D s \Biggr) .
}
\end{theorem}

\begin{remark}
For the application of Theorem~\ref{th:kappa:inv} to strictly stable L\'evy processes, see Theorem~\ref{th:kappa:stable}. In this case the limit in~\eqref{eq:kappa:inv:ratio} can be calculated by substituting $\eps = 0$ in the integral, and adding an extra term corresponding to the Dirac measure $c \delta_0(\D s)$ for an appropriate constant $c$. Similar procedure is applicable whenever $\arg \zeta(r)$ has a limit as $r \searrow 0$.
\end{remark}

\begin{remark}
Over the last decade, symmetric L\'evy processes (in particular those whose L\'evy--Khintchine exponent $\Psi$ is a Rogers function, that is, $\Psi(\xi) = g(\xi^2)$ for some $g \in \cbf$) attracted much attention; see, e.g.,~\cite{bib:bbkrsv09, bib:bgr13, bib:bgr14, bib:g13, bib:ksv12, bib:ksv13, bib:ksv14, bib:kkms10, bib:k11, bib:k12a, bib:k12, bib:ssv10} and the references therein for recent results in potential theory and spectral theory of these processes. There are, however, very few papers where similar problems are studied for asymmetric processes, see~\cite{bib:p14, bib:y12}. In the present paper the theory of Rogers functions is developed and applied to fluctuation theory of L\'evy processes. One can expect that the results of this article may be applied also to potential theory and spectral theory, thus stimulating its development in the asymmetric setting.
\end{remark}

The structure of the article is as follows. To facilitate reading, in Section~\ref{sec:ex} main ideas of the article are presented for the power-type Rogers functions, which correspond to strictly stable L\'evy processes. This part also contains a brief discussion of various parametrisations of strictly stable L\'evy processes, introduction to their fluctuation theory, statement of Theorem~\ref{th:kappa:inv} for power-type Rogers functions (Theorem~\ref{th:kappa:stable}) and discussion of some special cases.

Preliminary results, including classical theorems on harmonic and holomorphic functions, are collected in Section~\ref{sec:pre}. In Section~\ref{sec:cbf}, the notion of complete Bernstein functions is briefly recalled.

Section~\ref{sec:rogers} contains equivalent definitions of Rogers functions. Theorem~\ref{th:rogers} identifies the class of L\'evy--Khintchine exponents of L\'evy processes with completely monotone jumps with Stieltjes-type transforms, functions with a Stieltjes-type exponential representation and a class of Nevanlinna--Pick-type functions. At the end of this section, a few examples of Rogers functions are given.

Basic properties of Rogers functions are studied in Section~\ref{sec:basic}. First, non-zero, non-constant, non-degenerate, bounded and unbounded Rogers functions are defined. Next, the class $\rogers$ is proved to be closed under various transformations, for example, if $f \in \rogers$ and $f$ is not the zero function, then $\xi^2 / f(\xi)$, $\xi^2 f(1/\xi)$ and $1 / f(1/\xi)$ are Rogers functions. Finally, some estimates are studied. For example, it is proved that for $f \in \rogers$ and $\xi > 0$,
\formula{
 \frac{1}{\sqrt{2}} \, \frac{|\xi|^2}{1 + |\xi|^2} \le \frac{|f(\xi)|}{|f(1)|} \le \sqrt{2} \, (1 + |\xi|^2)
}
(Proposition~\ref{prop:rogers:bound}), and an upper bound for $|\xi f'(\xi) / f(\xi)|$ is found (Proposition~\ref{prop:rogers:prime}).

The line $\gamma_f$ along which a Rogers function $f$ takes real values is studied in Section~\ref{sec:real}. The key step is to show that if $\zeta \in \C_\ra$ and $f(\zeta) > 0$, then
\formula{
 f_{[\zeta]}(\xi) & = \frac{(\xi - \zeta) (\xi + \bar{\zeta})}{f(\xi) - f(\zeta)}
}
is a Rogers functions (Lemma~\ref{lem:rogers:quot}). Then $\gamma_f$ is defined so that $\gamma_f \cap \C_\ra$ is the set of those $\zeta \in \C)\ra$, for which $f(\zeta) > 0$. It is proved that $\gamma_f$ is a system of simple analytic curves (typically a single curve), and that for each $r > 0$, $\gamma_f \cap \C_\ra$ contains at most one point $\zeta$ with $|\zeta| = r$.

Next three sections discuss the Wiener--Hopf factorisation of Rogers functions. In Section~\ref{sec:wh} a short proof of the Rogers's result is provided: $f \in \rogers$ if and only if $f(\xi) = f^\ua(-i \xi) f^\da(i \xi)$ for complete Bernstein functions $f^\ua$, $f^\da$ (Theorem~\ref{th:rogers:wh}). After proving some simple properties of Wiener--Hopf factors $f^\ua$, $f^\da$, the notion of a balanced Rogers function is introduced (Definition~\ref{def:rogers:reg} and Lemma~\ref{lem:rogers:reg}). Various formulae for the Wiener--Hopf factors are provided in Lemmas~\ref{lem:rogers:curve} and~\ref{lem:rogers:wh:alt}, Corollaries~\ref{cor:rogers:wh} and~\ref{cor:rogers:wh:alt}, and Remark~\ref{rem:rogers:substitution}. The most classical expression is
\formula{
 f^\ua(\xi) & = c_\ua \exp \expr{\frac{1}{2 \pi} \int_{-\infty}^\infty \expr{\frac{1}{\xi + i r} - \frac{1}{1 + i r}} \log f(r) \D r} , \\
 f^\da(\xi) & = c_\da \exp \expr{\frac{1}{2 \pi} \int_{-\infty}^\infty \expr{\frac{1}{\xi - i r} - \frac{1}{1 - i r}} \log f(r) \D r} .
}
In the remaining part of this section, a technical lemma is proved, and nearly balanced Rogers functions are discussed.

The functions $\kappa^\ua(\tau; \xi)$ and $\kappa^\da(\tau; \xi)$ in Theorem~\ref{th:kappa:cbf} are, for a fixed $\tau$, the Wiener--Hopf factors of $f(\xi) + \tau$. The Wiener--Hopf factors are defined up to multiplication by a positive constant, which here may depend on $\tau$. Section~\ref{sec:xwh} contains those results on the Wiener--Hopf factors of $f(\xi) + \tau$ which do not depend on the choice of the $\tau$-dependent constant. This is the most technical part of the article, where holomorphic extensions to $\tau \in \C \setminus (-\infty, 0)$ of the Wiener--Hopf factors are studied (Lemma~\ref{lem:rogers:xwh}). Noteworthy, the boundary limits of these Wiener--Hopf factors are expressed using the Wiener--Hopf factors of $f_{[\zeta]}$ (Lemma~\ref{lem:rogers:xwh:boundary}). The results are first proved for balanced Rogers functions, and then extended to nearly balanced ones.

In Section~\ref{sec:norm} the proper choice of the $\tau$-dependent constant is done, which leads to an analytical construction of $\kappa^\ua(\tau; \xi)$ and $\kappa^\da(\tau; \xi)$. Main results follow now quickly from the theory developed in Section~\ref{sec:xwh}. A few examples are discussed at the end of Section~\ref{sec:norm}.

Finally, Section~\ref{sec:ft} starts with a short discussion of fluctuation theory of L\'evy processes. In particular, it is proved that the choice of $\tau$-dependent constant from Section~\ref{sec:norm} coincides with the classical probabilistic definition. Theorem~\ref{th:kappa:cbf} follows immediately. Nest, the proofs of Theorem~\ref{th:kappa:inv} and its version for power-type Rogers functions (Theorem~\ref{th:kappa:stable}) are provided. In the final part, inversion of the Laplace transform in~\eqref{eq:kappa:inv} is discussed, and connection to generalised eigenfunction expansion is given. The formula for the conjectured generalised eigenfunctions for power-type Rogers functions is provided (Lemma~\ref{lem:ee:stable}).

The following notation is used throughout the article. The right, left, upper and lower complex half-planes are denoted by $\C_\ra$, $\C_\la$, $\C_\ua$ and $\C_\da$, respectively. Their closures are denoted by $\overline{\C}_\ra$ etc. By $i \R$ we denote the imaginary axis. By $(z_1, z_2)$ and $[z_1, z_2]$ we denote line segments in the complex plane. In a similar manner, for $z \in \C \setminus \{0\}$, $(0, z \infty)$ denotes the ray with initial point $0$, which contains $z$. The complex plane $\C$ is commonly identified with $\R^2$. The Fourier transform of an integrable function is defined by $\fourier f(\xi) = \int_{-\infty}^\infty e^{-i \xi x} f(x) \D x$, and the Laplace transform is denoted by $\laplace f(\xi) = \int_0^\infty e^{-\xi x} f(x) \D x$. Finally, $f(0^+) = \lim_{\xi \searrow 0} f(\xi)$ and $f(\infty^-) = \lim_{\xi \nearrow \infty} f(\xi)$.

%
%

\section{Power-type Rogers functions}
\label{sec:ex}

The main part of the article is rather technical. For this reason, in this section key ideas are discussed for a still technical, but at least explicit example: $f(\xi) = a \xi^\alpha$ with $\alpha \in (0, 2]$ and $a \in \C \setminus \{0\}$ satisfying
\formula[]{
\label{eq:stable:a}
 |\Arg a| & \le \tfrac{\alpha \pi}{2} && \text{and} & |\Arg a| & \le (2 - \alpha) \tfrac{\pi}{2} .
}
This example is motivated by its probabilistic counterpart. Namely, $f(\xi) = a \xi^\alpha$ corresponds to the L\'evy--Khintchine exponent $\Psi$ of a strictly stable L\'evy process (in this section, $f$ and $\Psi$ are clearly distinguished), which is defined by
\formula[]{
\label{eq:stable:psi}
 \text{$\Psi(\xi) = a \xi^\alpha$ for $\xi > 0$,} && \text{$\Psi(\xi) = \bar{a} (-\xi)^\alpha$ for $\xi < 0$.}
}

\subsection{Parametrisations}

For $\alpha \ne 1$ it is customary parametrise strictly stable L\'evy processes by $(\alpha, \beta, k)$, where $\alpha \in (0, 2]$ is the \emph{stability index}, $\beta \in [-1, 1]$ is the \emph{skewness parameter} ($\beta = 0$ if $\alpha = 2$), and $k > 0$ is the \emph{scale parameter}. These parameters are related to $a$ through the formulae
\formula[]{
\label{eq:stable:a:beta}
 a & = k^\alpha (1 - i \beta \tan \tfrac{\alpha \pi}{2}) , & k & = (\re a)^{1/\alpha} , & \beta & = -\frac{\tan(\Arg a)}{\tan \tfrac{\alpha \pi}{2}} \, .
}
A different parametrisation is required for $\alpha = 1$ (the skewness parameter is then nonzero for stable L\'evy processes which are not strictly stable). Another common choice for parameters is $(\alpha, \ro, k)$, where $\alpha$ and $k$ are as above and $\ro = \pr(X_t > 0)$ is the positivity parameter. Then
\formula[]{
\label{eq:stable:a:ro}
 a & = k^\alpha (1 - i \tan((2 \ro - 1) \tfrac{\alpha \pi}{2})) , & \ro & = \tfrac{1}{2} - \tfrac{1}{\alpha \pi} \Arg a ,
}
and $\ro \in [0, 1]$ if $\alpha \in (0, 1]$, $\ro \in [1 - \tfrac{1}{\alpha}, \tfrac{1}{\alpha}]$ if $\alpha \in (1, 2]$. For $\alpha \ne 1$,
\formula[eq:stable:ro]{
 \ro & = \tfrac{1}{2} + \tfrac{1}{\alpha \pi} \arctan(\beta \tan \tfrac{\alpha \pi}{2}) .
}
If $|\Arg a| = \tfrac{\alpha \pi}{2}$ (and necessarily $\alpha \in (0, 1]$), then $X_t$ is monotone in $t$ (that is, $X_t$ or $-X_t$ is a subordinator). This corresponds to $\ro \in \{0, 1\}$, or to $|\beta| = 1$ (when $\alpha \ne 1$). If $|\Arg a| = (2 - \alpha) \tfrac{\pi}{2}$ and $\alpha \in (1, 2)$, then $X_t$ has one-sided jumps, but it is not monotone in $t$. This corresponds to $\ro \in \{1 - \tfrac{1}{\alpha}, \tfrac{1}{\alpha}\}$, or to $|\beta = 1|$. When $\alpha = 2$, then necessarily $a \in (0, \infty)$ and $X_t$ is the Brownian motion.

Yet another parametrisation is based on separation of negative and positive jumps. For $\alpha \in (0, 1)$, $\Psi(\xi) = c_\ua (-i \xi)^\alpha + c_\da (i \xi)^\alpha$ for some $c_\ua, c_\da \ge 0$. Then $a = e^{-i \alpha \pi/2} c_\ua + e^{i \alpha \pi/2} c_\da$. If $\alpha \in (1, 2)$, then $f(\xi) = -c_\ua (-i \xi)^\alpha - c_\da (i \xi)^\alpha$ for some $c_\ua, c_\da \ge 0$. In this case $a = e^{i \pi - i \alpha \pi/2} c_\ua + e^{-i \pi + i \alpha \pi/2} c_\da$. Furthermore,
\formula{
 \beta & = \frac{c_\ua - c_\da}{c_\ua + c_\da} \, , & \ro & = \frac{1}{2} + \frac{1}{\alpha \pi} \arctan \expr{\frac{c_\ua - c_\da}{c_\ua + c_\da} \, \tan \frac{\alpha \pi}{2}} ;
}
here $\alpha \in (0, 1) \cup (1, 2)$. When $\alpha = 1$, $a = c - i b$ for $c \ge 0$ and $b \in \R$. If $b = 0$, then $X_t$ is the Cauchy process. If $c = 0$, then $X_t$ is the (deterministic) process of uniform motion with velocity $b$.

\subsection{Wiener--Hopf factorisation}

For each $\tau > 0$, the Wiener--Hopf factors $\kappa^\ua(\tau; \xi)$ and $\kappa^\da(\tau; \xi)$ are holomorphic functions of $\xi \in \C_\ra$ such that $\kappa^\ua(\tau; -i \xi) \kappa^\da(\tau; i \xi) = \Psi(\xi) + \tau$ for $\xi \in \R$. Wiener--Hopf theory states that this condition defines $\kappa^\ua(\tau; \xi)$ and $\kappa^\da(\tau; \xi)$ uniquely up to a positive constant, which may depend on $\tau$:
\formula[eq:stable:kappa:pre]{
 \kappa^\ua(\tau; \xi) & = c_\ua(\tau) \exp\expr{-\frac{1}{2 \pi i} \int_{-\infty}^\infty \expr{\frac{1}{i \xi - z} - \frac{1}{i - z}} \log(\Psi(z) + \tau) \D z} , \\
 \kappa^\da(\tau; \xi) & = c_\da(\tau) \exp\expr{-\frac{1}{2 \pi i} \int_{-\infty}^\infty \expr{\frac{1}{i \xi + z} - \frac{1}{i + z}} \log(\Psi(z) + \tau) \D z}
}
for $\tau > 0$ and $\xi \in \C_\ra$. In probability, however, there is a natural choice of $c_\ua(\tau)$, which is unique up to a positive constant $c$. A fundamental result in fluctuation theory of L\'evy processes (often referred to as \emph{Fristedt formula}, but already present in works of Pecherski and Rogozin; see~\cite{bib:f74, bib:pr69, bib:r66}) states that
\formula[eq:stable:f]{
 \kappa^\ua(\tau; \xi) & = c \exp\expr{\int_0^\infty \int_{(0, \infty)} \frac{e^{-t} - e^{-\tau t - \xi x}}{t} \, \pr(X_t \in \D x) \D t} , \\
 \kappa^\da(\tau; \xi) & = \frac{1}{c} \, \exp\expr{\int_0^\infty \int_{(-\infty, 0)} \frac{e^{-t} - e^{-\tau t + \xi x}}{t} \, \pr(X_t \in \D x) \D t} ;
}
Using the fact that the Fourier transform of $\pr(X_t \in \D x)$ is $e^{-t \Psi(\xi)}$, one can derive an expression resembling the Baxter--Donsker formula (see~\cite{bib:bd57}),
\formula[eq:stable:wh]{
 \kappa^\ua(\tau; \xi) & = c \exp\expr{-\frac{1}{2 \pi i} \int_{-\infty}^\infty \expr{\frac{\log(\Psi(z) + \tau)}{i \xi - z} - \frac{\log \Psi(z)}{i - z}} \D z} , \\
 \kappa^\da(\tau; \xi) & = \frac{1}{c} \, \exp\expr{-\frac{1}{2 \pi i} \int_{-\infty}^\infty \expr{\frac{\log(\Psi(z) + \tau)}{i \xi + z} - \frac{\log \Psi(z)}{i + z}} \D z}
}
for all $\tau > 0$ and $\xi \in \C_\ra$. An alternative argument is provided in Corollary~\ref{cor:rogers:kappa:a} (see also Section~\ref{sec:ft}), so we omit the details. 

Since $\Psi(k \xi) = k^\alpha \Psi(\xi)$ for $k > 0$ and $\xi \in \R$, the process $X_{k t}$ has the same law as $k^{1/\alpha} X_t$. By~\eqref{eq:stable:f} one can prove that
\formula[eq:stable:scaling]{
 \kappa^\ua(\tau; \xi) & = \tau^\ro \kappa^\ua(1; \tau^{-1/\alpha} \xi) = \xi^{\ro \alpha} \kappa^\ua(\xi^{-\alpha} \tau; 1) , \\
 \kappa^\da(\tau; \xi) & = \tau^{1 - \ro} \kappa^\da(1; \tau^{-1/\alpha} \xi) = \xi^{(1 - \ro) \alpha} \kappa^\da(\xi^{-\alpha} \tau; 1) ,
}
where $\ro$ is the positivity parameter. (This property also follows from~\eqref{eq:stable:wh}, for a similar calculation see the proof of Theorem~\ref{th:kappa:stable} in Section~\ref{sec:ft}.)

\subsection{Main ideas}

Theorem~\ref{th:kappa:cbf} implies that if $|\Arg a| < \tfrac{\alpha \pi}{2}$ (see~\eqref{eq:stable:a}), then $\kappa^\ua(\tau; \xi)$ and $\kappa^\da(\tau; \xi)$ are complete Bernstein functions of both $\tau$ and $\xi$. In particular, this means that for each $\tau > 0$, $\kappa^\ua(\tau; \xi)$ and $\kappa^\da(\tau; \xi)$ extend to holomorphic functions of $\xi \in \C \setminus (-\infty, 0]$, and in a similar manner, for each $\xi > 0$, these functions extend to holomorphic functions of $\tau \in \C \setminus (-\infty, 0]$. These properties, however, do not follow clearly neither from~\eqref{eq:stable:wh}, nor from the probabilistic definition~\eqref{eq:stable:f}.

The key idea consists in appropriate deformations of the integration contour in~\eqref{eq:stable:wh}. Below only $\kappa^\ua(\tau; \xi)$ is considered; the argument for $\kappa^\da(\tau; \xi)$ is very similar.

First, observe that the integrand in~\eqref{eq:stable:wh} extends to a holomorphic function of $z$ in $\C_\la$ and in $\C_\ra$. Indeed, $\Psi(z) = a z^\alpha$ for $z > 0$ and $\Psi(z) = \bar{a} (-z)^\alpha$ for $z < 0$. Moreover, due to assumptions~\eqref{eq:stable:a}, one has $|\Arg a| \le (2 - \alpha) \tfrac{\pi}{2}$ and $|\Arg z^\alpha| < \tfrac{\alpha \pi}{2}$, so that $a z^\alpha$ extends to a holomorphic function in $\C_\ra$, and taking values in $\C \setminus (-\infty, 0]$. In a similar manner, $a (-z)^\alpha$ extends to a holomorphic function in $\C_\la$, also taking values in $\C \setminus (-\infty, 0]$.

Using the above extensions, Cauchy's theorem and an appropriate limiting procedure, one can prove that for $\thet \in (-\tfrac{\pi}{2}, \tfrac{\pi}{2})$, the integral in the first part of~\eqref{eq:stable:wh} is equal to
\formula{
 & \int_{-e^{-i \thet} \infty}^0 \expr{\frac{\log(\bar{a} (-z)^\alpha + \tau)}{i \xi - z} - \frac{\log(\bar{a} (-z)^\alpha)}{i - z}} \D z \\
 & \hspace*{5em} + \int_0^{e^{i \thet} \infty} \expr{\frac{\log(a z^\alpha + \tau)}{i \xi - z} - \frac{\log(a z^\alpha)}{i - z}} \D z .
}
By considering different values of $\thet$, this can be used to prove that $\kappa^\ua(\tau; \xi)$ is a complete Bernstein function of $\tau$ and $\xi$.

When $|\Arg a| < (2 - \alpha) \tfrac{\alpha \pi}{2}$, taking $\thet \searrow -\tfrac{\pi}{2}$ and writing $z = -i r$ gives
\formula{
 \kappa^\ua(\tau; \xi) & = c \exp \biggl( -\frac{1}{2 \pi i} \int_0^\infty \expr{\frac{\log(\bar{a} e^{i \alpha \pi / 2} r^\alpha + \tau)}{\xi + r} - \frac{\log(\bar{a} e^{i \alpha \pi / 2} r^\alpha)}{1 + r}} \D r \\
 & \hspace*{5em} + \frac{1}{2 \pi i} \int_0^\infty \expr{\frac{\log(a e^{-i \alpha \pi / 2} r^\alpha + \tau)}{\xi + r} - \frac{\log(a e^{-i \alpha \pi / 2} r^\alpha)}{1 + r}} \D r \biggr) .
}
Recall that $\im \log z = \Arg z$. Hence,
\formula{
 \kappa^\ua(\tau; \xi) & = c \exp\expr{-\frac{1}{\pi} \int_0^\infty \expr{\frac{\Arg(\bar{a} e^{i \alpha \pi / 2} r^\alpha + \tau)}{\xi + r} - \frac{\Arg(\bar{a} e^{i \alpha \pi / 2} r^\alpha)}{1 + r}} \D r} .
}
Since $\Arg(\bar{a} e^{i \alpha \pi / 2} r^\alpha + \tau) \in [0, \pi)$, the right-hand side of the above expression is virtually the exponential representation of a complete Bernstein function (see Theorem~\ref{th:cbf}\ref{it:cbf:c} below), and therefore $\kappa^\ua(\tau; \xi)$ is a complete Bernstein function of $\xi$ (we omit the details). With little effort, the above argument can be extended to the case $|\Arg a| = (2 - \alpha) \tfrac{\alpha \pi}{2}$.

To prove that $\kappa^\ua(\tau; \xi)$ is a complete Bernstein function of $\tau$, one takes $\thet = -\tfrac{1}{\alpha} \Arg a$, so that $a z^\alpha$ is real-valued for $z \in (0, e^{i \thet} \infty)$. Then $a (e^{i \thet} r)^\alpha = \bar{a} (-e^{-i \thet} r)^\alpha = |a| r^\alpha$ for $r > 0$, so that
\formula{
 \kappa^\ua(\tau; \xi) & = c \exp \biggl( -\frac{e^{-i \thet}}{2 \pi i} \int_0^\infty \expr{\frac{\log(|a| r^\alpha + \tau)}{i \xi + e^{-i \thet} r} - \frac{\log(|a| r^\alpha)}{i + e^{-i \thet} r}} \D r \\
 & \hspace*{5em} - \frac{e^{i \thet}}{2 \pi i} \int_0^\infty \expr{\frac{\log(|a| r^\alpha + \tau)}{i \xi - e^{i \thet} r} - \frac{\log(|a| r^\alpha)}{i - e^{i \thet} r}} \D r \biggr) \displaybreak[0] \\
 & = c \exp \biggl( \frac{1}{\pi} \int_0^\infty \biggl( \re\expr{\frac{1}{e^{i \thet} \xi - i r}} \log(|a| r^\alpha + \tau) \\
 & \hspace*{11em} - \re\expr{\frac{1}{e^{-i \thet} - i r}} \log(|a| r^\alpha) \biggr) \D r \biggr) .
}
Since $\re(1 / (e^{i \thet} \xi - i r)) = (\xi \cos \thet) / |e^{i \thet} \xi - i r|^2$ is an integrable function of $r \in (0, \infty)$, for some constant $\tilde{c} > 0$,
\formula{
 \kappa^\ua(\tau; \xi) & = \tilde{c} \exp\expr{\frac{1}{\pi} \int_0^\infty \frac{\xi \cos \thet}{\xi^2 - 2 \xi r \sin \thet + r^2} \, \log(|a| r^\alpha + \tau) \D r}
}
(the parenthesised expression in the right-hand side is often called \emph{Darling's integral}, see~\cite{bib:b73, bib:d56, bib:d87, bib:gj10, bib:h69, bib:ku10a, bib:ku13}). Furthermore, the weight $\tfrac{1}{\pi} (\xi \cos \thet) / (\xi^2 - 2 \xi r \sin \thet + r^2)$ is nonnegative and its integral over $r \in (0, \infty)$ is $1$. Therefore, $\kappa^\ua(\tau; \xi)$ is a weighted geometric integral average of $|a| r^\alpha + \tau$ with respect to $r \in (0, \infty)$. Since $|a| r^\alpha + \tau$ is a complete Bernstein function of $\tau$ for every $r > 0$, and since a geometric average of complete Bernstein functions is again a complete Bernstein function, one concludes that $\kappa^\ua(\tau; \xi)$ is a complete Bernstein function of $\tau$.

The proof that $\kappa^\ua(\tau; \xi)$ is a complete Bernstein function of $\xi$ provided above extends almost verbatim to any Rogers function $f$. It is essentially contained in~\cite{bib:r83}; a different argument for strictly stable L\'evy processes was given later in~\cite{bib:gj12}.

Unfortunately, similar direct extension of the proof that $\kappa^\ua(\tau; \xi)$ is a complete Bernstein function of $\tau$ fails for two reasons. First, the integral over the two rays $(-e^{-i \thet} \infty, 0)$ and $(0, e^{i \thet} \infty)$ needs to be replaced by an integral over an appropriate (typically unbounded) contour $\gamma_f$, over which $f$ takes real values. Handling this integral requires some regularity assumption on the Rogers function $f$ (namely, that $f$ is a balanced Rogers function). Furthermore, the integral in the last step of the argument is no longer a weighted geometric integral average: in general, the weight function fails to be nonnegative. Instead, integration by parts is used to transform the expression again to the exponential representation of a complete Bernstein function.

\subsection{Supremum and infimum functionals}

The supremum and infimum functionals
\formula{
 X^\ua_t & = \sup \tset{X_s : s \in [0, t]} , & X^\da_t & = \inf \tset{X_s : s \in [0, t]}
}
are described by the Wiener--Hopf factors:
\formula[]{
\label{eq:stable:pr}
 \int_0^\infty \ex \exp(-\xi X^\ua_t) \sigma e^{-\sigma t} \D t & = \frac{\kappa^\ua(\sigma; 0)}{\kappa^\ua(\sigma; \xi)} \, , & \int_0^\infty \ex \exp(\xi X^\da_t) \sigma e^{-\sigma t} \D t & = \frac{\kappa^\da(\sigma; 0)}{\kappa^\da(\sigma; \xi)};
}
see~\cite{bib:bd57,bib:r66,bib:f74}. The Mellin transforms of $X^\ua_t$ and $X^\da_t$ are well-known. For completeness, a simple derivation is included below. Only $X^\ua_t$ is considered, the calculations for $X^\da_t$ are very similar.

By Fubini,
\formula{
 \int_0^\infty \frac{\kappa^\ua(\sigma; 0)}{\kappa^\ua(\sigma; \xi)} \, \xi^{s-1} \D \xi & = \int_0^\infty \int_0^\infty \ex \exp(-\xi X^\ua_t) \sigma e^{-\sigma t} \xi^{s-1} \D t \D \xi \displaybreak[0] \\
 & = \Gamma(s) \int_0^\infty \ex (X^\ua_t)^{-s} \sigma e^{-\sigma t} \D t
}
for $s > 0$ and $\sigma > 0$. Since $X^\ua_t$ has the same distribution as $t^{1/\alpha} X_1^\ua$,
\formula{
 \int_0^\infty \frac{\kappa^\ua(\sigma; 0)}{\kappa^\ua(\sigma; \xi)} \, \xi^{s-1} \D \xi & = \Gamma(s) \ex (X_1^\ua)^{-s} \int_0^\infty t^{-s/\alpha} \sigma e^{-\sigma t} \D t = \Gamma(s) \Gamma(1 - \tfrac{s}{\alpha}) \sigma^{s/\alpha} \, \ex (X_1^\ua)^s .
}
This gives the formula for the Mellin transform of $X^\ua_t$ (see~\cite{bib:b73, bib:d56, bib:h69, bib:ku13}),
\formula[eq:stable:mellin]{
 \ex (X^\ua_t)^s & = \frac{(\tfrac{t}{\sigma})^{s/\alpha}}{\Gamma(s) \Gamma(1 - \tfrac{s}{\alpha})} \int_0^\infty \frac{\kappa^\ua(\sigma; 0)}{\kappa^\ua(\sigma; \xi)} \, \xi^{s-1} \D \xi
}
for $s, t, \sigma > 0$. Clearly, formula~\eqref{eq:stable:mellin} extends to $s \in \C$ such that $\ex (X^\ua_t)^{\re s}$ is finite. Noteworthy, the integral can be expressed using the double gamma function, see~\cite{bib:ku11, bib:ku13}. Hence, at least formally, inversion of the Mellin transform in~\eqref{eq:stable:mellin} can be written down, which provides an integral formula for the density of $X^\ua_t$.

The Wiener--Hopf factorisation of Rogers functions developed in this article provides an integral formula for the Laplace transform of $X^\ua_t$ for a much more general class of L\'evy processes than just the strictly stable processes (see Theorem~\ref{th:kappa:inv}). The following result is a version specialised to this context.

\begin{theorem}
\label{th:kappa:stable}
Let $X_t$ be a strictly stable L\'evy process with stability index $\alpha \in (0, 2]$, positivity parameter $\ro \in (0, 1)$ and scale parameter $k > 0$. Then for $t, \xi > 0$,
\formula[eq:kappa:stable:inv]{
 \ex \exp(-\xi X^\ua_t) & = \frac{\alpha}{\pi} \int_0^\infty \frac{u^{-(2 - \alpha) \ro} \sin(\ro \pi)}{1 + 2 u \cos(\ro \pi) + u^2} \exp\expr{I(u) + J(u) - \frac{k^\alpha t \xi^\alpha u^\alpha}{\cos((2 \ro - 1) \tfrac{\alpha \pi}{2})}} \D u ,
}
where
\formula[eq:kappa:stable:ints]{
 I(u) & = \frac{1}{\pi} \int_0^\infty \frac{\sin(\ro \pi)}{1 + 2 v \cos(\ro \pi) + v^2} \, \log \frac{(v - u) \sqrt{u^2 - 2 u v \cos(2 \ro \pi) + v^2}}{v^\alpha - u^\alpha} \, \D v , \\
 J(u) & = \frac{1}{\pi} \int_0^\infty \frac{1 + v \cos(\ro \pi)}{1 + 2 v \cos(\ro \pi) + v^2} \, \frac{1}{v} \Arg(u - v \cos(2 \ro \pi) + i v \sin(2 \ro \pi)) 
\D v .
}
A similar formula for $\exp(\xi X^\da_t)$ is obtained by replacing $\ro$ by $1 - \ro$.
\end{theorem}

\begin{remark}
When $\alpha \in (0, 1)$ and $s > 0$, then $\exp(-s \xi^\alpha)$ is the Laplace transform of the positive strictly stable distribution. This can be used to invert the Laplace transform in~\eqref{eq:kappa:stable:inv}. Similar method fails for $\alpha \in (1, 2)$, as then $\exp(-s \xi^2)$ fails to be the Laplace transform for $s > 0$.

When $\alpha = 2$, necessarily $\ro = \tfrac{1}{2}$. Therefore, $I(u) = J(u) = 0$ and~\eqref{eq:kappa:stable:inv} takes form
\formula{
 \ex \exp(-\xi X^\ua_t) & = \frac{2}{\pi} \int_0^\infty \frac{1}{1 + u^2} \, e^{-k^2 t \xi^2 u^2} \D u ,
}
which agrees with $\pr(X^\ua_t > x) = 2 \pr(X_t > x)$ for $x \ge 0$. Finally, when $\alpha = 1$, the statement of the theorem simplifies to
\formula[eq:kappa:1stable:inv]{
 \ex \exp(-\xi X^\ua_t) & = \frac{1}{\pi} \int_0^\infty \frac{u^{-\ro} \sin(\ro \pi)}{1 + 2 u \cos(\ro \pi) + u^2} \exp\expr{I(u) + J(u) - \frac{k t \xi u}{\cos((2 \ro - 1) \tfrac{\pi}{2})}} \D u ,
}
where
\formula[eq:kappa:1stable:ints]{
 I(u) & = \frac{1}{2 \pi} \int_0^\infty \frac{\sin(\ro \pi)}{1 + 2 v \cos(\ro \pi) + v^2} \, \log(u^2 - 2 u v \cos(2 \ro \pi) + v^2) \D v , \\
 J(u) & = \frac{1}{\pi} \int_0^\infty \frac{1 + v \cos(\ro \pi)}{1 + 2 v \cos(\ro \pi) + v^2} \, \frac{1}{v} \, \Arg(u - v \cos(2 \ro \pi) + i v \sin(2 \ro \pi)) \D v .
}
In particular, with $u = x \cos((2 \ro - 1) \tfrac{\pi}{2}) / (k t)$,
\formula[eq:kappa:1stable:inv2]{
 \pr(X^\ua_t \in \D x) & = \frac{1}{\pi} \frac{u^{-\ro} \sin(\ro \pi)}{1 + 2 u \cos(\ro \pi) + u^2} \, e^{I(u) + J(u)} \D x .
}
\end{remark}

%
%

\section{Preliminaries}
\label{sec:pre}

In this short section several classical results are recalled.

\begin{theorem}[Bernstein theorem]
\label{th:bernstein}
A function $f : (0, \infty) \to \R$ is completely monotone if and only if $f$ is the Laplace transform of a Radon measure $\mu$ on $[0, \infty)$, that is, for $x > 0$,
\formula[eq:bernstein]{
 f(x) & = \int_{[0, \infty)} e^{-s x} \mu(\D s) ,
}
and the integral is convergent for every $x > 0$. Furthermore, if $-1 < \alpha \le \beta$, then $\int_0^\infty x^\alpha \min(1, x^{\beta - \alpha}) f(x) \D x$ is finite if and only if $\int_0^\infty s^{-1 - \beta} \min(1, s^{\beta - \alpha}) \mu(\D s)$ is finite.
\end{theorem}

\begin{theorem}[Harmonic conjugate]
\label{th:holomorphic}
If $g : D \to \R$ is a harmonic function defined on a connected and simply connected open set $D \sub \C$, then there is a harmonic function $f : D \to \R$ such that $f + i g$ is holomorphic in $D$. The function $f$ is unique up to addition by a real constant.
\end{theorem}

\begin{theorem}[Herglotz theorem]
\label{th:harmonic:positive}
A function $f : \C_\ra \to [0, \infty)$ is harmonic if and only if there are $c \ge 0$ and a nonnegative Radon measure $m$ on $\R$ satisfying the integrability condition $\int_{\R} s^{-2} \min(1, s^2) m(\D s) < \infty$, such that
\formula[eq:harmonic:positive]{
 f(x + i y) & = c x + \frac{1}{\pi} \int_{\R} \frac{x}{x^2 + (y - s)^2} \, m(\D s)
}
for all $x > 0$ and $y \in \R$. Furthermore, in this case $m$ is the vague limit of measures $f(x + i s) \D s$ as $x \searrow 0$, and $c = \lim_{x \nearrow \infty} (f(x) / x)$. If $m$ is absolutely continuous, then its density function is equal to the limit of $f(x + i s)$ as $x \searrow 0$ for almost every $s \in \R$.
\end{theorem}

\begin{theorem}[Poisson representation and nontangential limits]
\label{th:harmonic:bounded}
A bounded function $f : \C_\ra \to \C$ is harmonic if and only if there is a bounded function $g$ on $\R$, such that
\formula[eq:harmonic:bounded]{
 f(x + i y) & = \frac{1}{\pi} \int_{-\infty}^\infty \frac{x}{x^2 + (y - s)^2} \, g(s) \D s
}
for all $x > 0$ and $y \in \R$. Furthermore, in this case for every $\alpha \in (0, \tfrac{\pi}{2})$,
\formula[eq:harmonic:limit]{
 g(s) & = \lim_{\substack{\xi \searrow 0 \\ |\Arg \xi| < \alpha}} f(s + \xi)
}
for almost all $s \in \R$.
\end{theorem}

The following proposition is a direct consequence of Taylor series expansion.

\begin{proposition}
\label{prop:harmonic:quot}
If $f$ is a real-analytic function in an open set $D \sub C$ and $f$ is zero on $D \cap \R$, then $f(\xi) / \im \xi$ extends to a real-analytic function on $D$. 
\end{proposition}

Throughout the article we commonly use the Cauchy's theorem and its extension. The basic form is stated below.

\begin{theorem}[residue theorem]
\label{th:residue}
If $f$ is a meromorphic function $f$ in a simply connected region $D \sub C$, $\gamma$ a simple contour contained in $D$, not containing any poles of $f$ and oriented counter-clockwise, then
\formula{
 \int_\gamma f(z) \D z & = \sum_w \Res(f; w) ,
}
where $\Res(f; w)$ is the residue of $f$ at $w$, and the sum spans over all poles $w$ of $f$ inside $\gamma$.
\end{theorem}

We often extend the above statement to unbounded contours, such as in the following statement.

\begin{corollary}
\label{cor:residue}
If $f$ is a holomorphic function $f$ in $\C_\ra$ and $|f(z)| \le c(\eps) (1 + |z|^2)^{-1} \log(|z| + |z|^{-1})$ when $|\Arg z| < \tfrac{\pi}{2} - \eps$, then
\formula{
 \int_0^{e^{i\thet} \infty} f(z) \D z & = \int_0^\infty f(z) \D z
}
for all $\thet \in (-\tfrac{\pi}{2}, \tfrac{\pi}{2})$.
\end{corollary}

This result, as well as similar ones used below, can be proved using Theorem~\ref{th:residue} and an appropriate limiting procedure.

%
%

\section{Complete Bernstein functions}
\label{sec:cbf}

The class of Rogers functions is closely related to the notion of complete Bernstein functions: on one hand, $f(\xi)$ is a complete Bernstein function if and only if $f(\xi^2)$ is a \emph{symmetric} Rogers function (see Section~\ref{sec:basic}); on the other hand, $f(\xi)$ is a Rogers function if and only if it admits a Wiener--Hopf type factorisation into the product $f(\xi) = f^\ua(-i \xi) f^\da(i \xi)$ of two complete Bernstein functions $f^\ua$, $f^\da$ (see Section~\ref{sec:wh}). In the present section equivalent definitions of complete Bernstein functions are briefly discussed. The primary reference for complete Bernstein functions and related notions is~\cite{bib:ssv10}.

\begin{theorem}[{see~\cite[Section~2.5]{bib:k11} and~\cite[Chapter~6]{bib:ssv10}}]
\label{th:cbf}
Let $f$ be a holomorphic function in $\C \setminus (-\infty, 0]$. The following conditions are equivalent:
\begin{enumerate}[label=\rm (\alph*)]
\item\label{it:cbf:a}
There are constants $c_0, c_1 \ge 0$, and a completely monotone function $\tilde{\nu}$ on $(0, \infty)$ satisfying the integrability condition $\int_0^\infty \min(1, z) \tilde{\nu}(z) \D z < \infty$, such that
\formula[eq:cbf:a]{
 f(\xi) = c_0 + c_1 \xi + \int_0^\infty (1 - e^{-\xi z}) \tilde{\nu}(z) \D z
}
for all $\xi \in \C_\ra$.
\item\label{it:cbf:b}
There are constants $c_0, c_1 \ge 0$ and a Radon measure $\tilde{\mu}$ on $(0, \infty)$ satisfying the integrability condition $\int_{(0, \infty)} s^{-2} \min(1, s) \tilde{\mu}(\D s) < \infty$, such that
\formula[eq:cbf:b]{
 f(\xi) & = c_0 + c_1 \xi + \frac{1}{\pi} \int_{(0, \infty)} \frac{\xi}{\xi + s} \, \frac{\tilde{\mu}(\D s)}{s}
}
for all $\xi \in \C \setminus (-\infty, 0]$.
\item\label{it:cbf:c}
There is a constant $\tilde{c} \ge 0$ and a function $\tilde{\ph} : (0, \infty) \to [0, \pi]$ such that
\formula[eq:cbf:c]{
 f(\xi) & = \tilde{c} \, \exp\expr{\frac{1}{\pi} \int_0^\infty \expr{\frac{\xi}{\xi + s} - \frac{1}{1 + s}} \frac{\tilde{\ph}(s) \D s}{s}}
}
for all $\xi \in \C \setminus (-\infty, 0]$.
\item\label{it:cbf:d}
$f(\xi) \in \overline{\C}_{\ua}$ for all $\xi \in \C_\ua$, and $f(\xi) \in [0, \infty)$ for $\xi \in (0, \infty)$.
\item\label{it:cbf:aa}
There are constants $c_0, c_1 \ge 0$, and a function $\nu$, such that $\nu$ is completely monotone on $(0, \infty)$, $\nu(-x) = \nu(x)$, $\int_{-\infty}^\infty \min(1, z^2) \nu(z) \D z < \infty$, and
\formula[eq:cbf:aa]{
 f(\xi^2) = c_0 + c_1 \xi^2 + \int_{-\infty}^\infty (1 - e^{i \xi z} + i \xi (1 - e^{-|z|}) \sign z) \nu(z) \D z
}
for all $\xi \in (0, \infty)$.
\item\label{it:cbf:bb}
There are constants $c_0, c_1 \ge 0$, and a symmetric Radon measure $\mu$ on $\R \setminus \{0\}$ satisfying the integrability condition $\int_{\R \setminus \{0\}} |s|^{-3} \min(1, s^2) \mu(\D s) < \infty$, such that
\formula[eq:cbf:bb]{
 f(\xi^2) & = c_0 + c_1 \xi^2 + \frac{1}{\pi} \int_{\R \setminus \{0\}} \expr{\frac{\xi}{\xi + i s} + \frac{i \xi \sign s}{1 + |s|}} \frac{\mu(\D s)}{|s|}
}
for all $\xi \in \C_\ra$.
\item\label{it:cbf:cc}
There is a constant $c \ge 0$ and a function $\ph : \R \to [0, \pi]$ such that $\ph(-s) = \ph(s)$ for all $s > 0$, and
\formula[eq:cbf:cc]{
 f(\xi^2) & = c \exp\expr{\frac{1}{\pi} \int_{-\infty}^\infty \expr{\frac{\xi}{\xi + i s} - \frac{1}{1 + |s|}} \frac{\ph(s)}{|s|} \, \D s}
}
for all $\xi \in \C_\ra$.
\item\label{it:cbf:dd}
For all $\xi \in \C_\ra$, $f(\xi^2) / \xi \in \overline{\C}_\ra$, and $f(\xi) \in [0, \infty)$ for $\xi \in (0, \infty)$.
\end{enumerate}
The constants $c_0$ and $c_1$ in representations~\eqref{eq:cbf:a}, \eqref{eq:cbf:b}, \eqref{eq:cbf:aa} and~\eqref{eq:cbf:bb} are equal, and they are uniquely determined by the function $f$,
\formula[eq:cbf:const]{
 c_0 & = \lim_{\xi \searrow 0} f(\xi), & c_1 & = \lim_{\xi \nearrow \infty} \frac{f(\xi)}{\xi} \, .
}
Furthermore,
\formula[]{
\label{eq:cbf:nu}
 \tilde{\nu}(z) & = \frac{1}{\pi} \int_{(0, \infty)} e^{-s z} \tilde{\mu}(\D s), & \nu(z) & = \frac{1}{2 \pi} \int_{(0, \infty)} \frac{\exp(-s \sqrt{|z|})}{\sqrt{s}} \, \tilde{\mu}(\D s)
}
for $z > 0$ and $z \in \R$, respectively, and if $0 < a < b$, then
\formula[eq:cbf:mu:alt]{
 \mu([\sqrt{a}, \sqrt{b}]) & = \int_{[a, b]} \frac{1}{\sqrt{4 s}} \, \tilde{\mu}(\D s) .
}
Moreover, also $\tilde{\mu}$ is uniquely determined by the function $f$,
\formula[eq:cbf:mu]{
 \pi c_0 \delta_0(\D s) + \frac{\tilde{\mu}(\D s)}{s} & = \lim_{t \searrow 0} \expr{-\im \frac{f(-s + i t)}{-s + i t} \, \D s} ,
}
with the vague limit of measures on the right-hand side. Finally, $\tilde{c} = f(1)$ in~\ref{it:cbf:c}, and if $f$ is not identically zero, $\tilde{\ph}(s)$ and $\ph(s) = \tilde{\ph}(s^2)$ are given uniquely by the formula
\formula[eq:cbf:ph]{
 \tilde{\ph}(s) & = \lim_{t \searrow 0} \Arg f(-s + i t) ,
}
for almost all $s > 0$.
\end{theorem}

Formula~\eqref{eq:cbf:b} in~\ref{it:cbf:b} is referred to as \emph{Stieltjes representation} of~$f$, while~\eqref{eq:cbf:c} in~\ref{it:cbf:c} is the \emph{exponential representation} of~$f$.

\begin{proof}
Equivalence of~\ref{it:cbf:a}, \ref{it:cbf:b}, \ref{it:cbf:c} and~\ref{it:cbf:d} is standard, see~\cite[Theorems~6.2 and~6.17]{bib:ssv10}. Formulae~\eqref{eq:cbf:const} and~\eqref{eq:cbf:mu}, as well as the expressions for $\tilde{\nu}$ and $\tilde{\ph}$, are parts of these results, see~\cite[Corollary~6.3]{bib:ssv10}. Conditions~\ref{it:cbf:aa}, \ref{it:cbf:bb}, \ref{it:cbf:cc} and~\ref{it:cbf:dd} are prepared for the definition of Rogers functions. Condition~\ref{it:cbf:aa} and the formula for $\nu$ were given explicitly in~\cite[Corollary~2.16 and Proposition~2.18]{bib:k11}, but similar results are already contained in~\cite{bib:r83}. Equivalence of~\ref{it:cbf:b} and~\ref{it:cbf:bb}, as well as of~\ref{it:cbf:c} and~\ref{it:cbf:cc}, is proved by symmetrizing the integrals in~\eqref{eq:cbf:bb} and~\eqref{eq:cbf:cc}. Finally, the Nevanlinna--Pick-type condition~\ref{it:cbf:dd} is equivalent to~\ref{it:cbf:bb} in greater generality, as will be proved in Theorem~\ref{th:rogers} below.
\end{proof}

\begin{definition}[{see~\cite[Definition~6.1]{bib:ssv10}}]
A function $f$ satisfying any of the equivalent conditions of Theorem~\ref{th:cbf} is a \emph{complete Bernstein function}. The class of complete Bernstein functions is denoted by $\cbf$.
\end{definition}

\begin{remark}
\label{rem:cbf}
In~\ref{it:cbf:a}, $f(\xi)$ is given as the Laplace exponent of a subordinator (a nondecreasing L\'evy process) with completely monotone density of the L\'evy measure, killed at the rate $c_0$. According to~\ref{it:cbf:aa}, $f(\xi^2)$ is the L\'evy--Khintchine exponent of a symmetric L\'evy process with completely monotone density of the L\'evy measure on $(0, \infty)$, killed at the rate $c_0$.

Conditions~\ref{it:cbf:b} and~\ref{it:cbf:bb} are related to the representation theorem for the nonnegative harmonic functions $\im f(\xi)$ and $\im f(\xi^2)$ (Theorem~\ref{th:harmonic:positive}), while~\ref{it:cbf:c} and~\ref{it:cbf:cc} correspond to the representation theorem for the bounded harmonic functions $\Arg f(\xi)$ and $\Arg f(\xi^2)$ (Theorem~\ref{th:harmonic:bounded}).
\end{remark}

The following two properties of complete Bernstein functions are well-known.

\begin{proposition}
\label{prop:cbf:arg}
For a nonzero function $f$ holomorphic in $\C \setminus (-\infty, 0]$, the following conditions are equivalent:
\begin{enumerate}[label=\rm (\alph*)]
\item $f \in \cbf$;
\item $0 \le \Arg f(\xi) \le \Arg \xi$ for $\xi \in \C_\ua$;
\item $-\Arg \xi \le \Arg f(\xi) \le 0$ for $\xi \in \C_\da$.
\end{enumerate}
\end{proposition}

\begin{proof}
By Theorem~\ref{th:cbf}\itref{it:cbf:d}, if $f \in \cbf$, then $\Arg f(\xi) \ge 0$ for $\xi \in \C_\ua$. Furthermore, $\xi / f(\xi)$ is a complete Bernstein function, and hence, again by Theorem~\ref{th:cbf}\itref{it:cbf:d}, $\Arg \xi - \Arg f(\xi) \ge 0$ for $\xi \in \C_\ua$. Conversely, if $0 \le \Arg f(\xi) \le \Arg \xi$ for $\xi \in \C_\ua$, then, by continuity, $f(\xi) \in [0, \infty)$ for all $\xi \in (0, \infty)$, and hence, by Theorem~\ref{th:cbf}\itref{it:cbf:d}, $f \in \cbf$. A similar argument works for $\C_\da$.
\end{proof}

\begin{proposition}[{\cite[Corollary~7.6]{bib:ssv10}}]
\label{prop:cbf:limit}
If $f_n$ is a sequence of complete Bernstein functions and $f(\xi) = \lim_{n \to \infty} f_n(\xi)$ exists for every $\xi \in (0, \infty)$, then either $f$ is infinite on $(0, \infty)$, or $f$ extends to a complete Bernstein function. In the latter case, the convergence extends to all $\xi \in \C \setminus (-\infty, 0]$, and
\formula{
 & \lim_{n \to \infty} (c_{0,n} \delta_0(\D s) + c_{1,n} \delta_\infty(\D s) + s^{-2} \min(1, s) \tilde{\mu}_n(\D s)) \\
 & \hspace*{5em} = c_0 \delta_0(\D s) + c_1 \delta_\infty(\D s) + s^{-2} \min(1, s) \tilde{\mu}(\D s)
}
in the sense of vague convergence of measures on $[0, \infty]$; here $c_0$, $c_1$ and $\tilde{\mu}$ are the parameters in the Stieltjes representation~\eqref{eq:cbf:b} of $f$ in Theorem~\ref{th:cbf}\ref{it:cbf:b}, and $c_{0,n}$, $c_{1,n}$ and $\tilde{\mu}_n$ correspond in a similar way to $f_n$.
\end{proposition}

The last result of this section is required for the study of \emph{nearly} balanced Rogers functions.

\begin{proposition}
\label{prop:cbf:quot}
\begin{enumerate}[label=\rm (\alph*)]
\item\label{it:cbf:quot:a}
If $r > 0$, $f \in \cbf$, $f$ extends to a holomorphic function in $\C \setminus [-r, 0]$ and $f(\xi) \ge 0$ for $\xi \in (-\infty, -r)$, then $(\xi + r) f(\xi)$ is a complete Bernstein function of $\xi$.
\item\label{it:cbf:quot:b}
If $r > 0$, $f \in \cbf$, $f$ extends to a holomorphic function in $\C \setminus [-r, 0]$ and $f(\xi) \le 0$ for $\xi \in (-\infty, -r)$, then $f(\xi) / (\xi + r)$ is a complete Bernstein function of $\xi$.
\end{enumerate}
\end{proposition}

\begin{proof}
By Theorem~\ref{th:cbf}\ref{it:cbf:b}, with the notation of~\eqref{eq:cbf:b}, in both parts of the proposition,
\formula{
 f(\xi) & = c_0 + c_1 \xi + \frac{1}{\pi} \int_{(0, r]} \frac{\xi}{\xi + s} \, \frac{\tilde{\mu}(\D s)}{s}
}
for all $\xi \in \C \setminus [-r, 0]$; indeed, $\tilde{\mu}((r, \infty)) = 0$ by~\eqref{eq:cbf:mu}. If $f(\xi) \ge 0$ for $\xi \in (-\infty, -r)$, then $0 \le c_1 = \lim_{\xi \searrow -\infty} (f(\xi) / \xi) \le 0$, that is, $c_1 = 0$. Hence,
\formula{
 (\xi + r) f(\xi) & = c_0 (\xi + r) + \frac{1}{\pi} \int_{(0, r]} \frac{\xi (\xi + r)}{\xi + s} \, \frac{\tilde{\mu}(\D s)}{s} \\
 & = c_0 r + \expr{c_0 + \frac{1}{\pi} \int_{(0, r]} \frac{\tilde{\mu}(\D s)}{s}} \xi + \frac{1}{\pi} \int_{(0, r]} \frac{\xi}{\xi + s} \, \frac{(r - s) \tilde{\mu}(\D s)}{s}
}
is a complete Bernstein function of $\xi$. If $f(\xi) \le 0$ for $\xi \in (-\infty, -r)$, then, by monotone convergence,
\formula{
 0 & \ge \lim_{\xi \nearrow -r} f(\xi) = c_0 - c_1 r + \frac{1}{\pi} \int_{(0, r]} \frac{r}{r - s} \, \frac{\tilde{\mu}(\D s)}{s} \, ;
}
in particular, $\tilde{\mu}(\{r\}) = 0$. Furthermore,
\formula{
 \frac{f(\xi)}{\xi + r} & = c_0 \, \frac{1}{\xi + r} + c_1 \, \frac{\xi}{\xi + r} + \frac{1}{\pi} \int_{(0, r]} \frac{\xi}{(\xi + s) (\xi + r)} \, \frac{\tilde{\mu}(\D s)}{s} \displaybreak[0] \\
 & = \frac{c_0}{r} + \expr{c_1 - \frac{c_0}{r} - \frac{1}{\pi} \int_{(0, r]} \frac{\tilde{\mu}(\D s)}{s (r - s)}} \frac{\xi}{\xi + r} + \frac{1}{\pi} \int_{(0, r]} \frac{\xi}{\xi + s} \, \frac{\tilde{\mu}(\D s)}{s (r - s)} \, .
}
Since the parenthesised expression is nonnegative, $f(\xi) / (\xi + r)$ is a complete Bernstein function of $\xi$.
\end{proof}

%
%

\section{Definition of Rogers functions}
\label{sec:rogers}

In this section the notion of Rogers functions is formally introduced. In Theorem~\ref{th:rogers} the equivalence of conditions~\ref{it:cbf:aa}, \ref{it:cbf:bb}, \ref{it:cbf:cc} and~\ref{it:cbf:dd} in Theorem~\ref{th:cbf} is generalised to the asymmetric case. Major part of the proof of this result is a simplified version of the argument given in the Rogers's paper~\cite{bib:r83}. At the end of this section, some examples of Rogers functions are provided.

\begin{theorem}
\label{th:rogers}
Let $f$ be a holomorphic function in the right complex half-plane $\C_\ra$. The following conditions are equivalent:
\begin{enumerate}[label=\rm (\alph*)]
\item\label{it:rogers:a}
There are constants $c_0 \ge 0$, $c_1 \in \R$ and $c_2 \ge 0$, and a function $\nu$ on $\R \setminus \{0\}$ satisfying the integrability condition $\int_{-\infty}^\infty \min(1, z^2) \nu_\pm(z) \D z < \infty$, such that $\nu(z)$ and $\nu(-z)$ are completely monotone on $(0, \infty)$, and
\formula[eq:rogers:a]{
 f(\xi) & = c_0 - i c_1 \xi + c_2 \xi^2 + \int_{-\infty}^\infty (1 - e^{i \xi z} + i \xi (1 - e^{-|z|}) \sign z) \nu(z) \D z
}
for all $\xi \in (0, \infty)$.
\item\label{it:rogers:b}
There are constants $c_0 \ge 0$, $c_1 \in \R$ and $c_2 \ge 0$, and a Radon measure $\mu$ on $\R \setminus \{0\}$ satisfying the integrability condition $\int_{\R \setminus \{0\}} |s|^{-3} \min(1, s^2) \mu(\D s) < \infty$, such that
\formula[eq:rogers:b]{
 f(\xi) & = c_0 - i c_1 \xi + c_2 \xi^2 + \frac{1}{\pi} \int_{\R \setminus \{0\}} \expr{\frac{\xi}{\xi + i s} + \frac{i \xi \sign s}{1 + |s|}} \frac{\mu(\D s)}{|s|}
}
for all $\xi \in \C_\ra$; equivalently, for a fixed $r > 0$,
\formula[eq:rogers:bb]{
 f(\xi) & = c_0 - i \tilde{c}_1 \xi + c_2 \xi^2 + \frac{1}{\pi} \int_{\R \setminus \{0\}} \expr{\frac{\xi}{\xi + i s} + \frac{i \xi s}{r^2 + s^2}} \frac{\mu(\D s)}{|s|}
}
for the same $c_0$, $c_2$ and $\mu$, and for some $\tilde{c}_1 \in \R$.
\item\label{it:rogers:c}
There is a constant $c \ge 0$ and a function $\ph : \R \to [0, \pi]$, such that
\formula[eq:rogers:c]{
 f(\xi) & = c \exp\expr{\frac{1}{\pi} \int_{-\infty}^\infty \expr{\frac{\xi}{\xi + i s} - \frac{1}{1 + |s|}} \frac{\ph(s)}{|s|} \, \D s}
}
for all $\xi \in \C_\ra$.
\item\label{it:rogers:d}
For all $\xi \in \C_\ra$, $f(\xi) / \xi \in \overline{\C}_\ra$.
\end{enumerate}
The constants $c_0$, $c_1$, $\tilde{c}_1$ and $c_2$ in representations~\eqref{eq:rogers:a}, \eqref{eq:rogers:b} and~\eqref{eq:rogers:bb} are equal, and they are uniquely determined by the function $f$,
\formula[eq:rogers:const]{
 c_0 & = \lim_{\xi \searrow 0} f(\xi), & c_1 & = \lim_{\xi \nearrow \infty} \frac{-\im f(\xi)}{\xi} \, , & c_2 & = \lim_{\xi \nearrow \infty} \frac{f(\xi)}{\xi^2} \, , & \tilde{c}_1 & = \im f(r) \, .
}
Furthermore, for $z > 0$,
\formula[eq:rogers:nu]{
 \nu(z) & = \frac{1}{\pi} \int_{(0, \infty)} e^{-s z} \mu(\D s) , & \nu(-z) = \frac{1}{\pi} \int_{(-\infty, 0)} e^{s z} \mu(\D s) ,
}
and also $\mu$ is uniquely determined by the function $f$,
\formula[eq:rogers:mu]{
 \pi c_0 \delta_0(\D s) + \mu(\D s) & = \lim_{t \searrow 0} \expr{\re \frac{f(t + i s)}{t + i s} \, \D s} ,
}
with the vague limit of measures in the right-hand side. If $f$ is not identically zero, then $\ph$ is uniquely determined by the formula
\formula[eq:rogers:ph]{
 \ph(s) & = -\sign s \, \lim_{t \searrow 0} \Arg f(t - i s)
}
for almost all $s \in \R$.
\end{theorem}

Formula~\eqref{eq:cbf:b} in~\ref{it:cbf:b} will be referred to as \emph{Stieltjes representation} of~$f$, while~\eqref{eq:cbf:c} in~\ref{it:cbf:c} is the \emph{exponential representation} of~$f$.

Equivalence of~\ref{it:rogers:a}, \ref{it:rogers:b} and~\ref{it:rogers:c} was given in~\cite{bib:r83}, see formulae~(14)--(18) therein. For completeness, a simple proof is provided. Condition~\ref{it:rogers:d} is related to the Nevanlinna--Pick functions (it is discussed also in~\cite{bib:jf12}).

\begin{proof}
By Bernstein's representation theorem (Theorem~\ref{th:bernstein}), the formula $\nu_\pm(z) = \int_{(0, \infty)} e^{-z s} \mu_\pm(\D s)$ defines a one to one correspondence between completely monotone functions $\nu_\pm$ on $(0, \infty)$ satisfying $\int_0^\infty \min(1, z^2) \nu_\pm(z) \D z < \infty$, and Radon measures $\mu_\pm$ on $(0, \infty)$ satisfying $\int_{(0, \infty)} |s|^{-3} \min(1, s^2) \mu_\pm(\D s) < \infty$. By taking $\mu(E) = \mu_+(E \cap (0, \infty)) + \mu_-(-E \cap (0, \infty))$, and $\nu(z) = \nu_+(z) \ind_{(0, \infty)}(z) + \nu_-(-z) \ind_{(-\infty, 0)}(z)$ and a short calculation, one obtains equivalence of~\ref{it:rogers:a} and the first part of~\ref{it:rogers:b} (we omit the details). The equivalence of both formulations of~\ref{it:rogers:b} follows by comparing the right-hand sides.

By a direct verification, \ref{it:rogers:b} implies~\ref{it:rogers:d}. Conversely, by the representation theorem for nonnegative harmonic functions (Theorem~\ref{th:harmonic:positive}), for any holomorphic function $f$ in $\C_\ra$ satisfying $\re (f(\xi) / \xi) \ge 0$, one has
\formula[eq:rogers:harmonic1]{
 \re \frac{f(\xi)}{\xi} & = c_2 \re \xi + \frac{1}{\pi} \int_{\R} \frac{\re \xi}{(\re \xi)^2 + (s + \im \xi)^2} \, m(\D s) \\
 & = c_2 \re \xi + \frac{1}{\pi} \int_{\R} \re \expr{\frac{1}{\xi + i s} + \frac{i \sign s}{1 + |s|}} m(\D s)
}
for some $c_2 \ge 0$ and a Radon measure $m$ on $\R$ satisfying the integrability condition $\int_{\R} s^{-2} \min(1, s^2) m(\D s) < \infty$ (note that $m(E)$ above is equal to $m(-E)$ in Theorem~\ref{th:harmonic:positive}). By Theorem~\ref{th:holomorphic}, it follows that for some $c_1 \in \R$,
\formula{
 \frac{f(\xi)}{\xi} & = -c_1 i + c_2 \xi + \frac{1}{\pi} \int_{\R} \expr{\frac{1}{\xi + i s} + \frac{i \sign s}{1 + |s|}} m(\D s) .
}
This is equivalent to~\eqref{eq:rogers:b} after taking $c_0 = \pi m(\{0\})$ and $\mu(\D s) = \ind_{\R \setminus \{0\}}(s) |s|^{-1} m(\D s)$. Hence, \ref{it:rogers:b} and~\ref{it:rogers:d} are equivalent.

If $f(\xi) = 0$ for some $\xi \in \C_\ra$ and~\ref{it:rogers:d} holds, then, by the maximum principle, $f$ is identically zero. If $f$ is nonzero and~\ref{it:rogers:d} holds, then $\Arg (f(\xi) / \xi) = \im (\log (f(\xi) / \xi))$ is a bounded harmonic function in $\C_\ra$, taking values in $[-\frac{\pi}{2}, \frac{\pi}{2}]$. By the representation theorem (Theorem~\ref{th:harmonic:bounded}),
\formula[eq:rogers:harmonic2]{
 \im \expr{\log \frac{f(\xi)}{\xi}} & = \frac{1}{\pi} \int_{-\infty}^\infty \frac{\re \xi}{(\re \xi)^2 + (s + \im \xi)^2} \, g(s) \D s
}
for some $g : \R \to [-\frac{\pi}{2}, \frac{\pi}{2}]$ (again, $g(s)$ is equal to $g(-s)$ in Theorem~\ref{th:harmonic:bounded}). Since $\im (\log \xi)$ is given by a similar integral with $g(s)$ replaced by $-\frac{\pi}{2} \sign s$, one obtains
\formula[eq:rogers:harmonic3]{
 \im(\log f(\xi)) & = -\frac{1}{\pi} \int_{-\infty}^\infty \frac{\re \xi}{(\re \xi)^2 + (s + \im \xi)^2} \expr{\frac{\pi}{2} \sign s - g(s)} \D s \\
 & = \frac{1}{\pi} \int_{-\infty}^\infty \im \expr{\frac{\xi}{\xi + i s} - \frac{1}{1 + |s|}} \expr{\frac{\pi}{2} - g(s) \sign s} \frac{\D s}{|s|} .
}
Representation~\eqref{eq:rogers:c} follows by Theorem~\ref{th:holomorphic}, with $\ph(s) = \frac{\pi}{2} - g(s) \sign s$. Conversely, if~\ref{it:rogers:c} holds, then, by the above argument, the boundary limits of the harmonic function $\Arg (f(\xi) / \xi)$ are in $[-\frac{\pi}{2}, \frac{\pi}{2}]$, which proves~\ref{it:rogers:d}.

Uniqueness properties and formulae for $c_2$ and $\mu$ in~\eqref{eq:rogers:const} and~\eqref{eq:rogers:mu} follow from the representation theorem for the harmonic function $\re (f(\xi) / \xi)$ in~\eqref{eq:rogers:harmonic1}. In a similar manner, formula~\eqref{eq:rogers:ph} is a consequence of the representation theorem for bounded harmonic function in~\eqref{eq:rogers:harmonic2}. Formulae for $c_0$ and $c_1$ are simple corollaries of~\eqref{eq:rogers:b} and dominated convergence.
\end{proof}

\begin{definition}
\label{def:rogers}
A function satisfying any of the equivalent conditions of Theorem~\ref{th:rogers} is said to be a \emph{Rogers function}. The class of Rogers functions is denoted by $\rogers$.
\end{definition}

\begin{remark}
\label{rem:rogers}
Condition~\ref{it:rogers:a} states that $f \in \rogers$ if and only if $f$ is the L\'evy--Khintchine exponent of a L\'evy process, whose L\'evy measure has a density function $\nu$ such that $\nu(z)$ and $\nu(-z)$ are completely monotone on $(0, \infty)$, killed at the rate $c_0$. Such a L\'evy process (with no killing) is said to have \emph{completely monotone jumps}.

More precisely, the L\'evy--Khintchine exponent $\Psi(\xi)$ of the L\'evy process described above, is equal to $f(\xi)$ for $\xi > 0$ and to $\overline{f(-\xi)}$ for $\xi < 0$. This L\'evy--Khintchine exponent extends to a holomorphic function on $\C \setminus i \R$, given by $f(\xi)$ when $\re \xi > 0$, and by $\overline{f(-\bar{\xi})}$ when $\re \xi < 0$. For this reason, the domain of definition of a Rogers function $f$ is automatically extended to $\C \setminus i \R$ by the formula $f(\xi) = \overline{f(-\bar{\xi})}$ for $\xi \in \C_\la$. Note that if $f \in \rogers$ corresponds to the L\'evy process $X_t$, then also $\overline{f(\bar{\xi})} = f(-\xi)$ is a Rogers function, corresponding to the dual process $\hat{X}_t = -X_t$.

As in Remark~\ref{rem:cbf}, conditions~\ref{it:rogers:b} and~\ref{it:rogers:c} are consequences of representation theorems for harmonic functions. Condition~\ref{it:rogers:d} turns out to be very convenient in applications.
\end{remark}

\begin{remark}
\label{rem:rogers:symmetric}
By comparing Theorems~\ref{th:cbf} and~\ref{th:rogers}, there is a one to one correspondence between \emph{symmetric} Rogers functions $f$ (satisfying $f(\bar{\xi}) = \overline{f(\xi)}$, or $f(\xi) = f(-\xi)$) and complete Bernstein functions $g$, given by the formula $f(\xi) = g(\xi^2)$. Therefore, L\'evy processes with completely monotone jumps can be thought of as asymmetric analogues of L\'evy processes with L\'evy--Khintchine exponent $g(\xi^2)$ for a complete Bernstein function $g$. The latter class have been studied in~\cite{bib:k11, bib:k12, bib:kmr12}, but also e.g.\ in~\cite{bib:bbkrsv09, bib:bgr14, bib:g13, bib:ksv12, bib:ksv13, bib:ksv14} in different contexts.

The following are equivalent conditions for a Rogers function $f$ to be symmetric: $f$ is real-valued on $(0, \infty)$; $c_1 = 0$ and $\mu$ is a symmetric measure; $c_1 = 0$ and $\nu(z) = \nu(-z)$ for $z > 0$; $c_1 = 0$ and $\ph(s) = \ph(-s)$ for $s > 0$.
\end{remark}

\begin{remark}
\label{rem:rogers:cbf}
Another link between Rogers functions and complete Bernstein functions can be stated as follows. If $f$ is a complete Bernstein function, then $f(-i \xi)$ and $f(i \xi)$ are Rogers functions of $\xi$ (by Theorem~\ref{th:cbf}\ref{it:cbf:b} and Theorem~\ref{th:rogers}\ref{it:rogers:b}).

Conversely, if $f$ is a Rogers function, $f$ extends to a holomorphic function in $\C \setminus (-i \infty, 0]$, and this extension satisfies $f(i \xi) \ge 0$ for $\xi > 0$, then $f(i \xi)$ defines a complete Bernstein function of $\xi$. In a similar manner, if $f \in \rogers$, $f$ extends to a holomorphic function in $\C \setminus [0, i \infty)$ and $f(-i \xi) \ge 0$ for $\xi \in (0, \infty)$, then $f(-i \xi)$ is a complete Bernstein function of $\xi$. This also follows easily from Theorem~\ref{th:rogers}\ref{it:rogers:b} and Theorem~\ref{th:cbf}\ref{it:cbf:b} (we omit the details).

Rogers functions of the form $f(-i \xi)$ or $f(i \xi)$ for $f \in \cbf$ can be thought of as completely asymmetric Rogers functions. The corresponding L\'evy processes are nondecreasing or nonincreasing, respectively. This classes should be distinguished from the larger families of Rogers functions corresponding to spectrally positive or spectrally negative L\'evy processes (i.e. processes without negative of positive jumps, respectively).
\end{remark}

\begin{example}
\label{ex:rogers}
The following functions of $\xi$ belong to $\rogers$:
\begin{enumerate}[label=\rm (\alph*)]
\item $f(\xi) = \frac{1}{2} \xi^2 - i b \xi$ for $\xi \in \C_\ra \cup \C_\la$, the L\'evy--Khintchine exponent of the Brownian motion with drift; here $b \in \R$;
\item $f(\xi) = a \xi^\alpha$ for $\xi \in \C_\ra$ and $f(\xi) = \bar{a} (-\xi)^\alpha$ for $\xi \in \C_\la$, the L\'evy--Khintchine exponent of a strictly stable L\'evy process, discussed in detail in Section~\ref{sec:ex}; here $\alpha \in (0, 2]$ and $a$ satisfies~\eqref{eq:stable:a};
\item $f(\xi) = -i b \xi$ for $\xi \in \C_\ra \cup \C_\la$, the L\'evy--Khintchine exponent of the (deterministic) process of uniform motion with velocity $b$; here $b \in \R$;
\item $f(\xi) = a \xi^\alpha - i b \xi$ for $\xi \in \C_\ra$ and $f(\xi) = \bar{a} (-\xi)^\alpha - i b \xi$ for $\xi \in \C_\la$, the L\'evy--Khintchine exponent of a stable (but not necessarily strictly stable) L\'evy process; here $\alpha \in (0, 2]$, $a$ satisfies~\eqref{eq:stable:a} and $b \in \R$;
\item\label{ex:rogers:e} $f(\xi) = \xi / (\xi - a i) - i b \xi$ for $\xi \in \C_\ra \cup \C_\la$, the L\'evy--Khintchine exponent of the classical risk process (the Cram\'er--Lundberg model) with exponentially distributed claims; here $a, b > 0$;
\item any linear combination with nonnegative coefficients of the examples given above.
\end{enumerate}
\end{example}

Noteworthy, the if the measure $\mu$ in Theorem~\ref{th:rogers}\ref{it:rogers:b} is purely atomic, with atoms forming a discrete subset of $\R$, then $f$ is meromorphic, and it is the L\'evy--Khintchine exponent of a meromorphic L\'evy process, studied in detail in~\cite{bib:kkp13}.

%
%

\section{Basic properties of Rogers functions}
\label{sec:basic}

Significance of complete Bernstein functions is partially due to a great number of properties of various type they have. In this section some of these properties are generalised to the asymmetric setting of Rogers functions. We begin with distinguishing nonzero, nondegenerate, bounded and unbounded Rogers functions. Next, various operations on the class of Rogers functions are studied (Propositions~\ref{prop:rogers:prop} and~\ref{prop:rogers:conv}). Finally, we provide estimates of Rogers functions and related notions (Propositions~\ref{prop:rogers:bound}, \ref{prop:rogers:prime} and~\ref{prop:rogers:imag}).

\subsection{Simple subclasses of Rogers functions}

Clearly, $\rogers$ is a convex cone of functions, and if $f \in \rogers$, then $f(a \xi)$ is a Rogers function for arbitrary $a > 0$. By Theorem~\ref{th:rogers}\ref{it:rogers:d}, any $f \in \rogers$ maps $\C_\ra$ into $\C \setminus (-\infty, 0)$.

\begin{proposition}
\label{prop:rogers:trivial}
If $f \in \rogers$ is not identically $0$, then $f(\xi) \ne 0$ for all $\xi \in \C_\ra$. If $f \in \rogers$ is not of the form $f(\xi) = -c_1 i \xi$ for some $c_1 \in \R$, then $f(\xi) / \xi \in \C_\ra$ for all $\xi \in \C_\ra$, that is, $f(\xi) / \xi$ maps $\C_\ra$ into $\C_\ra$.
\end{proposition}

\begin{proof}
Since $\re (f(\xi) / \xi)$ is nonnegative and harmonic in $\C_\ra$, it is either identically equal to $0$ or everywhere positive. In the former case, $f(\xi) = -c_1 i \xi$. The result follows.
\end{proof}

If $f(\xi) = -c_1 i \xi$, the Rogers function $f$ is said to be \emph{degenerate}; otherwise, $f$ is \emph{nondegenerate}. If $f \in \rogers$ is not constantly zero, $f$ is said to be \emph{nonzero}. If $f \in \rogers$ is a bounded function on $(0, \infty)$, then $f$ is said to be \emph{bounded}. Note that a bounded Rogers function may fail to be a bounded function on $\C_\ra$.

\begin{proposition}
\label{prop:rogers:bounded}
Every bounded Rogers function $f$ has the form
\formula[eq:rogers:bounded]{
 f(\xi) & = c_0 + \frac{1}{\pi} \int_{\R \setminus \{0\}} \frac{\xi}{\xi + i s} \, \frac{\mu(\D s)}{|s|} \, ,
}
where $c_0 \ge 0$ and $\mu(\D s) / |s|$ is a finite measure. In this case $f(\infty^-) = \lim_{\xi \nearrow \infty} f(\xi)$ exists and it is a real number. If $f \in \rogers$ is not bounded, then $\lim_{\xi \nearrow \infty} |f(\xi)| = \infty$. Furthermore, with the notation of Theorem~\ref{th:rogers}\ref{it:rogers:b}, one of the following conditions is satisfied: $c_2 > 0$, or $\mu(\D s) / |s|$ is an infinite measure, or $\mu(\D s) / |s|$ is a finite measure and $f(\xi) + i c \xi$ is given by the right-hand side of~\eqref{eq:rogers:bounded} for some $c \in \R \setminus \{0\}$.
\end{proposition}

\begin{proof}
The statement follows easily from Theorem~\ref{th:rogers}: if $f \in \rogers$ and $\re f(\xi)$ is a bounded function of $\xi \in (0, \infty)$, then, with the notation of Theorem~\ref{th:rogers}\ref{it:rogers:b}, the function
\formula{
 \re f(\xi) & = c_0 + c_2 \xi^2 + \frac{1}{\pi} \int_{-\R \setminus \{0\}} \frac{\xi^2}{\xi^2 + s^2} \, \frac{\mu(\D s)}{|s|}
}
remains bounded as $\xi \nearrow \infty$, so that $c_2 = 0$ and the measure $\mu(\D s) / |s|$ is finite. Therefore, for some $c \in \R$, $f(\xi) + i c \xi$ is given by the right-hand side of~\eqref{eq:rogers:bounded}. If $c = 0$, then $f$ is bounded; otherwise, $f$ is not bounded.
\end{proof}

\subsection{Operations on Rogers functions}

The class of Rogers functions is closed under various operations, described in the following results.

\begin{proposition}
\label{prop:rogers:prop}
If $f, g \in \rogers$, then the following functions of $\xi$ belong to $\rogers$:
\begin{enumerate}[label=\rm (\alph*)]
\item\label{it:rogers:prop:a} $\xi^2 f(1 / \xi)$;
\item\label{it:rogers:prop:b} $\xi^2 / f(\xi)$ and $1 / f(1 / \xi)$ if $f$ is nonzero;
\item\label{it:rogers:prop:c} $\xi^{1 - \alpha} f(\xi^\alpha)$ if $\alpha \in [-1, 1]$;
\item\label{it:rogers:prop:d} $g(\xi) f(\xi / g(\xi))$ if $g$ is nonzero;
\item\label{it:rogers:prop:e} $(f(\xi))^\alpha (g(\xi))^{1 - \alpha}$ if $\alpha \in [0, 1]$;
\item\label{it:rogers:prop:f} $((f(\xi))^\alpha + (g(\xi))^\alpha)^{1 / \alpha}$ if $\alpha \in [-1, 1]$.
\end{enumerate}
\end{proposition}

\begin{proof}
All statements follow directly from the condition given in Theorem~\ref{th:rogers}\ref{it:rogers:d}. For example, $\re(\xi^{-1} (\xi^{1 - \alpha} f(\xi^\alpha))) = \re(f(\xi^\alpha) / \xi^\alpha) \ge 0$ proves~\ref{it:rogers:prop:c}.
\end{proof}

Note that the statement~\ref{it:rogers:prop:d} contains as a special case the statement~\ref{it:rogers:prop:c}, which in turn generalises~\ref{it:rogers:prop:a}.

\begin{proposition}
\label{prop:rogers:conv}
If $f \in \rogers$ and $g \in \cbf$, then $g(f(\xi))$ is a Rogers function of $\xi$.
\end{proposition}

\begin{proof}
With no loss of generality one may assume that both $g$ and $f$ are nonzero. For $\xi \in \C_\ra$ with $\im f(\xi) \ge 0$ one has $0 \le \Arg g(f(\xi)) \le \Arg f(\xi)$ (by Proposition~\ref{prop:cbf:arg}), and therefore
\formula{
 -\tfrac{\pi}{2} & \le \Arg (g(f(\xi)) / \xi) \le \Arg(f(\xi) / \xi) \le \tfrac{\pi}{2} .
}
Hence, $\Arg(g(f(\xi)) / \xi) \in [-\tfrac{\pi}{2}, \tfrac{\pi}{2}]$. By a similar argument $\Arg(g(f(\xi)) / \xi) \in [-\tfrac{\pi}{2}, \tfrac{\pi}{2}]$ when $\xi \in \C_\ra$ and $\im f(\xi) \le 0$. The result follows by Theorem~\ref{th:rogers}\ref{it:rogers:d}.
\end{proof}

\begin{proposition}
\label{prop:rogers:bounded:prop}
If $f$ is a bounded Rogers function and $c \ge f(\infty^-)$, then $c - f(1 / \xi)$ is a (bounded) Rogers function of $\xi$.
\end{proposition}

\begin{proof}
By Proposition~\ref{prop:rogers:bounded} and dominated convergence, with the notation of~\eqref{eq:rogers:bounded},
\formula{
 c & \ge c_0 + \frac{1}{\pi} \int_{\R \setminus \{0\}} \frac{\mu(\D s)}{|s|} \, ,
}
and
\formula{
 c - f(1 / \xi) & = c - c_0 - \frac{1}{\pi} \int_{-\R \setminus \{0\}} \frac{1}{1 + i \xi s} \, \frac{\mu(\D s)}{|s|} \displaybreak[0] \\
 & = \expr{c - c_0 - \frac{1}{\pi} \int_{\R \setminus \{0\}} \frac{\mu(\D s)}{|s|}} + \frac{1}{\pi} \int_{\R \setminus \{0\}} \frac{i \xi s}{1 + i \xi s} \, \frac{\mu(\D s)}{|s|} \, .
}
Substituting $s = -1 / t$ gives
\formula{
 c - f(1 / \xi) & = \expr{c - c_0 - \frac{1}{\pi} \int_{\R \setminus \{0\}} \frac{\mu(\D s)}{|s|}} + \frac{1}{\pi} \int_{\R \setminus \{0\}} \frac{\xi}{\xi + i t} \, \frac{\tilde{\mu}(\D t)}{|t|} \, ,
}
where $\int_{[a, b]} |t|^{-1} \tilde{\mu}(\D t) = \int_{[-1/a, -1/b]} |s|^{-1} \mu(\D s)$ if $0 < a < b$ or $a < b < 0$. The result follows by Proposition~\ref{prop:rogers:bounded}.
\end{proof}

\subsection{Estimates of Rogers functions}

Application of dominated convergence for integrals involving Rogers functions often requires estimates contained in the results stated below. The first one is typically applied for $r = 1$ or $r = |\xi|$.

\begin{proposition}
\label{prop:rogers:bound}
If $f \in \rogers$, $r > 0$ and $\xi \in \C_\ra$, then
\formula[eq:rogers:bound]{
 \frac{1}{\sqrt{2}} \, \frac{|\xi|^2}{r^2 + |\xi|^2} \expr{\frac{\re \xi}{|\xi|}} |f(r)| \le |f(\xi)| & \le \sqrt{2} \, \frac{r^2 + |\xi|^2}{r^2} \expr{\frac{|\xi|}{\re \xi}} |f(r)| .
}
\end{proposition}

\begin{proof}
Since $f(r \xi)$ is a Rogers function of $\xi$, with no loss of generality one may assume that $r = 1$. Let $s \in \R$ and $\xi = x + i y$, where $x > 0$ and $y \in \R$. By a simple calculation,
\formula{
 (x^2 + (y + s)^2)(1 + x^2 + y^2) - (1 + s^2) x^2 = (x^2 + y (y + s))^2 + (y + s)^2 \ge 0 ,
}
so that
\formula[eq:rogers:aux:est]{
 |\xi + i s|^2 & = x^2 + (y + s)^2 \ge \frac{(1 + s^2) x^2}{1 + |\xi|^2} \, .
}
Furthermore,
\formula{
 \abs{\frac{\xi}{\xi + i s} + \frac{i \xi s}{1 + s^2}} & = \frac{|\xi| |1 + i \xi s|}{|\xi + i s| (1 + s^2)} \le \frac{|\xi| \sqrt{1 + |\xi|^2}}{|\xi + i s| \sqrt{1 + s^2}} \, .
}
It follows that
\formula{
 \abs{\frac{\xi}{\xi + i s} + \frac{i \xi \sign s}{1 + |s|}} & \le \frac{|\xi| (1 + |\xi|^2)}{x (1 + s^2)} \, .
}
By Theorem~\ref{th:rogers}\ref{it:rogers:b}, with the notation of~\eqref{eq:rogers:bb} (for $r = 1$),
\formula[eq:rogers:bound:aux]{
 |f(\xi)| & \le \frac{|\xi| (1 + |\xi|^2)}{x} \expr{c_0 + |\tilde{c}_1| + c_2 + \frac{1}{\pi} \int_{\R \setminus \{0\}} \frac{1}{1 + s^2} \, \frac{\mu(\D s)}{|s|}} .
}
On the other hand,
\formula{
 \frac{1}{1 + i s} + \frac{i s}{1 + s^2} & = \frac{1}{1 + s^2} ,
}
so that
\formula{
 \re f(1) + |\im f(1)| = c_0 + |\tilde{c}_1| + c_2 + \frac{1}{\pi} \int_{\R \setminus \{0\}} \frac{1}{1 + s^2} \, \frac{\mu(\D s)}{|s|} \, .
}
Since $\re f(1) + |\im f(1)| \le \sqrt{2} |f(1)|$, the upper bound for $|f(\xi)|$ follows. The lower bound is obtained from the upper bound for the Rogers function $\xi^2 / f(\xi)$ (if $f$ is nonzero).
\end{proof}

\begin{proposition}
\label{prop:rogers:prime:b}
If $f \in \rogers$ and $r > 0$, then
\formula[eq:rogers:prime:b]{
 f'(\xi) & = -i \tilde{c}_1 + 2 c_2 \xi + \frac{1}{\pi} \int_{\R \setminus \{0\}} \expr{\frac{i s}{(\xi + i s)^2} + \frac{i s}{r^2 + s^2}} \frac{\mu(\D s)}{|s|}
}
for all $\xi \in \C_\ra$, with $\tilde{c}_1$, $c_2$ and $\mu$ as in Theorem~\ref{th:rogers}\ref{it:rogers:b}.
\end{proposition}

\begin{proof}
Formula~\eqref{eq:rogers:prime:b} follows from Theorem~\ref{th:rogers}\ref{it:rogers:b} by dominated convergence. Indeed, let $\xi \in \C_\ra$. If $|\eps| \le \frac{1}{2} \re \xi$ and $s \in \R$, then
\formula{
 \abs{\frac{\xi + \eps}{\xi + \eps + i s} - \frac{\xi}{\xi + i s} + \frac{i \eps s}{1 + s^2}} & = \abs{\frac{i \eps s (1 + \xi^2 + 2 i \xi s + \eps s + i \eps s)}{(\xi + \eps + i s) (\xi + i s) (1 + s^2)}} \displaybreak[0] \\
 & \le c \, \frac{|s| (1 + |s|)}{(1 + |s|) (1 + |s|) (1 + s^2)} \le c |s|^{-2} \min(1, |s|^3)
}
for some $c$ (depending on $\xi$).
\end{proof}

\begin{proposition}
\label{prop:rogers:prime}
If $f \in \rogers$ and $\xi \in \C_\ra$, then
\formula[eq:rogers:prime]{
 \abs{\frac{\xi f'(\xi)}{f(\xi)}} & \le 4 (1 + \sqrt{2}) \expr{1 - \frac{|\im \xi|}{|\xi|}}^{-1} \expr{\frac{|\xi|}{\re \xi}} .
}
\end{proposition}

\begin{proof}
By Proposition~\ref{prop:rogers:prime:b}, with the notation of~\eqref{eq:rogers:prime:b} (for $r = |\xi|$, see~Theorem~\ref{th:rogers}\ref{it:rogers:b}),
\formula{
 \xi f'(\xi) & = -i \tilde{c}_1 \xi + 2 c_2 \xi^2 + \frac{1}{\pi} \int_{\R \setminus \{0\}} \expr{\frac{i \xi s}{(\xi + i s)^2} + \frac{i \xi s}{|\xi|^2 + s^2}} \frac{\mu(\D s)}{|s|} \displaybreak[0] \\
 & = -i \tilde{c}_1 \xi + 2 c_2 \xi^2 + \frac{1}{\pi} \int_{\R \setminus \{0\}} \frac{i \xi s (\xi^2 + 2 i \xi s + |\xi|^2)}{(1 + s^2) (\xi + i s)^2} \, \frac{\mu(\D s)}{|s|} \, .
}
Let $\xi = x + i y$ with $x > 0$, $y \in \R$. By a simple calculation,
\formula{
 (x^2 + (y + s)^2) |\xi| - (|\xi| - |y|) (x^2 + y^2 + s^2) = |y| (|\xi| + s \sign y)^2 \ge 0 ,
}
so that
\formula{
 |\xi + i s|^2 & = x^2 + (y + s)^2 \ge \expr{1 - \frac{|y|}{|\xi|}} (|\xi|^2 + s^2) .
}
Together with the inequality $|s| (|\xi| + |s|) \le \tfrac{1}{2} (1 + \sqrt{2}) \, (|\xi|^2 + s^2)$, this yields
\formula{
 |\xi f'(\xi)| & \le |\tilde{c}_1| |\xi| + 2 c_2 |\xi|^2 + \frac{1}{\pi} \expr{1 - \frac{|y|}{|\xi|}}^{-1} \int_{\R \setminus \{0\}} \frac{2 |\xi|^2 |s| (|\xi| + |s|)}{(|\xi|^2 + s^2)^2} \, \frac{\mu(\D s)}{|s|} \displaybreak[0] \\
 & \le (1 + \sqrt{2}) \expr{1 - \frac{|y|}{|\xi|}}^{-1} \expr{c_0 + |\tilde{c}_1| |\xi| + c_2 |\xi|^2 + \frac{1}{\pi} \int_{\R \setminus \{0\}} \frac{|\xi|^2}{|\xi|^2 + s^2} \, \frac{\mu(\D s)}{|s|}} .
}
The parenthesised expression is equal to $\re f(|\xi|) + |\im f(|\xi|)|$, which is not greater than $\sqrt{2} |f(|\xi|)|$. Therefore,
\formula{
 |\xi f'(\xi)| & \le (\sqrt{2} + 2) \expr{1 - \frac{|y|}{|\xi|}}^{-1} |f(|\xi|)| .
}
The result follows by Proposition~\ref{prop:rogers:bound} (with $r = |\xi|$).
\end{proof}

\begin{proposition}
\label{prop:rogers:imag}
If $f \in \rogers$, then
\formula{
 \int_0^\infty \frac{|\im f(\xi)|}{\xi (1 + \xi^2)} \, \D \xi < \infty .
}
\end{proposition}

\begin{proof}
By Theorem~\ref{th:rogers}\ref{it:rogers:b}, with the notation of~\eqref{eq:rogers:bb} (for $r = 1$), for $\xi > 0$,
\formula{
 \frac{|\im f(\xi)|}{\xi} & \le \tilde{c}_1 + \frac{1}{\pi} \int_{\R \setminus \{0\}} \abs{\frac{1}{\xi^2 + s^2} - \frac{1}{1 + s^2}} \mu(\D s) \displaybreak[0] \\
 & = \tilde{c}_1 + \frac{1}{\pi} \int_{\R \setminus \{0\}} \frac{|1 - \xi^2|}{(1 + s^2)(\xi^2 + s^2)} \, \mu(\D s) .
}
Since $|1 - \xi^2| / (1 + \xi^2) \le 1$ and $\int_0^\infty |s| / (\xi^2 + s^2) \D \xi = \tfrac{\pi}{2}$, by Fubini,
\formula{
 \int_0^\infty \frac{|\im f(\xi)|}{\xi (1 + \xi^2)} \, \D \xi & \le \frac{\tilde{c}_1 \pi}{2} + \frac{1}{2} \int_{\R \setminus \{0\}} \frac{1}{|s| (1 + s^2)} \, \mu(\D s) < \infty ,
}
as desired.
\end{proof}

%
%

\section{Real values of Rogers functions}
\label{sec:real}

Description of the set of points at which a Rogers function takes real values play a crucial role in our development. The main result in this direction is Theorem~\ref{th:rogers:real}, the proof of which relies upon the generalisation of a well-known property of complete Bernstein functions, given in Lemma~\ref{lem:rogers:quot}. A modification of the proof of Lemma~\ref{lem:rogers:quot} yields a new property of complete Bernstein functions (Lemma~\ref{lem:cbf:quot}). Some examples are provided at the end of the section.

\subsection{Difference quotients}

Recall that for any $\zeta \in (0, \infty)$ and $f \in \cbf$, the difference quotient $(\xi - \zeta) / (f(\xi) - f(\zeta))$ (extended continuously at $\xi = \zeta$) is a complete Bernstein function of $\xi$.

\begin{lemma}
\label{lem:rogers:quot}
If $f \in \rogers$ and $\zeta \in \C_\ra$, then
\formula[eq:rogers:quot:asym]{
 g(\xi) & = \frac{\xi^2}{(\xi - \zeta) (\xi + \bar{\zeta})} \expr{f(\xi) - \re f(\zeta) - \frac{i \xi + \im \zeta}{\re \zeta} \, \im f(\zeta)} ,
}
defined for $\xi \in \C_\ra \setminus \{\zeta\}$ and extended continuously at $\xi = \zeta$, is a Rogers function. In particular, if $f$ is a nonconstant Rogers function, $\zeta \in \C_\ra$ and $f(\zeta) \in (0, \infty)$, then
\formula[eq:rogers:quot]{
 f_{[\zeta]}(\xi) & = \frac{(\xi - \zeta) (\xi + \bar{\zeta})}{f(\xi) - f(\zeta)} \, , && \xi \in \C_\ra \setminus \{\zeta\} ,
}
extended continuously at $\xi = \zeta$ so that $f_{[\zeta]}(\zeta) = (2 \re \zeta) / f'(\zeta)$, defines a Rogers function.
\end{lemma}

\begin{proof}
By Theorem~\ref{th:rogers}\ref{it:rogers:b} and a direct calculation, with the notation of~\eqref{eq:rogers:b},
\formula[eq:rogers:b:alt]{
 f(\xi) & = c_0 + i c \xi + c_2 (\xi^2 - 2 i \xi \im \zeta) + \frac{1}{\pi} \int_{\R \setminus \{0\}} \expr{\frac{\xi}{\xi + i s} + i \xi \, \frac{s}{|\zeta + i s|^2}} \frac{\mu(\D s)}{|s|}
}
for some $c \in \R$. Observe that
\formula[eq:rogers:b:lambda]{
 f(\zeta) & = c_0 + i c \zeta + c_2 (\zeta^2 - 2 i \zeta \im \zeta) + \frac{1}{\pi} \int_{\R \setminus \{0\}} \frac{\zeta (\bar{\zeta} - i s) + i \zeta s}{|\zeta + i s|^2} \, \frac{\mu(\D s)}{|s|} \\
 & = c_0 + i c \zeta + c_2 |\zeta|^2 + \frac{1}{\pi} \int_{\R \setminus \{0\}} \frac{|\zeta|^2}{|\zeta + i s|^2} \frac{\mu(\D s)}{|s|} .
}
Hence, $\im f(\zeta) = c \re \zeta$. By~\eqref{eq:rogers:b:alt} and~\eqref{eq:rogers:b:lambda} and a short calculation,
\formula{
 f(\xi) - \re f(\zeta) - \frac{i \xi + \im \zeta}{\re \zeta} \, \im f(\zeta) & = c_2 (\xi^2 - 2 i \xi \im \zeta - |\zeta|^2) \\
 & \hspace*{-3em} + \frac{1}{\pi} \int_{\R \setminus \{0\}} \expr{\frac{\xi}{\xi + i s} + i \xi \, \frac{s}{|\zeta + i s|^2} - \frac{|\zeta|^2}{|\zeta + i s|^2}} \frac{\mu(\D s)}{|s|} \, .
}
By another simple calculation,
\formula{
 \frac{\xi}{\xi + i s} + i \xi \, \frac{s}{|\zeta + i s|^2} - \frac{|\zeta|^2}{|\zeta + i s|^2} & = \frac{\xi |\zeta + i s|^2 + (i \xi s - |\zeta|^2) (\xi + i s)}{(\xi + i s) |\zeta + i s|^2} \displaybreak[0] \displaybreak[0] \\
 & = \frac{i s (\xi^2 - 2 i \xi \re \zeta - |\zeta|^2)}{(\xi + i s) |\zeta + i s|^2} \, ,
}
and
\formula{
 (\xi^2 - 2 i \xi \im \zeta - |\zeta|^2) & = (\xi - \zeta) (\xi + \bar{\zeta}) .
}
It follows that
\formula[eq:rogers:quot:aux1]{
 \hspace*{3em} & \hspace*{-3em} f(\xi) - \re f(\zeta) - \frac{i \xi + \im \zeta}{\re \zeta} \, \im f(\zeta) \\
 & = (\xi - \zeta) (\xi + \bar{\zeta}) \expr{c_2 + \frac{1}{\pi} \int_{\R \setminus \{0\}} \frac{i s}{\xi + i s} \, \frac{\mu(\D s)}{|s| |\zeta + i s|^2}} .
}
Therefore,
\formula[eq:rogers:quot:aux2]{
 g(\xi) & = c_2 \xi^2 + \frac{1}{\pi} \int_{\R \setminus \{0\}} \frac{i s \xi^2}{\xi + i s} \, \frac{\mu(\D s)}{|s| |\zeta + i s|^2} \\
 & = c_2 \xi^2 + \frac{1}{\pi} \int_{\R \setminus \{0\}} \expr{\frac{\xi}{\xi + i s} + \frac{i \xi}{s}} \frac{|s| \mu(\D s)}{|\zeta + i s|^2} \\
 & = c_2 \xi^2 + i \xi \expr{\frac{1}{\pi} \int_{\R \setminus \{0\}} \frac{1}{s (1 + s^2)} \, \frac{|s| \mu(\D s)}{|\zeta + i s|^2}} \\
 & \qquad + \frac{1}{\pi} \int_{\R \setminus \{0\}} \expr{\frac{\xi}{\xi + i s} + \frac{i \xi s}{1 + s^2}} \frac{|s| \mu(\D s)}{|\zeta + i s|^2} .
}
By Theorem~\ref{th:rogers}\ref{it:rogers:b}, $g \in \rogers$. For the second statement of the lemma, simply note that $f_{[\zeta]}(\xi) = \xi^2 / g(\xi)$.
\end{proof}

\begin{remark}
\label{rem:rogers:quot}
If $f$ is a nonconstant Rogers function and $f(\zeta) \in (0, \infty)$ for some $\zeta \in \C_\ra$, then $f$ is nondegenerate, and from~\eqref{eq:rogers:quot:aux2} it follows that also $f_{[\zeta]}$ is nondegenerate. In particular, $f_{[\zeta]}(\xi)$ is well-defined for $\xi \in \C_\ra$. Hence $f(\zeta) \ne f(\xi)$ for all $\xi \in \C_\ra \setminus \{\zeta\}$. This means that a nonconstant $f \in \rogers$ takes \emph{real} values at unique points in $\C_\ra$ (a similar statement for nonreal values is not true in general). This result is strengthened below in Theorem~\ref{th:rogers:real}.

If $f \in \rogers$ is nonconstant and symmetric, then $f(\zeta) \in (0, \infty)$ if and only if $\zeta \in (0, \infty)$, and in this case $f_{[\zeta]}$ is symmetric for every $\zeta \in (0, \infty)$.
\end{remark}

The next result specialises Lemma~\ref{lem:rogers:quot} for Rogers functions of the form $f(-i \xi)$ or $f(i \xi)$ for a complete Bernstein function $f$ (see Remark~\ref{rem:rogers:cbf}). When $\zeta \in (0, i \infty)$ in the following lemma, the corresponding result was proved in~\cite{bib:k11}.

\begin{lemma}[{see~\cite[Lemma~2.20]{bib:k11}}]
\label{lem:cbf:quot}
If $f$ is a complete Bernstein function and $\zeta \in \C_\ua$, then for some complete Bernstein function $g$
\formula{
 \frac{f(\xi)}{(\xi - \zeta) (\xi - \bar{\zeta})} & = \frac{1}{2 i \im \zeta} \expr{\frac{f(\zeta)}{\xi - \zeta} - \frac{f(\bar{\zeta})}{\xi - \bar{\zeta}}} - \frac{g(\xi)}{\xi} \, .
}
Furthermore, the constants $c_0$ and $c_1$ in the Stieltjes representation~\eqref{eq:cbf:b} (Theorem~\ref{th:cbf}\ref{it:cbf:b}) of the complete Bernstein function $g$ are equal to $0$.
\end{lemma}

Lemma~\ref{lem:cbf:quot} is only needed in Section~\ref{subsec:ee}, where generalised eigenfunction expansion is discussed.

\begin{proof}
If $\zeta \in \C_\ua$, then $i \bar{\zeta} \in \C_\ra$. By Theorem~\ref{th:cbf}\ref{it:cbf:b}, $f(-i \xi)$ is a Rogers function and 
\formula{
 f(-i \xi) & = c_0 + i c \xi + \frac{1}{\pi} \int_{(0, \infty)} \expr{\frac{\xi}{\xi + i s} + i \xi \, \frac{s}{|\zeta + i s|^2}} \frac{\tilde{\mu}(\D s)}{s}
}
for some $c \in \R$ and $c_0$, $\tilde{\mu}$ as in Theorem~\ref{th:cbf}\ref{it:cbf:b}. By~\eqref{eq:rogers:quot:aux1} from the proof of Lemma~\ref{lem:rogers:quot} for $f(-i \xi)$ and $i \bar{\zeta}$, one obtains
\formula{
 f(\xi) - \re f(\bar{\zeta}) - \frac{-\xi + \re \zeta}{\im \zeta} \, \im f(\bar{\zeta}) & = (i \xi - i \bar{\zeta}) (i \xi - i \zeta) \, \frac{1}{\pi} \int_{(0, \infty)} \frac{i s}{i \xi + i s} \, \frac{\mu(\D s)}{|s| |i \bar{\zeta} + i s|^2}
}
for $\xi \in \C \setminus (-\infty, 0]$. After simplification, this gives
\formula{
 f(\xi) - \frac{(\xi - \bar{\zeta}) f(\zeta) - (\xi - \zeta) f(\bar{\zeta})}{2 i \im \zeta} & = -(\xi - \bar{\zeta}) (\xi - \zeta) \, \frac{1}{\pi} \int_{(0, \infty)} \frac{1}{\xi + s} \, \frac{\mu(\D s)}{|\zeta + s|^2} \, .
}
Hence,
\formula{
 \frac{\xi}{2 i \im \zeta} \expr{\frac{f(\zeta)}{\xi - \zeta} - \frac{f(\bar{\zeta})}{\xi - \bar{\zeta}}} - \frac{\xi f(\xi)}{(\xi - \zeta) (\xi - \bar{\zeta})} & = \frac{1}{\pi} \int_{(0, \infty)} \frac{\xi}{\xi + s} \, \frac{s}{|\zeta + s|^2} \, \frac{\mu(\D s)}{s} \, ,
}
and because $s / |\zeta + s|^2$ is bounded, the right-hand side defines a complete Bernstein function of $\xi$.
\end{proof}

\subsection{Line of real values}

\newcommand{\putimage}{%
\begingroup%
  \makeatletter%
  \ifx\svgwidth\undefined%
    \setlength{\unitlength}{117.53bp}%
    \ifx\svgscale\undefined%
      \relax%
    \else%
      \setlength{\unitlength}{\unitlength * \real{\svgscale}}%
    \fi%
  \else%
    \setlength{\unitlength}{\svgwidth}%
  \fi%
  \global\let\svgwidth\undefined%
  \global\let\svgscale\undefined%
  \makeatother%
  \begin{picture}(1,1.5105952)%
    \put(0,0){\includegraphics[width=\unitlength]{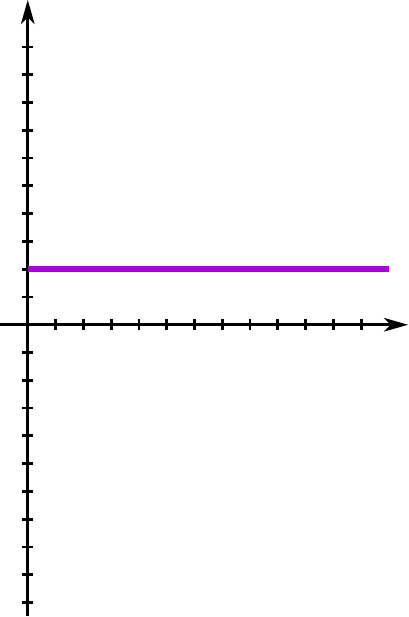}}%
    \put(0.20420318,0.68076452){\makebox(0,0)[cb]{\smash{\raisebox{-\height}{$1$}}}}%
    \put(0.34033864,0.68076452){\makebox(0,0)[cb]{\smash{\raisebox{-\height}{$2$}}}}%
    \put(0.47647409,0.68076452){\makebox(0,0)[cb]{\smash{\raisebox{-\height}{$3$}}}}%
    \put(0.61260955,0.68076452){\makebox(0,0)[cb]{\smash{\raisebox{-\height}{$4$}}}}%
    \put(0.74874500,0.68076452){\makebox(0,0)[cb]{\smash{\raisebox{-\height}{$5$}}}}%
    \put(0.88488046,0.68076452){\makebox(0,0)[cb]{\smash{\raisebox{-\height}{$6$}}}}%
    \put(0.03403386,0.85093384){\makebox(0,0)[rb]{\smash{\raisebox{-0.5\height}{$ 1$}}}}%
    \put(0.03403386,0.98706929){\makebox(0,0)[rb]{\smash{\raisebox{-0.5\height}{$ 2$}}}}%
    \put(0.03403386,1.12320475){\makebox(0,0)[rb]{\smash{\raisebox{-0.5\height}{$ 3$}}}}%
    \put(0.03403386,1.25934020){\makebox(0,0)[rb]{\smash{\raisebox{-0.5\height}{$ 4$}}}}%
    \put(0.03403386,1.39547566){\makebox(0,0)[rb]{\smash{\raisebox{-0.5\height}{$ 5$}}}}%
    \put(0.03403386,0.03412111){\makebox(0,0)[rb]{\smash{\raisebox{-0.5\height}{$-5$}}}}%
    \put(0.03403386,0.17025656){\makebox(0,0)[rb]{\smash{\raisebox{-0.5\height}{$-4$}}}}%
    \put(0.03403386,0.30639202){\makebox(0,0)[rb]{\smash{\raisebox{-0.5\height}{$-3$}}}}%
    \put(0.03403386,0.44252747){\makebox(0,0)[rb]{\smash{\raisebox{-0.5\height}{$-2$}}}}%
    \put(0.03403386,0.57866293){\makebox(0,0)[rb]{\smash{\raisebox{-0.5\height}{$-1$}}}}%
  \end{picture}%
  \global\let\svgname\undefined%
\endgroup%
}

\begin{figure}
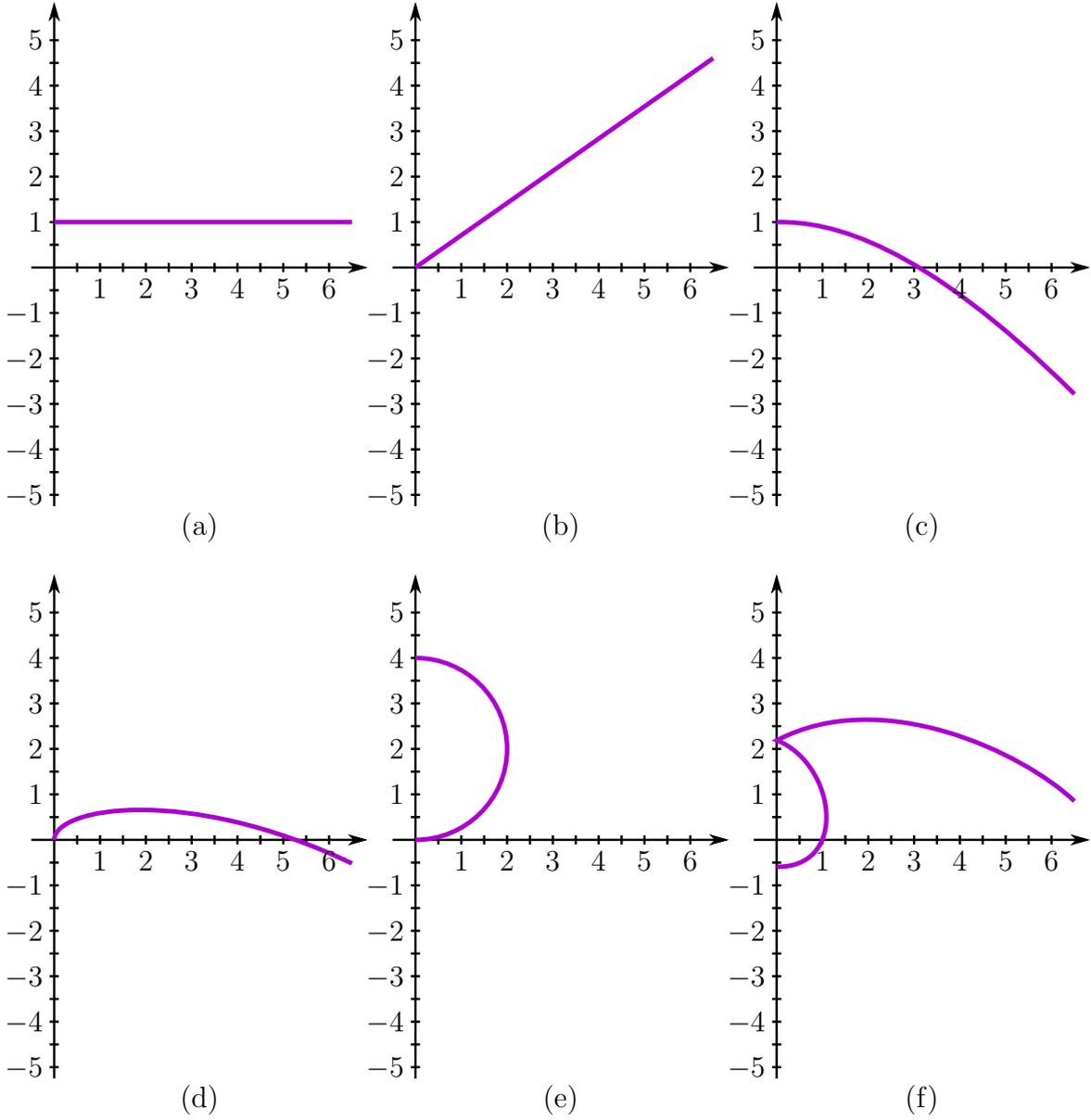

\centering
\begin{tabular}{ccc}
\def\svgwidth{0.3\textwidth}\def\svgname{bm}\putimage
&
\def\svgwidth{0.3\textwidth}\def\svgname{stable}\putimage
&
\def\svgwidth{0.3\textwidth}\def\svgname{tempered}\putimage
\\[-0.2em]
(a)
&
(b)
&
(c)
\\[1em]
\def\svgwidth{0.3\textwidth}\def\svgname{mixed}\putimage
&
\def\svgwidth{0.3\textwidth}\def\svgname{meromorphic-1}\putimage
&
\def\svgwidth{0.3\textwidth}\def\svgname{meromorphic-2}\putimage
\\[-0.2em]
(d)
&
(e)
&
(f)
\end{tabular}
\caption{Plot of $\gamma_f \cap \C_\ra$ for:
(a)~$f(\xi) = \tfrac{1}{2} \xi^2 - i \xi$;
(b)~$f(\xi) = 2 (-i \xi)^{1/2} + (i \xi)^{1/2}$;
(c)~$f(\xi) = (-i \xi + 1)^{1/2} + 3 (i \xi + 19)^{1/2}$;
(d)~$f(\xi) = 3 (-i \xi)^{1/3} + (i \xi)^{2/3}$;
(e)~$f(\xi) = \frac{\xi^2}{i \xi + 2}$;
(f)~$f(\xi) = \frac{-i \xi + 2}{-i \xi + 3} \, \frac{i \xi}{i \xi + a} \, \frac{i \xi + 15}{i \xi + 20}$ for $a \approx 0.10966$.
}
\label{fig:rogers:real}
\end{figure}

Recall that every $f \in \rogers$ is automatically extended to $\C_\la$ by the formula $f(\xi) = \overline{f(-\bar{\xi})}$ for $\xi \in \C_\la$. It can be easily checked that this extension is also given by~\eqref{eq:rogers:c} in Theorem~\ref{th:rogers}\ref{it:rogers:c}. If $s_0 \in \R$ and, with the notation of~\eqref{eq:rogers:c}, $\ph(s) = 0$ for almost all $s$ in some neighbourhood of $s_0$, then~\eqref{eq:rogers:c} defines the holomorphic extension of $f$ in the neighbourhood of $-i s_0$, and clearly $f(-i s_0) \in (0, \infty)$. This motivates the following definition.

\begin{definition}
\label{def:rogers:curve}
For a nonzero Rogers function $f$, let $\dom_f = \C \setminus (-i \esssupp \ph)$, where $\ph$ is the function in the exponential representation of $f$ (Theorem~\ref{th:rogers}\ref{it:rogers:c}) and $\esssupp$ denotes the essential support.

Suppose that $f$ is nonconstant. Since $f$ extends to a holomorphic function on $\dom_f$ and $f$ takes positive values in $\dom_f \cap i \R$, the function $h(\xi) = (\im f(\xi)) / (\re \xi)$ extends to a real-analytic function in $\dom_f$ (Proposition~\ref{prop:harmonic:quot}). Let $\gamma_f$ be the system of curves given by the equation $h(\xi) = 0$, and oriented so that positive values of $h$ lie on the left-hand side of $\gamma_f$ (cf.\ Theorem~\ref{th:rogers:real} below).
\end{definition}

\begin{remark}
The set $\dom_f$ may fail to be a maximal domain on which $f$ is holomorphic. For example, if $f(\xi) = \xi^2$, then $\dom_f = \C \setminus i \R$, despite the fact that $f$ extends to an entire function.
\end{remark}

Clearly, $\gamma_f$ is symmetric with respect to the imaginary axis. Furthermore, $\gamma_f \cap \C_\ra$ is the set of zeros (the \emph{nodal set}) of a harmonic function $\im f$ (or, equivalently, $\Arg f$), hence it is a system of piecewise analytic curves. In particular, $\gamma_f$ is piecewise analytic. In the next lemma the properties of $\gamma_f$ are studied in more detail.

\begin{theorem}
\label{th:rogers:real}
Let $f$ be a nonconstant Rogers function. 
\begin{enumerate}[label=\rm (\alph*)]
\item
\label{it:rogers:real:a}
The system of curves $\gamma_f$ intersects any centred circle in at most two points. For every $r > 0$, there is at most one $\zeta \in \gamma_f \cap \C_\ra$ such that $|\zeta| = r$, and in this case
\formula[eq:rogers:real]{
 \sign (\im f(\xi)) & = \sign \expr{\Arg \frac{\xi}{\zeta}}
}
for all $\xi \in \C_\ra$ with $|\xi| = r$. In other words, the system of curves $\gamma_f$ splits $\C_\ra$ into two sets $\{ \xi \in \C_\ra : \im f(\xi) > 0 \}$ and $\{ \xi \in \C_\ra : \im f(\xi) < 0 \}$, the first of which is located between $\gamma_f \cap \C_\ra$ and $(0, i \infty)$, and the other one between $(-i \infty, 0)$ and $\gamma_f \cap \C_\ra$ (see Figure~\ref{fig:rogers:real}).
\end{enumerate}
The point $\zeta$ defined above, if it exists, is denoted by $\zeta_f(r)$, and $|\gamma_f|$ is the set of those $r > 0$ for which the point $\zeta$ defined above exists.
\begin{enumerate}[label=\rm (\alph*),resume]
\item
\label{it:rogers:real:b}
The function $\zeta_f(r)$ can be continuously extended to $(0, \infty)$ in such a way that $\zeta_f(r) \in \{-i r, i r\}$ for every $r \in (0, \infty) \setminus |\gamma_f|$. Furthermore, $\zeta_f(r) \in \dom_f$ for every $r \in (0, \infty) \setminus \overline{|\gamma_f|}$.
\item
\label{it:rogers:real:c}
The system of curves $\gamma_f \cap \C_\ra$ consists of pairwise disjoint simple analytic curves. Furthermore, $\gamma_f \cap i \R$ is the closure of the set of endpoints of $\gamma_f \cap \C_\ra$ and it contains no interval.
\item
\label{it:rogers:real:d}
The function $\lambda_f(r) = f(\zeta_f(r))$, defined for $r \in (0, \infty) \setminus \partial |\gamma_f|$, is strictly increasing, and $\lambda_f'(r) > 0$ for each $r \in (0, \infty) \setminus \partial |\gamma_f|$. 
\end{enumerate}
\end{theorem}

Here $\overline{|\gamma_f|}$ and $\partial |\gamma_f|$ denote the closure and the boundary (in $(0, \infty)$) of the set $|\gamma_f|$.

\begin{proof}
By definition, $\gamma_f \cap \C_\ra$ is the nodal set of the bounded harmonic function $h(\xi) = \Arg f(\xi) = \im (\log f(\xi))$, or, equivalently, the function $g(\xi) = h(\xi) / \re \xi$. By Theorem~\ref{th:rogers}\ref{it:rogers:c} and the identity
\formula{
 \im \frac{x + i y}{x + i y + i s} = - \frac{x s}{x^2 + (s + y)^2} \, ,
}
for $x \in \R \setminus \{0\}$ and $y \in \R$ one has, with the notation of~\eqref{eq:rogers:c},
\formula[eq:rogers:real:aux1]{
 g(x + i y) & = \frac{1}{\pi} \int_{-\infty}^\infty \frac{1}{x^2 + (s + y)^2} \, (-\ph(s) \sign s) \, \D s ,
}
as in~\eqref{eq:rogers:harmonic3}. By dominated convergence and a short calculation (we omit the details),
\formula[eq:rogers:angular]{
 (x \tfrac{\partial}{\partial y} - y \tfrac{\partial}{\partial x}) g(x + i y) & = \frac{1}{\pi} \int_{-\infty}^\infty \frac{-2 x s}{(x^2 + (s + y)^2)^2} \, (-\ph(s) \sign s) \, \D s \\
 & = \frac{2}{\pi} \int_{-\infty}^\infty \frac{x |s|}{(x^2 + (s + y)^2)^2} \, \ph(s) \D s .
}
Recall that $f \in \rogers$ is nonconstant, so that $\ph$ is nonzero on a set of positive Lebesgue measure. Hence, the right-hand side is strictly positive for all $x > 0$ and $y \in \R$. It follows that for all $x > 0$, $y \in \R$,
\formula[eq:rogers:zeros]{
 (x \tfrac{\partial}{\partial y} - y \tfrac{\partial}{\partial x}) g(x + i y) & > 0 .
}
In particular, the gradient of $g$ is nonvanishing in $\C_\ra$. By the implicit function theorem, the nodal set $\{ \xi \in \C_\ra : h(\xi) = 0 \} = \gamma_f \cap \C_\ra$ consists of a family of mutually disjoint simple analytic curves. These curves necessarily start and end on the boundary of $\C_\ra$, or at complex infinity. The first statement of part~\ref{it:rogers:real:c} is proved.

The inequality~\eqref{eq:rogers:zeros} can be rephrased as follows: given $r > 0$, the function $g_r(\thet) = g(r e^{i \thet})$ has a strictly positive derivative in $\thet \in (-\frac{\pi}{2}, \frac{\pi}{2})$. Clearly, $g_r$ has at most one zero $\thet(r) \in (-\frac{\pi}{2}, \frac{\pi}{2})$, corresponding to a zero $\zeta_f(r) = r e^{i \thet(r)}$ of $g$. Furthermore, in this case $g_r(\thet) > 0$ for $\thet > \thet(r)$, and $g_r(\thet) < 0$ for $\thet < \thet(r)$. Formula~\eqref{eq:rogers:real} and part~\ref{it:rogers:real:a} follow.

The definition of $\zeta_f$ can be extended to $(0, \infty)$ so that $\zeta_f(r) = i r$ when $h_r$ is negative on $(-\frac{\pi}{2}, \frac{\pi}{2})$, and $\zeta_f(r) = -i r$ when $h_r$ is positive on $(-\frac{\pi}{2}, \frac{\pi}{2})$. By the definition, if $\zeta_f(r) = i r$ for $r$ in some interval $(r_1, r_2)$, then $\ph(-r) = 0$ for almost all $r \in (r_1, r_2)$, and therefore $(i r_1, i r_2) \sub \dom_f$. In a similar manner if $\zeta_f(r) = -i r$ for $r \in (r_1, r_2)$, then $(-i r_2, -i r_1) \sub \dom_f$. This proves part~\ref{it:rogers:real:b}.

Recall that $h$ and $g$ extend to $\dom_f$. Suppose that $r_0 > 0$ and $i r_0 \in \dom_f$. In this case $h(i r_0) = 0$, and hence $g(i r_0) = \tfrac{\partial}{\partial x} h(i r_0)$. Furthermore, if $i r_0$ is not in the closure of $\gamma_f \cap \C_\ra$, then either $\zeta_f(r) = i r$ in the neighbourhood of $r_0$, or $\zeta_f(r) \ne i r$ in the neighbourhood of $r_0$. In the former case $h(x + i y) < 0$ for $x + i y \in \C_\ra$ sufficiently close to $i r_0$, while in the latter $h(x + i y) > 0$ for $x + i y \in \C_\ra$ sufficiently close to $i r_0$. By Hopf's lemma, $g(i y) < 0$ in the former case, and $g(i y) > 0$ in the latter; in both cases, $g(i y) \ne 0$. In a similar way, $g(-i y) \ne 0$ if $r_0 > 0$ and $-i r_0 \in \dom_f$. Part~\ref{it:rogers:real:c} is proved.

It remains to establish the properties of $\lambda_f(r) = f(\zeta_f(r))$. As in the previous paragraph, if $r_0 \in (0, \infty) \setminus \overline{|\gamma_f|}$, then either $\zeta_f(r) = i r$ in a neighbourhood of $r_0$, or $\zeta_f(r) = -i r$ in a neighbourhood of $r_0$. In the former case, by Theorem~\ref{th:rogers}\ref{it:rogers:c} and dominated convergence (we omit the details),
\formula{
 \lambda_f'(r_0) & = -i f'(i r_0) = f(i r_0) \exp\expr{\frac{1}{\pi} \int_{-\infty}^\infty \frac{s}{(r_0 + s)^2} \, \frac{\ph(s) \D s}{|s|}} \, .
}
Hence, by~\eqref{eq:rogers:real:aux1}, $\lambda_f'(r_0) = \lambda_f(r_0) g(i r_0)$, which is positive by the former part of the proof. Similar argument works for the other case, that is, $\zeta_f(r) = -i r$ in a neighbourhood of $r_0$.

Let $r \in |\gamma_f|$ and $\zeta_f(r) = r e^{i \thet}$. Recall that $h(\xi) = g(\xi) \re \xi$ and $g(r e^{i \thet}) = 0$, so that $\tfrac{\partial}{\partial \thet} h(r e^{i \thet}) = (r \cos \thet) \tfrac{\partial}{\partial \thet} g(r e^{i \thet})$. Hence, by~\eqref{eq:rogers:zeros}, $\tfrac{\partial}{\partial \thet} h(r e^{i \thet}) > 0$. Since $h(\xi) = \im (\log f(\xi))$, one has
\formula{
 r \tfrac{\partial}{\partial r} (\re f)(r e^{i \thet}) & = \tfrac{\partial}{\partial \thet} (\im f)(r e^{i \thet}) = f(r e^{i \thet}) \tfrac{\partial}{\partial \thet} h(r e^{i \thet}) > 0 .
}
Differentiating both sides of $|\zeta_f(r)|^2 = r^2$ gives $\re (\overline{\zeta_f'(r)} \zeta_f(r)) = r > 0$. By the definition, $\lambda_f$ only takes real values, so that $0 = \im \lambda_f'(r) = \im (f'(\zeta_f(r)) \zeta_f'(r))$. Hence, $f'(\zeta_f(r)) = a \overline{\zeta_f'(r)}$ for some $a \in \R$, and
\formula{
 r \tfrac{\partial}{\partial r} (\re f)(r e^{i \thet}) & = \re(r e^{i \thet} f'(r e^{i \thet})) = \re(\zeta_f(r) a \overline{\zeta_f'(r)}) = a r .
}
It follows that $a > 0$, and therefore $\lambda_f'(r) = f'(\zeta_f(r)) \zeta_f'(r) = a |\zeta_f'(r)|^2 > 0$, as desired.

The above arguments prove that $\lambda_f(r)$ is an increasing function of $r$ on every interval in $(0, \infty) \setminus \partial |\gamma_f|$, but this is not sufficient if $|\gamma_f| \ne (0, \infty)$ and $|\gamma_f| \ne \varnothing$. Suppose that $r \in |\gamma_f|$ and $\zeta = \zeta_f(r)$; in particular, $f$ is nondegenerate. By Lemma~\ref{lem:rogers:quot}, $f_{[\zeta]}$ (defined in~\eqref{eq:rogers:quot}) is a Rogers function, and by Remark~\ref{rem:rogers:quot}, $f_{[\zeta]}$ is nondegenerate. By Proposition~\ref{prop:rogers:trivial}, $f_{[\zeta]}(\xi) / \xi \in \C_\ra$ for all $\xi \in \C_\ra$.

We claim that $\tilde{g}(\xi) = \xi^{-1} (\xi - \zeta) (\xi + \bar{\zeta}) \in \C_\ra$ when $\xi \in \C_\ra$ and $|\xi| > r$. Indeed, the function $\Arg \tilde{g}(\xi) = \Arg(\xi - \zeta) + \Arg(\xi + \bar{\zeta}) - \Arg \xi$ is a bounded harmonic function on the set $D_r = \{ \xi \in \C_\ra : |\xi| > r \}$, it is continuous on $\{ \xi \in \C_\ra : |\xi| \ge r \} \setminus \{\zeta\}$, and its boundary values are $\Arg \tilde{g}(\xi) = \frac{\pi}{2} \sign(\im(\xi - \zeta))$ for $\xi \in \partial D_r \setminus \{\zeta\}$ (we omit the details). The claim follows by the maximum principle.

Recall that $\zeta = \zeta_f(r)$ and take $\xi = \C_\ra$ with $|\xi| > r$. Then $f_{[\zeta]}(\xi) / \xi \in \C_\ra$, $\tilde{g}(\xi) \in \C_\ra$, and so $f(\xi) - f(\zeta) = \tilde{g}(\xi) / (f_{[\zeta]}(\xi) / \xi) \in \C \setminus (-\infty, 0]$. If $\xi \in \gamma_f$, then $f(\xi) - f(\zeta)$ is real, and hence positive. This proves that $\lambda_f$ is increasing on $|\gamma_f|$.

It follows that $\tilde{h}(\xi) = \Arg(f(\xi) - f(\zeta))$ is a bounded harmonic function in $D_r$, $\tilde{h}(\xi) > 0$ if $\im f(\xi) > 0$ and $\tilde{h}(\xi) < 0$ if $\im f(\xi) < 0$. Recall that $f(\xi) - f(\zeta) = \tilde{g}(\xi) / (f_{[\zeta]}(\xi) / \xi)$ for $\xi \in \C_\ra$, and that $\Arg \tilde{g}(\xi) = \frac{\pi}{2} \sign(\im(\xi - \zeta))$ and $\Arg (f_{[\zeta]}(\xi) / \xi) \in (-\tfrac{\pi}{2}, \tfrac{\pi}{2})$ for $\xi \in \partial D_r \setminus \{\zeta\}$. Therefore, for almost all $s > r$,
\formula[eq:rogers:real:aux2]{
 \lim_{t \searrow 0} \tilde{h}(i s + t) & \in [0, \pi] , & \lim_{t \searrow 0} \tilde{h}(-i s + t) & \in [-\pi, 0] .
}
Suppose that $r_0 \in (0, \infty) \setminus \overline{|\gamma_f|}$, $r_0 > r$ and that the above limits exist for $s = r_0$. Then either $\zeta_f(r_0) = i r_0$ or $\zeta_f(r_0) = -i r_0$. In the former case, $\im f(\xi) < 0$ and $\tilde{h}(\xi) < 0$ for $\xi \in \C_\ra$ sufficiently close to $i r_0$, and therefore $\lim_{t \searrow 0} \tilde{h}(i r_0 + t) \le 0$. Similarly, if $\zeta_f(r_0) = -i r_0$, then $\lim_{t \searrow 0} \tilde{h}(-i r_0 + t) \ge 0$. By~\eqref{eq:rogers:real:aux2},
\formula{
 \lim_{t \searrow 0} \tilde{h}(\zeta_f(r_0) + t) = 0 .
}
By continuity of $f$ at $\zeta_f(r_0)$, $f(\zeta_f(r_0)) - f(\zeta) \in [0, \infty)$ for almost all, and hence for all $r_0 \in (0, \infty) \setminus \overline{|\gamma_f|}$ such that $r_0 > r$.

A similar argument shows that $f(\zeta_f(r_0)) - f(\zeta) \in (-\infty, 0]$ for all $r_0 \in (0, \infty) \setminus \overline{|\gamma_f|}$ such that $r_0 < r$. (Alternatively, one can reuse the result for $r_0 > r$ by considering the Rogers function $1 / f(1 / \xi)$.)
\end{proof}

\begin{example}
\label{ex:rogers:real}
\begin{enumerate}[label=\rm (\alph*)]
\item For $f(\xi) = \frac{1}{2} \xi^2 - i b \xi$ (the L\'evy--Khintchine exponent of the Brownian motion with drift; $b \in \R$), one has $\dom_f = \C \setminus ((-i \infty, -2 i b] \cup [2 i b, i \infty))$, $\gamma_f = b i + \R$, $\zeta_f(r) = (r^2 - b^2)^{1/2} + b i$ and $\lambda_f(r) = \tfrac{1}{2} r^2$ for $r \ge |b|$, and $\zeta_f(r) = i r \sign b$ and $\lambda_f(r) = |b| r - \tfrac{1}{2} r^2$ for $r \in (0, b)$.
\item For $f(\xi) = a \xi^\alpha$ (the L\'evy--Khintchine exponent of a strictly stable L\'evy process; $\alpha \in (0, 2]$, $a$ satisfies~\eqref{eq:stable:a}), one has $\gamma_f = (-e^{-i \thet} \infty, 0) \cup (0, e^{i \thet} \infty)$, $\zeta_f(r) = r e^{i \thet}$ for $r > 0$ and $\lambda_f(r) = c r^\alpha$ for $r > 0$, where $\thet = -\tfrac{1}{\alpha} \Arg a$ and $c = |a|$ (see Section~\ref{sec:ex}).
\end{enumerate}
\end{example}

%
%

\section{Wiener--Hopf factorisation}
\label{sec:wh}

The proof that the Wiener--Hopf factors of a Rogers function are complete Bernstein functions was essentially given in~\cite[Theorem~2]{bib:r83}, where the condition $f \in \rogers$ is proved to be equivalent to the factorisation $f(\xi) + 1 = g^\ua(-i \xi) g^\da(i \xi)$ with $g^\ua, g^\da \in \cbf$. In this section a slightly modified statement is given in Theorem~\ref{th:rogers:wh}, and a simplified version of the proof from~\cite{bib:r83} is presented. Next we provide some examples and develop estimates for the Wiener--Hopf factors (Proposition~\ref{prop:rogers:wh:est}). After defining the notion of balanced Rogers functions, in Lemma~\ref{lem:rogers:curve} and in Corollaries~\ref{cor:rogers:wh} and~\ref{cor:rogers:wh:alt} alternative formulae for the Wiener--Hopf factors are given (Corollary~\ref{cor:rogers:wh} corresponds to the most common definition). The relation between the asymptotic behaviour of the ratio of two Rogers functions and the ratio of their Wiener--Hopf factors is discussed in Lemma~\ref{lem:rogers:wh:limit}. Finally, we introduce the notion of nearly balanced Rogers functions and study some of their properties.

\subsection{Wiener--Hopf factorisation theorem}

Thanks to Theorem~\ref{th:rogers}\ref{it:rogers:c}, the following fundamental result has a surprisingly simple proof.

\begin{theorem}[{see~\cite[Theorem~2]{bib:r83}}]
\label{th:rogers:wh}
A function $f$ is a nonzero Rogers function if and only if $f$ admits a decomposition $f(\xi) = f^\ua(-i \xi) f^\da(i \xi)$ for nonzero complete Bernstein functions $f^\ua$, $f^\da$. There is a unique \emph{normalised} choice of $f^\ua$ and $f^\da$ satisfying $f^\ua(1) = f^\da(1)$, and any other pair of factors is equal to $c f^\ua$ and $c^{-1} f^\da$ for some $c > 0$.
\end{theorem}

\begin{proof}
Suppose that $f \in \rogers$ and $f$ is not identically $0$. With the notation of~\eqref{eq:rogers:c} in Theorem~\ref{th:rogers}\ref{it:rogers:c}, define
\formula[eq:rogers:wh]{
 f^\ua(\xi) & = \sqrt{c} \, \exp \expr{\frac{1}{\pi} \int_0^\infty \expr{\frac{\xi}{\xi + s} - \frac{1}{1 + s}} \, \ph(s) \D s} \\
 f^\da(\xi) & = \sqrt{c} \, \exp \expr{\frac{1}{\pi} \int_0^\infty \expr{\frac{\xi}{\xi + s} - \frac{1}{1 + s}} \, \ph(-s) \D s} .
}
By Theorem~\ref{th:cbf}\ref{it:cbf:c}, $f^\ua, f^\da \in \cbf$, and $f^\ua(1) = f^\da(1) = \sqrt{c}$. Theorem~\ref{th:rogers}\ref{it:rogers:c} gives
\formula{
 f^\ua(-i \xi) f^\da(i \xi) & = c \, \exp \expr{\frac{1}{\pi} \int_{-\infty}^\infty \expr{\frac{\xi}{\xi + i s} - \frac{1}{1 + |s|}} \, \ph(s) \D s} = f(\xi) ,
}
as desired. To prove the converse, one reverses the above argument. Uniqueness follows by uniqueness of the representations in Theorem~\ref{th:cbf}\ref{it:cbf:c} and Theorem~\ref{th:rogers}\ref{it:rogers:c}.
\end{proof}

\begin{remark}
\label{rem:rogers:wh}
For symmetric Rogers functions $f$, the Wiener--Hopf factors $f^\ua$ and $f^\da$ are equal, and therefore their normalisation is quite natural. In the general case, the condition $f^\ua(1) = f^\da(1)$ is somewhat artificial, and in many cases it may not be the most convenient one; see, for example, the extended Wiener--Hopf factorisation in Section~\ref{sec:norm}. Nevertheless, from now on, the symbols $f^\ua$, $f^\da$ always correspond to the normalised Wiener--Hopf factors.

Note, however, that the expressions $f^\ua(\xi_1) / f^\ua(\xi_2)$, $f^\da(\xi_1) / f^\da(\xi_2)$ and $f^\ua(\xi_1) f^\da(\xi_2)$ do not depend on the choice of the normalisation.
\end{remark}

\begin{remark}
\label{rem:rogers:wh:ext}
By the formula $f^\ua(\xi) = f(i \xi) / f^\da(-\xi)$, $f^\ua$ extends to a holomorphic function in $(\C \setminus (-\infty, 0]) \cup (-i \dom_f)$, where $\dom_f$ is the domain of $f$ (see Definition~\ref{def:rogers:curve}). In a similar manner, $f^\da$ extends to a holomorphic function in $(\C \setminus (-\infty, 0]) \cup (i \dom_f)$. These extensions are again given by~\eqref{eq:rogers:wh}.
\end{remark}

\begin{example}
\label{ex:rogers:wh}
\begin{enumerate}[label=\rm (\alph*)]
\item
\label{it:rogers:wh:bm}
The Wiener--Hopf factors for $f(\xi) = \frac{1}{2} \xi^2 - i b \xi$ are given by
\formula{
 f^\ua(\xi) & = \sqrt{\frac{1 + 2 b}{2}} \, \xi, & f^\da(\xi) & = \frac{\xi + b}{\sqrt{2 (1 + 2 b)}}
}
for $b \ge 0$, and by
\formula{
 f^\ua(\xi) & = \frac{\xi - b}{\sqrt{2 (1 - 2 b)}} \, , & f^\da(\xi) & = \sqrt{\frac{1 - 2 b}{2}} \, \xi
}
for $b \le 0$.
\item
\label{it:rogers:wh:st}
When $f(\xi) = a \xi^\alpha$ ($\alpha \in (0, 2]$, $a$ satisfying~\eqref{eq:stable:a}), the Wiener--Hopf factors are power functions, $f^\ua(\xi) = |a|^{1/2} \xi^{\beta_\ua}$, $f^\da(\xi) = |a|^{1/2} \xi^{\beta_\da}$, where
\formula{
 \beta_\ua & = \tfrac{\alpha}{2} - \tfrac{1}{\pi} \Arg a , & \beta_\da & = \tfrac{\alpha}{2} + \tfrac{1}{\pi} \, \Arg a .
}
Hence, $\beta_\ua = \ro \alpha$ and $\beta_\da = (1 - \ro) \alpha$; see~\eqref{eq:stable:scaling} and Section~\ref{sec:ex}.
\end{enumerate}
\end{example}

\subsection{Basic properties of Wiener--Hopf factors}

Wiener--Hopf factorisation is consistent with multiplicative operations on Rogers functions.

\begin{proposition}
\label{prop:rogers:wh:prop}
If $f$ is a nonzero Rogers function, then the Wiener--Hopf factors of
\formula{
 \frac{\xi^2}{f(\xi)} \, , && \frac{1}{f(1 / \xi)} \, , && \xi^2 f(1 / \xi)
}
are, respectively,
\formula{
 \frac{\xi}{f^\ua(\xi)} \text{ and } \frac{\xi}{f^\da(\xi)} \, , && \frac{1}{f^\da(1 / \xi)} \text{ and } \frac{1}{f^\ua(1 / \xi)} \, , && \xi f^\da(1 / \xi) \text{ and } \xi f^\ua(1 / \xi) .
}
\end{proposition}

\begin{proof}
To prove the third statement, observe that $(-i \xi) f^\da(1 / (-i \xi)) (i \xi) (f^\ua(1 / (i \xi)) = \xi^2 (f^\da(i / \xi) f^\ua(-i / \xi)) = \xi^2 f(1 / \xi)$. The other ones are proved in a similar manner.
\end{proof}

Estimates of Wiener--Hopf factors play an important role.

\begin{proposition}
\label{prop:rogers:wh:est}
If $f$ is a nonzero Rogers function and $\xi > 0$, then
\formula[eq:rogers:wh:est]{
 \sqrt{\frac{|f(1)|}{2}} \, \frac{\xi}{1 + \xi} \le f^\ua(\xi) & \le \sqrt{2 |f(1)|} \, (1 + \xi) , \\
 \sqrt{\frac{|f(1)|}{2}} \, \frac{\xi}{1 + \xi} \le f^\da(\xi) & \le \sqrt{2 |f(1)|} \, (1 + \xi) .
}
\end{proposition}

\begin{proof}
Recall that $|f^\ua(-i)| |f^\da(i)| = |f(1)|$ and $f^\ua(1) = f^\da(1)$. From the Stieltjes representation (Theorem~\ref{th:cbf}\ref{it:cbf:b}) and the inequality
\formula{
 \frac{1}{\sqrt{2}} \, \abs{\frac{-i}{-i + s}} & \le \frac{1}{1 + s} \le \re \expr{(1 + i) \, \frac{-i}{-i + s}} ,
}
it follows that
\formula{
 \frac{|f^\ua(-i)|}{\sqrt{2}} & \le f^\ua(1) \le \re ((1 + i) f^\ua(-i)) \le \sqrt{2} |f^\ua(-i)| .
}
This and a similar estimate for $f^\da$ yield
\formula{
 (f^\ua(1))^2 & = f^\ua(1) f^\da(1) \le 2 |f^\ua(-i)| |f^\da(i)| = 2 |f(1)|^2 ,
}
and a similar lower bound with constant $\frac{1}{2}$. Since $f^\ua$ is a complete Bernstein function, one has (by concavity and positivity)
\formula{
 \min(1, \xi) & \le \frac{f^\ua(\xi)}{f^\ua(1)} \le \max(1, \xi)
}
for $\xi \in (0, \infty)$. This completes the proof of the estimate for $f^\ua$. Estimate for $f^\da$ is proved in a similar way.
\end{proof}

\subsection{Balanced Rogers functions}

Main results of this paper require regularity condition stated in the following definition.

\begin{definition}
\label{def:rogers:reg}
The system of curves $\gamma$ is said to be \emph{balanced} if:
\begin{enumerate}[label=\rm (\alph*)]
\item
\label{it:rogers:reg:a}
$\gamma \cap \C_\ra$ is a single curve, $\gamma \cap i \R$ is either empty or equal to $\{0\}$, and $\gamma$ is symmetric with respect to the imaginary axis;
\item
\label{it:rogers:reg:b}
for all $r > 0$, the intersection of $\gamma$ with $\{ z \in \C : |z| = r \}$ is equal to $\{\zeta(r), -\overline{\zeta(r)}\}$ for some $\zeta(r) \in \C_\ra$;
\item
\label{it:rogers:reg:c}
the orientation of $\gamma$ is chosen so that $\zeta(r)$ ($r \in (0, \infty)$) is a parametrisation of $\gamma \cap \C_\ra$, and $-\overline{\zeta(-r)}$ ($r \in (-\infty, 0)$) is a parametrisation of $\gamma \cap \C_\la$;
\item
\label{it:rogers:reg:d}
$\sup \set{|\Arg \zeta(r))| : r \in (0, \infty)} < \frac{\pi}{2}$;
\item
\label{it:rogers:reg:e}
$|\zeta'(r)| (1 + r^2)^{-1} \log(r + 1 / r)$ is integrable over $r \in (0, \infty)$.
\end{enumerate}
The function $\zeta(r)$ is said to be the \emph{canonical parametrisation} of $\gamma \cap \C_\ra$. A point $\xi \in \C \setminus \gamma$ is said to \emph{lie above $\gamma$}, denoted $\xi \in \gamma^\ua$, if
\formula{
 \Arg_\ua \xi \in (\Arg \zeta(|\xi|), \pi - \Arg \zeta(|\xi|)) ,
}
where $\Arg_\ua$ is the continuous version of the argument function on $\C \setminus (-i \infty, 0]$. In a similar manner, $\xi \in \C \setminus \gamma$ \emph{lies below $\gamma$}, denoted $\xi \in \gamma^\da$, if
\formula{
 \Arg_\da \xi \in (-\pi - \Arg \zeta(|\xi|), \Arg \zeta(|\xi|)) ,
}
where $\Arg_\da$ is the continuous version of the argument function on $\C \setminus [0, i \infty)$. A Rogers function $f$ is said to be \emph{balanced} if it is nonconstant, and $\gamma_f$ is balanced.
\end{definition}

Note that the definition of $\zeta(r)$ for $\gamma_f$ coincides with $\zeta_f(r)$ defined in Theorem~\ref{th:rogers:real}.

It turns out that $f \in \rogers$ is a balanced Rogers function if and only if the line of real values $\gamma_f$ is contained in the sector $\{ \xi \in \C_\ra : |\Arg \xi| < \tfrac{\pi}{2} - \eps\}$ for some $\eps > 0$.

\begin{lemma}
\label{lem:rogers:reg}
If $f$ is a nonconstant Rogers function and $\gamma_f$ satisfies condition~\ref{it:rogers:reg:d} in Definition~\ref{def:rogers:reg}, then it automatically satisfies condition~\ref{it:rogers:reg:e}, and therefore $f$ is a balanced Rogers function. In fact, in this case $|\zeta_f'(r)|$ is bounded by a constant, which depends only on the supremum in condition~\ref{it:rogers:reg:d} in Definition~\ref{def:rogers:reg}.
\end{lemma}

\begin{proof}
Let $\zeta_f(r) = r e^{i \thet(r)}$. Then $\zeta_f'(r) = (1 + i r \thet'(r)) e^{i \thet(r)}$, so that it suffices to prove that $r \thet'(r)$ is bounded in $r \in (0, \infty)$.

Define $g(\xi) = (\Arg f(\xi)) / (\re \xi)$, as in the proof of Theorem~\ref{th:rogers:real}. Since $g(\zeta_f(r)) = 0$, the gradient of $g$ at $\zeta_f(r)$ is orthogonal to $\zeta_f'(r)$. Therefore,
\formula{
 r \thet'(r) & = -\frac{r \tfrac{\partial}{\partial r} g(r e^{i \thet(r)})}{\tfrac{\partial}{\partial \thet} g(r e^{i \thet(r)})} = -\frac{(x \tfrac{\partial}{\partial x} + y \tfrac{\partial}{\partial y}) g(r e^{i \thet(r)})}{(x \tfrac{\partial}{\partial y} - y \tfrac{\partial}{\partial x}) g(r e^{i \thet(r)})} \, .
}
By~\eqref{eq:rogers:angular},
\formula{
 (x \tfrac{\partial}{\partial y} - y \tfrac{\partial}{\partial x}) g(x + i y) & = \frac{2}{\pi} \int_{-\infty}^\infty \frac{x |s|}{(x^2 + (s + y)^2)^2} \, \ph(s) \D s
}
for $x > 0$ and $y \in \R$; here $\ph$ is as in Theorem~\ref{th:rogers}\ref{it:rogers:c}. A similar calculation shows that
\formula[eq:rogers:radial]{
 (x \tfrac{\partial}{\partial x} + y \tfrac{\partial}{\partial y}) g(x + i y) & = \frac{1}{\pi} \int_{-\infty}^\infty \frac{-2 x^2 - 2 y (s + y)}{(x^2 + (s + y)^2)^2} \, (-\ph(s) \sign s) \D s \\
 & = -2 g(x + i y) - \frac{2}{\pi} \int_{-\infty}^\infty \frac{(s + y) |s|}{(x^2 + (s + y)^2)^2} \, \ph(s) \D s .
}
Since $g(r e^{i \thet(r)}) = 0$, it suffices to prove that for some $c > 0$ and all $r > 0$,
\formula{
 \abs{\int_{-\infty}^\infty \frac{(s + y) |s|}{(x^2 + (s + y)^2)^2} \, \ph(s) \D s} & \le c \int_{-\infty}^\infty \frac{x |s|}{(x^2 + (s + y)^2)^2} \, \ph(s) \D s ,
}
where $x + i y = r e^{i \thet(r)} = \zeta_f(r)$. Equivalently,
\formula{
 \int_{-\infty}^\infty \frac{c r \cos \thet(r) \pm (s + r \sin \thet(r))}{(r^2 + s^2 + 2 r s \sin \thet(r))^2} \, |s| \ph(s) \D s & \ge 0
}
for both choices of sign. After substituting $s = r u$, one needs to prove that
\formula[eq:rogers:reg:goal]{
 \int_{-\infty}^\infty \frac{c \cos \thet(r) \pm (u + \sin \thet(r))}{(1 + u^2 + 2 u \sin \thet(r))^2} \, |u| \ph(r u) \D u & \ge 0
}
for both choices of sign and all $r > 0$.

Let $\delta = \sup \set{|\Arg \zeta(r))| : r \in (0, \infty)}$. By the assumption, $\delta < \tfrac{\pi}{2}$, and by Theorem~\ref{th:rogers:real}, $g(r t e^{i \delta}) \ge 0$ and $g(r t^{-1} e^{-i \delta}) \le 0$ for all $t > 0$. By~\eqref{eq:rogers:real:aux1},
\formula{
 0 & \le t^2 g(r t^{-1} e^{-i \delta}) - t^{-2} g(r t^{-1} e^{-i \delta}) \displaybreak[0] \\
 & = \int_{-\infty}^\infty \expr{\frac{t^2 (-\ph(s) \sign s)}{(r t \cos \delta)^2 + (s + r t \sin \delta)^2} - \frac{t^{-2} (-\ph(s) \sign s)}{(r t^{-1} \cos \delta)^2 + (s - r t^{-1} \sin \delta)^2}} \D s \displaybreak[0] \\
 & = \int_{-\infty}^\infty \frac{(t + t^{-1}) (2 r \sin \delta - (t - t^{-1}) s)}{(r^2 t^2 + s^2 + 2 r t s \sin \delta) (r^2 t^{-2} + s^2 - 2 r t^{-1} s \sin \delta)} \, |s| \ph(s) \D s .
}
Substitution $s = r u$ gives
\formula[eq:rogers:reg:aux1]{
 \int_{-\infty}^\infty \frac{2 \sin \delta - (t - t^{-1}) u}{(t^2 + u^2 + 2 t u \sin \delta) (t^{-2} + u^2 - 2 t^{-1} u \sin \delta)} \, |u| \ph(r u) \D u & \ge 0 .
}
Define
\formula{
 c_1(\delta, u) & = \inf \set{\frac{(1 + u^2 + 2 u \sin \thet)^2}{(t^2 + u^2 + 2 t u \sin \delta) (t^{-2} + u^2 - 2 t^{-1} u \sin \delta)} : \thet \in [-\delta, \delta], \, t = \sqrt{2} \pm 1 } , \displaybreak[0] \\
 c_2(\delta, u) & = \sup \set{\frac{(1 + u^2 + 2 u \sin \thet)^2}{(t^2 + u^2 + 2 t u \sin \delta) (t^{-2} + u^2 - 2 t^{-1} u \sin \delta)} : \thet \in [-\delta, \delta], \, t = \sqrt{2} \pm 1 }
}
for $u \in \R$. By a simple calculation, $c_1(\delta, u)$ and $c_2(\delta, u)$ converge to $1$ as $u \nearrow \infty$ or $u \searrow -\infty$, and furthermore $u (c_1(\delta, u) - 1)$ and $u (c_2(\delta, u) - 1)$ are bounded in $u \in \R$.

By taking $t = \sqrt{2} - 1$ in~\eqref{eq:rogers:reg:aux1},
\formula{
 \int_{-\infty}^\infty \frac{(c_1(\delta, u) \ind_{(-\infty, -\sin \delta)}(u) + c_2(\delta, u) \ind_{(-\sin \delta, \infty)}(u)) (\sin \delta + u)}{(1 + u^2 + 2 u \sin \thet)^2} \, |u| \ph(r u) \D u & \ge 0
}
for all $\thet \in [-\delta, \delta]$. Therefore, if
\formula{
 c_3(\delta) & = \sup \set{(c_1(\delta, u) \ind_{(-\infty, -\sin \delta)}(u) + c_2(\delta, u) \ind_{(-\sin \delta, \infty)}(u)) (\sin \delta + u) - u} \in [0, \infty) ,
}
then it follows that
\formula{
 \int_{-\infty}^\infty \frac{c_3(\delta) + u}{(1 + u^2 + 2 u \sin \thet)^2} \, |u| \ph(r u) \D u & \ge 0
}
for all $\thet \in [-\delta, \delta]$ and $r > 0$. In a similar manner, by considering $t = \sqrt{2} + 1$,
\formula{
 \int_{-\infty}^\infty \frac{c_3(\delta) - u}{(1 + u^2 + 2 u \sin \thet)^2} \, |u| \ph(r u) \D u & \ge 0 .
}
Formula~\eqref{eq:rogers:reg:goal} follows, with $c = (c_3(\delta) + \sin \delta) / \cos \delta$ (recall that $\thet(r) \in [-\delta, \delta]$).
\end{proof}

\begin{example}
\label{ex:rogers:balanced}
\begin{enumerate}[label=\rm (\alph*)]
\item The Rogers function $f(\xi) = \tfrac{1}{2} \xi^2 - i b \xi$ (the L\'evy--Khintchine exponent of the Brownian motion with drift; $b \in \R$) is a balanced Rogers function if and only if $b = 0$.
\item The Rogers function $f(\xi) = a \xi^\alpha$ (the L\'evy--Khintchine exponent of a strictly stable L\'evy process; $\alpha \in (0, 2]$, $a$ satisfies~\eqref{eq:stable:a}) is balanced if and only if $|\Arg a| < \tfrac{\alpha \pi}{2}$ (see~\eqref{eq:stable:a}).
\item The Rogers function $f(\xi) = \xi / (\xi - a i) - i b \xi$ (the L\'evy--Khintchine exponent of the classical risk process; $a, b > 0$) is not a balanced Rogers function.
\end{enumerate}
\end{example}

\begin{proposition}
\label{prop:rogers:reg:lambda}
If $f$ is a balanced Rogers function, then
\formula[eq:rogers:reg:lambda:zero]{
 \lambda_f(0^+) & = f(0^+) = c_0 ,
}
where $c_0$ is the constant in the Stieltjes representation of $f$ (Theorem~\ref{th:rogers}\ref{it:rogers:b}). If $f$ is in addition bounded, then
\formula[eq:rogers:reg:lambda:infty]{
 \lambda_f(\infty^-) & = f(\infty^-) ;
}
otherwise, $\lim_{r \nearrow \infty} \lambda_f(r) = \infty$.
\end{proposition}

\begin{proof}
Since $f(\xi) - c_0$ is a Rogers function of $\xi$, by Proposition~\ref{prop:rogers:bound} (with $r = |\xi|$) and Definition~\ref{def:rogers:reg}, there is $c$ such that $|\lambda_f(r) - c_0| = |f(\zeta_f(r)) - c_0| \le c |f(r) - c_0|$ for all $r > 0$. This and~\eqref{eq:rogers:const} in Theorem~\ref{th:rogers} proves~\eqref{eq:rogers:reg:lambda:zero}.

If $f$ bounded, then it is given by~\eqref{eq:rogers:bounded} in Proposition~\ref{prop:rogers:bounded}. By condition~\ref{it:rogers:reg:d} in Definition~\ref{def:rogers:reg}, there is $c$ such that $|\xi / (\xi + i s)| \le |\xi| / \re \xi \le c$ for $\xi \in \gamma_f \cap \C_\ra$ and $s \in \R$. Hence, dominated convergence can be applied to~\eqref{eq:rogers:bounded}, and
\formula{
 \lim_{r \nearrow \infty} \lambda_f(r) & = c_0 + \frac{1}{\pi} \int_{\R \setminus \{0\}} \frac{\mu(\D s)}{|s|} = \lim_{\xi \nearrow \infty} f(\xi) ,
}
as desired. If $f$ is unbounded, then, by Proposition~\ref{prop:rogers:bound} (with $r = |\xi|$) and condition~\ref{it:rogers:reg:d} in Definition~\ref{def:rogers:reg}, there is $c > 0$ such that $|\lambda_f(r)| = |f(\zeta_f(r))| \ge c |f(r)|$, and so, by Proposition~\ref{prop:rogers:bounded}, $\lim_{r \nearrow \infty} \lambda_f(r) = \infty$.
\end{proof}

\subsection{Further formulae for the Wiener--Hopf factors}

Various useful expressions for the Wiener--Hopf factors can be obtained by appropriate contour deformations, and by integration by parts.

\begin{lemma}
\label{lem:rogers:curve}
If $f$ is a nonconstant Rogers function and $\gamma$ is a balanced curve, then
\formula[eq:rogers:wh:uu]{
 \frac{f^\ua(\xi_1)}{f^\ua(\xi_2)} & = \exp \expr{-\frac{1}{2 \pi i} \int_\gamma \expr{\frac{1}{i \xi_1 - z} - \frac{1}{i \xi_2 - z}} \log f(z) \D z}
}
for all $\xi_1, \xi_2$ such that $i \xi_1, i \xi_2 \in \gamma^\ua$, and
\formula[eq:rogers:wh:dd]{
 \frac{f^\da(\xi_1)}{f^\da(\xi_2)} & = \exp \expr{-\frac{1}{2 \pi i} \int_\gamma \expr{\frac{1}{i \xi_1 + z} - \frac{1}{i \xi_2 + z}} \log f(z) \D z}
}
for all $\xi_1, \xi_2$ such that $-i \xi_1, -i \xi_2 \in \gamma^\da$. Furthermore,
\formula[eq:rogers:wh:ud]{
 f^\ua(\xi_1) f^\da(\xi_2) & = \exp \expr{-\frac{1}{2 \pi i} \int_\gamma \expr{\frac{1}{i \xi_1 - z} + \frac{1}{i \xi_2 + z}} \log f(z) \D z}
}
for all $\xi_1, \xi_2$ such that $i \xi_1 \in \gamma^\ua$ and $-i \xi_2 \in \gamma^\da$.
\end{lemma}

Note that if $\xi_1, \xi_2 \in (0, \infty)$, then the above expressions can be rewritten as
\formula[eq:rogers:wh:real]{
 \frac{f^\ua(\xi_1)}{f^\ua(\xi_2)} & = \exp \expr{-\frac{1}{\pi} \int_0^\infty \im \expr{\expr{\frac{\zeta'(r)}{i \xi_1 - \zeta(r)} - \frac{\zeta'(r)}{i \xi_2 - \zeta(r)}} \log \lambda(r)} \D r} , \\
 \frac{f^\da(\xi_1)}{f^\da(\xi_2)} & = \exp \expr{-\frac{1}{\pi} \int_0^\infty \im \expr{\expr{\frac{\zeta'(r)}{i \xi_1 + \zeta(r)} - \frac{\zeta'(r)}{i \xi_2 + \zeta(r)}} \log \lambda(r)} \D r} , \\
 f^\ua(\xi_1) f^\da(\xi_2) & = \exp \expr{-\frac{1}{\pi} \int_0^\infty \im \expr{\expr{\frac{\zeta'(r)}{i \xi_1 - \zeta(r)} + \frac{\zeta'(r)}{i \xi_2 + \zeta(r)}} \log \lambda(r)} \D r} ,
}
where $\zeta(r)$ is the canonical parametrisation of $\gamma \cap \C_\ra$ and $\lambda(r) = f(\zeta(r))$.

\begin{proof}
By Proposition~\ref{prop:rogers:bound} (with $r = 1$) and conditions~\ref{it:rogers:reg:d} and~\ref{it:rogers:reg:e} in Definition~\ref{def:rogers:reg}, the integrals in~\eqref{eq:rogers:wh:uu}, \eqref{eq:rogers:wh:dd} and~\eqref{eq:rogers:wh:ud} are well-defined. Denote the right-hand side of formula~\eqref{eq:rogers:wh:uu} by $h(\xi_1, \xi_2)$. By an appropriate contour deformation (we omit the details),
\formula{
 h(\xi_1, \xi_2) & = \exp \expr{-\frac{1}{2 \pi i} \int_{(-i e^{-i \eps} \infty, 0) \cup (0, -i e^{i \eps} \infty)} \expr{\frac{1}{i \xi_1 - z} - \frac{1}{i \xi_2 - z}} \log f(z) \D z}
}
for all sufficiently small $\eps > 0$. By a natural parametrisation,
\formula{
 h(\xi_1, \xi_2) & = \exp \left( \frac{1}{2 \pi i} \int_0^\infty \left( \expr{\frac{e^{i \eps}}{\xi_1 + e^{i \eps} r} - \frac{e^{i \eps}}{\xi_2 + e^{i \eps} r}} \log f(-i e^{i \eps} r) \right. \right. \\
 & \hspace*{10em} \left. \left. {} - \expr{\frac{e^{-i \eps}}{\xi_1 + e^{-i \eps} r} - \frac{e^{-i \eps}}{\xi_2 + e^{-i \eps} r}} \log f(-i e^{-i \eps} r) \right) \D r \right) .
}
Recall that $f(-i e^{-i \eps} r)$ is the complex conjugate of $f(-i e^{i \eps} r)$, so that
\formula{
 h(\xi_1, \xi_2) & = \exp \left( \frac{1}{\pi} \int_0^\infty \left( \re \expr{\frac{e^{i \eps}}{\xi_1 + e^{i \eps} r} - \frac{e^{i \eps}}{\xi_2 + e^{i \eps} r}} \Arg f(-i e^{i \eps} r) \right. \right. \\
 & \hspace*{10em} \left. \left. {} + \im \expr{\frac{e^{i \eps}}{\xi_1 + e^{i \eps} r} - \frac{e^{i \eps}}{\xi_2 + e^{i \eps} r}} \log |f(-i e^{i \eps} r|) \right) \D r \right) .
}
By Theorem~\ref{th:harmonic:bounded} (see also~\eqref{eq:rogers:ph} in Theorem~\ref{th:rogers}), $\Arg f(-i e^{i \eps} r)$ converges as $\eps \searrow 0$ to $-\ph(r)$ for almost all $r > 0$, where $\ph$ is the function in the exponential representation of $f$ (Theorem~\ref{th:rogers}\ref{it:rogers:c}). By Proposition~\ref{prop:rogers:bound} (with $r = 1$), dominated convergence applies as $\eps \searrow 0$ (we omit the details), and it follows that
\formula{
 h(\xi_1, \xi_2) & = \exp \expr{\frac{1}{\pi} \int_0^\infty \expr{\frac{1}{\xi_1 + r} - \frac{1}{\xi_2 + r}} \, \ph(r) \D r} .
}
Formula~\eqref{eq:rogers:wh:uu} follows by the first part of~\eqref{eq:rogers:wh}. Formula~\eqref{eq:rogers:wh:dd} is proved in a similar way.

The proof of~\eqref{eq:rogers:wh:ud} follows the same line, but the residue theorem is used in contour deformation. Suppose first that $\xi_2$ is not real. By an argument as in the first part of the proof (we omit the details), the right-hand side of~\eqref{eq:rogers:wh:ud} is equal to
\formula{
 & \exp \expr{\log f(-i \xi_2) - \frac{1}{2 \pi i} \int_{(-i e^{-i \eps} \infty, 0) \cup (0, -i e^{i \eps} \infty)} \expr{\frac{1}{i \xi_1 - z} + \frac{1}{i \xi_2 + z}} \log f(z) \D z} \displaybreak[0] \\
 & \hspace*{1em} = \exp \left( \log f(-i \xi_2) + \frac{1}{2 \pi i} \int_0^\infty \left( \expr{\frac{e^{-i \eps}}{\xi_1 + e^{-i \eps} r} + \frac{e^{-i \eps}}{\xi_2 - e^{-i \eps} r}} \log f(-i e^{-i \eps} r) \right. \right. \\
 & \hspace*{17em} \left. \left. - \expr{\frac{e^{i \eps}}{\xi_1 + e^{i \eps} r} + \frac{e^{i \eps}}{\xi_2 - e^{i \eps} r}} \log f(-i e^{i \eps} r) \right) \D r \right)
}
for all $\eps > 0$ small enough. When $\eps \searrow 0$, the right-hand side converges to
\formula{
 \exp \expr{\log f(-i \xi_2) + \frac{1}{\pi} \int_0^\infty \expr{\frac{1}{\xi_1 + r} + \frac{1}{\xi_2 - r}} \, \ph(r) \D r} .
}
By~\eqref{eq:rogers:wh} and Theorem~\ref{th:rogers:wh}, the above expression is equal to
\formula{
 \frac{f(-i \xi_2) f^\ua(\xi_1)}{f^\ua(-\xi_2)} & = f^\ua(\xi_1) f^\ua(\xi_2) ,
}
as desired. Formula~\eqref{eq:rogers:wh:ud} for $\xi_2 \in (0, \infty)$ follows by continuity.
\end{proof}

\begin{corollary}
\label{cor:rogers:wh}
Let $f$ be a nonzero Rogers function. Then for all $\xi_1, \xi_2 \in \C_\ra$,
\formula[eq:rogers:wh:0]{
 \frac{f^\ua(\xi_1)}{f^\ua(\xi_2)} & = \exp \expr{\frac{1}{2 \pi} \int_{-\infty}^\infty \expr{\frac{1}{\xi_1 + i r} - \frac{1}{\xi_2 + i r}} \log f(r) \D r} , \\
 \frac{f^\da(\xi_1)}{f^\da(\xi_2)} & = \exp \expr{\frac{1}{2 \pi} \int_{-\infty}^\infty \expr{\frac{1}{\xi_1 - i r} - \frac{1}{\xi_2 - i r}} \log f(r) \D r} , \\
 f^\ua(\xi_1) f^\da(\xi_2) & = \exp \expr{\frac{1}{2 \pi} \int_{-\infty}^\infty \expr{\frac{1}{\xi_1 + i r} + \frac{1}{\xi_2 - i r}} \log f(r) \D r} .
}
\end{corollary}

\begin{proof}
If $f$ is nonconstant, then the result follows directly from Lemma~\ref{lem:rogers:curve}. Otherwise, when $f(\xi) = c$ for all $\xi \in (0, \infty)$, one easily checks that both sides of first two lines of~\eqref{eq:rogers:wh:0} are equal to $1$, while both sides of the third line of~\eqref{eq:rogers:wh:0} are equal to $c$.
\end{proof}

\begin{lemma}
\label{lem:rogers:wh:alt}
If $f$ is a nonconstant Rogers function, $\gamma$ is a balanced curve, $\zeta(r)$ is the canonical parametrisation of $\gamma \cap \C_\ra$, $\lambda(r) = f(\zeta(r))$ and $\xi_1, \xi_2, s > 0$, then
\formula[eq:rogers:wh:alt:uudd]{
 \log \frac{f^\ua(\xi_1)}{f^\ua(\xi_2)} & = \frac{1}{\pi} \int_0^\infty \Arg \expr{\frac{\zeta(r) - i \xi_2}{\zeta(r) - i \xi_1}} \re \frac{\lambda'(r)}{\lambda(r)} \, \D r - \frac{1}{\pi} \log \frac{\xi_1}{\xi_2} \Arg \lambda(s) \\
 & + \frac{1}{\pi} \int_0^\infty \expr{\log \abs{\frac{\zeta(r) - i \xi_2}{\zeta(r) - i \xi_1}} + \ind_{(0, s)}(r) \log \frac{\xi_1}{\xi_2}} \im \frac{\lambda'(r)}{\lambda(r)} \, \D r , \\
 \log \frac{f^\da(\xi_1)}{f^\da(\xi_2)} & = \frac{1}{\pi} \int_0^\infty \Arg \expr{\frac{\zeta(r) + i \xi_1}{\zeta(r) + i \xi_2}} \re \frac{\lambda'(r)}{\lambda(r)} \, \D r - \frac{1}{\pi} \log \frac{\xi_2}{\xi_1} \Arg \lambda(s) \\
 & + \frac{1}{\pi} \int_0^\infty \expr{\log \abs{\frac{\zeta(r) + i \xi_1}{\zeta(r) + i \xi_2}} + \ind_{(0, s)}(r) \log \frac{\xi_2}{\xi_1}} \im \frac{\lambda'(r)}{\lambda(r)} \, \D r ,
}
and
\formula[eq:rogers:wh:alt:ud]{
 \log (f^\ua(\xi_1) f^\da(\xi_2)) & = \frac{1}{\pi} \int_0^\infty \expr{\Arg \expr{\frac{\zeta(r) + i \xi_2}{\zeta(r) - i \xi_1}} - \pi \ind_{(0, s)}(r)} \re \frac{\lambda'(r)}{\lambda(r)} \, \D r \\
 & + \frac{1}{\pi} \int_0^\infty \expr{\log \abs{\frac{\zeta(r) + i \xi_2}{\zeta(r) - i \xi_1}} + \ind_{(0, s)}(r) \log \frac{\xi_1}{\xi_2}} \im \frac{\lambda'(r)}{\lambda(r)} \, \D r \\
 & + \log |\lambda(s)| - \frac{1}{\pi} \log \frac{\xi_1}{\xi_2} \Arg \lambda(s) .
}
\end{lemma}

\begin{proof}
By Lemma~\ref{lem:rogers:curve}, $f^\ua(\xi_1) / f^\ua(\xi_2)$ is given by~\eqref{eq:rogers:wh:real}. Integration by parts gives
\formula[eq:rogers:wh:alt:uu:1]{
 \hspace*{5em} & \hspace*{-5em} \int \expr{\frac{\zeta'(r)}{i \xi_1 - \zeta(r)} - \frac{\zeta'(r)}{i \xi_2 - \zeta(r)}} \log \lambda(r) \D r \\
 & = \log \expr{c \, \frac{\zeta(r) - i \xi_2}{\zeta(r) - i \xi_1}} \log \lambda(r) - \int \log \expr{c \, \frac{\zeta(r) - i \xi_2}{\zeta(r) - i \xi_1}} \frac{\lambda'(r)}{\lambda(r)} \, \D r
}
for arbitrary $c > 0$. By Proposition~\ref{prop:rogers:bound} (with $r = 1$) and condition~\ref{it:rogers:reg:d} in Definition~\ref{def:rogers:reg}, the function
\formula{
 \log \expr{c \, \frac{\zeta(r) - i \xi_2}{\zeta(r) - i \xi_1}} \log \lambda(r)
}
tends to $0$ as $r \nearrow \infty$ if $c = 1$, and tends to $0$ as $r \searrow 0$ if $c = \xi_1 / \xi_2$. It follows that for any $s > 0$,
\formula{
 \log \frac{f^\ua(\xi_1)}{f^\ua(\xi_2)} & = \frac{1}{\pi} \int_0^s \im \expr{\log \expr{\frac{\xi_1}{\xi_2} \, \frac{\zeta(r) - i \xi_2}{\zeta(r) - i \xi_1}} \frac{\lambda'(r)}{\lambda(r)}} \D r \\
 & \hspace*{3em} + \frac{1}{\pi} \int_s^\infty \im \expr{\log \expr{\frac{\zeta(r) - i \xi_2}{\zeta(r) - i \xi_1}} \frac{\lambda'(r)}{\lambda(r)}} \D r - \frac{1}{\pi} \log \frac{\xi_1}{\xi_2} \Arg \lambda(s) ,
}
where the integrals are improper integrals at $0$ and at $\infty$. We claim that these integrals are absolutely convergent. Indeed, by condition~\ref{it:rogers:reg:d} in Definition~\ref{def:rogers:reg} and a simple calculation (we omit the details), for some $c_1$ (depending on $\gamma$, $\xi_1$, $\xi_2$ and $s$),
\formula{
 \abs{\log \expr{\frac{\zeta(r) - i \xi_2}{\zeta(r) - i \xi_1}}} & \le \frac{c_1}{r}
}
for $r \in (s, \infty)$, and
\formula{
 \abs{\log \expr{\frac{\xi_1}{\xi_2} \, \frac{\zeta(r) - i \xi_2}{\zeta(r) - i \xi_1}}} & \le c_1 r
}
for $r \in (0, s)$. By Proposition~\ref{prop:rogers:prime} and condition~\ref{it:rogers:reg:d} in Definition~\ref{def:rogers:reg}, for some $c_2$ (depending on $\gamma$ and $f$),
\formula[eq:rogers:wh:est:aux]{
 \abs{\frac{\lambda'(r)}{\lambda(r)}} & = \abs{\frac{f'(\zeta(r)) \zeta'(r)}{f(\zeta(r))}} \le c_2 \abs{\frac{\zeta'(r)}{\zeta(r)}} = \frac{c_2 |\zeta'(r)|}{r} \, .
}
The claim follows by condition~\ref{it:rogers:reg:e} in Definition~\ref{def:rogers:reg}, and the first part of~\eqref{eq:rogers:wh:alt:uudd} is proved.

The other part of~\eqref{eq:rogers:wh:alt:uudd} is proved in a similar way. The proof of~\eqref{eq:rogers:wh:alt:ud} is slightly different, in this case one uses $c = -\xi_1 / \xi_2$ for the integral over $(0, s)$. Note that
\formula{
 \log \expr{-\frac{\xi_1}{\xi_2} \, \frac{\zeta(r) + i \xi_2}{\zeta(r) - i \xi_1}} & = \log \frac{\zeta(r) + i \xi_2}{\zeta(r) - i \xi_1} + \log \frac{\xi_1}{\xi_2} - i \pi .
}
This leads to
\formula{
 \log (f^\ua(\xi_1) f^\da(\xi_2)) \hspace*{-3em} & \hspace*{3em} = \frac{1}{\pi} \int_0^s \im \expr{\log \expr{-\frac{\xi_1}{\xi_2} \, \frac{\zeta(r) + i \xi_2}{\zeta(r) - i \xi_1}} \frac{\lambda'(r)}{\lambda(r)}} \D r \\
 & + \frac{1}{\pi} \int_s^\infty \im \expr{\log \expr{\frac{\zeta(r) + i \xi_2}{\zeta(r) - i \xi_1}} \frac{\lambda'(r)}{\lambda(r)}} \D r - \frac{1}{\pi} \log \frac{\xi_1}{\xi_2} \Arg \lambda(s) + \log |\lambda(s)| ,
}
as desired.
\end{proof}

\begin{corollary}
\label{cor:rogers:wh:alt}
If $f$ is a balanced Rogers function, $\zeta_f$ is the canonical parametrisation of $\gamma_f \cap \C_\ra$, $\lambda_f(r) = f(\zeta_f(r))$ and $\xi_1, \xi_2 > 0$, then
\formula[eq:rogers:wh:alt:uu0]{
 \frac{f^\ua(\xi_1)}{f^\ua(\xi_2)} & = \exp \expr{-\frac{1}{\pi} \int_0^\infty \im \expr{\frac{\zeta_f'(r)}{i \xi_1 - \zeta_f(r)} - \frac{\zeta_f'(r)}{i \xi_2 - \zeta_f(r)}} \log \lambda_f(r) \D r} \\
 & = \exp \expr{\frac{1}{\pi} \int_0^\infty \Arg \expr{\frac{\zeta_f(r) - i \xi_2}{\zeta_f(r) - i \xi_1}} \frac{\lambda_f'(r)}{\lambda_f(r)} \, \D r} ,
}
and
\formula[eq:rogers:wh:alt:dd0]{
 \frac{f^\da(\xi_1)}{f^\da(\xi_2)} & = \exp \expr{-\frac{1}{\pi} \int_0^\infty \im \expr{\frac{\zeta_f'(r)}{i \xi_1 + \zeta_f(r)} - \frac{\zeta_f'(r)}{i \xi_2 + \zeta_f(r)}} \log \lambda_f(r) \D r} \\
 & = \exp \expr{\frac{1}{\pi} \int_0^\infty \Arg \expr{\frac{\zeta_f(r) + i \xi_1}{\zeta_f(r) + i \xi_2}} \frac{\lambda_f'(r)}{\lambda_f(r)} \, \D r} .
}
Furthermore,
\formula[eq:rogers:wh:alt:ud1]{
 f^\ua(\xi_1) f^\da(\xi_2) & = \exp \expr{-\frac{1}{\pi} \int_0^\infty \im \expr{\frac{\zeta_f'(r)}{i \xi_1 - \zeta_f(r)} + \frac{\zeta_f'(r)}{i \xi_2 + \zeta_f(r)}} \log \lambda_f(r) \D r} \\
 & = \lambda_f(s) \exp \expr{\frac{1}{\pi} \int_0^\infty \expr{\Arg \expr{\frac{\zeta_f(r) + i \xi_2}{\zeta_f(r) - i \xi_1}} - \pi \ind_{(0, s)}(r)} \frac{\lambda_f'(r)}{\lambda_f(r)} \, \D r}
}
for all $s > 0$, and if $f(0^+) > 0$, then
\formula[eq:rogers:wh:alt:ud0]{
 f^\ua(\xi_1) f^\da(\xi_2) & = f(0^+) \exp \expr{\frac{1}{\pi} \int_0^\infty \Arg \expr{\frac{\zeta_f(r) + i \xi_2}{\zeta_f(r) - i \xi_1}} \frac{\lambda_f'(r)}{\lambda_f(r)} \, \D r}
}
for all $\xi_1, \xi_2 > 0$.
\end{corollary}

\begin{proof}
First three statements follow directly from Lemmas~\ref{lem:rogers:curve} and~\ref{lem:rogers:wh:alt}. Suppose that $f(0^+)$ is positive. In this case, by Theorem~\ref{th:rogers}\ref{it:rogers:b} (see~\eqref{eq:rogers:const}), $f(\xi) - f(0^+)$ is a Rogers function, and by Proposition~\ref{prop:rogers:bound} and condition~\ref{it:rogers:reg:d} in Definition~\ref{def:rogers:reg}, $\lambda_f(s) - f(0^+) = f(\zeta_f(s)) - f(0^+)$ converges to $0$ as $s \searrow 0$. Furthermore, $\lambda_f'(r) / \lambda_f(r)$ is integrable over $(0, 1)$, so by dominated convergence, the integral in~\eqref{eq:rogers:wh:alt:ud1} converges to the integral in~\eqref{eq:rogers:wh:alt:ud0} as $s \searrow 0$. This proves~\eqref{eq:rogers:wh:alt:ud0}.
\end{proof}

\begin{remark}
\label{rem:rogers:substitution}
By a substitution $s = \lambda_f(r)$, the expressions given in Corollary~\ref{cor:rogers:wh:alt} can be rewritten as
\formula[eq:rogers:wh:alt:subst]{
 \frac{f^\ua(\xi_1)}{f^\ua(\xi_2)} & = \exp \expr{\frac{1}{\pi} \int_{f(0^+)}^{f(\infty^-)} \Arg \expr{\frac{f^{-1}(s) - i \xi_2}{f^{-1}(s) - i \xi_1}} \frac{1}{s} \, \D s} , \\
 \frac{f^\da(\xi_1)}{f^\da(\xi_2)} & = \exp \expr{\frac{1}{\pi} \int_{f(0^+)}^{f(\infty^-)} \Arg \expr{\frac{f^{-1}(s) + i \xi_1}{f^{-1}(s) + i \xi_2}} \frac{1}{s} \, \D s} , \\
 f^\ua(\xi_1) f^\da(\xi_2) & = f(0^+) \exp \expr{\frac{1}{\pi} \int_{f(0^+)}^{f(\infty^-)} \Arg \expr{\frac{f^{-1}(s) + i \xi_2}{f^{-1}(s) - i \xi_1}} \frac{1}{s} \, \D s} ,
}
where $f(0^+) = \lambda_f(0^+)$ and, slightly abusing the notation, $f(\infty^-) = \lambda_f(\infty^-)$.
\end{remark}

\begin{remark}
\label{rem:rogers:extension}
Lemma~\ref{lem:rogers:curve} apparently can be extended to more general curves $\gamma$, e.g.\ if the integral is understood as an improper Lebesgue integral. In a similar way, expressions given in Lemma~\ref{lem:rogers:wh:alt} and Corollary~\ref{cor:rogers:wh:alt} seemingly can be extended to more general Rogers functions. However, mostly due to technical problems with handling improper integrals (e.g.\ when using Morera's theorem), this article deals only with balanced curves.
\end{remark}

\subsection{Ratios of Rogers functions}

The following result is crucial for the normalisation of extended Wiener--Hopf factors in Section~\ref{sec:norm}.

\begin{lemma}
\label{lem:rogers:wh:limit}
Let $f, \tilde{f}$ be nonzero Rogers functions. If $\lim_{\xi \searrow 0} (f(\xi) / \tilde{f}(\xi))$ exists and it is a positive number, and $\xi^{-1} \Arg (f(\xi) / \tilde{f}(\xi))$ is integrable in $\xi \in (0, 1)$, then the limits of $f^\ua(\xi) / \tilde{f}^\ua(\xi)$ and $f^\da(\xi) / \tilde{f}^\da(\xi)$ as $\xi \searrow 0$ exist and they are positive numbers. More precisely,
\formula{
 \lim_{\xi \searrow 0} \frac{f^\ua(\xi)}{\tilde{f}^\ua(\xi)} & = \frac{f^\ua(1)}{\tilde{f}^\ua(1)} \, \exp\expr{\frac{1}{\pi} \, \int_0^\infty \frac{1}{1 + r^2} \, \log \abs{\frac{f(r)}{\tilde{f}(r)}} \D r} \times \\
 & \hspace*{3em} \times \expr{\lim_{\xi \searrow 0} \frac{f(\xi)}{\tilde{f}(\xi)}}^{1/2} \exp\expr{\frac{1}{\pi} \, \int_0^\infty \frac{1}{r (1 + r^2)} \, \Arg \expr{\frac{f(r)}{\tilde{f}(r)}} \D r} , \displaybreak[0] \\
 \lim_{\xi \searrow 0} \frac{f^\da(\xi)}{\tilde{f}^\da(\xi)} & = \frac{f^\da(1)}{\tilde{f}^\da(1)} \, \exp\expr{\frac{1}{\pi} \, \int_0^\infty \frac{1}{1 + r^2} \, \log \abs{\frac{f(r)}{\tilde{f}(r)}} \D r} \times \\
 & \hspace*{3em} \times \expr{\lim_{\xi \searrow 0} \frac{f(\xi)}{\tilde{f}(\xi)}}^{1/2} \exp\expr{-\frac{1}{\pi} \, \int_0^\infty \frac{1}{r (1 + r^2)} \, \Arg \expr{\frac{f(r)}{\tilde{f}(r)}} \D r} .
}
Similarly, if $\lim_{\xi \to \infty} (f(\xi) / \tilde{f}(\xi))$ exists and $\xi^{-1} \Arg (f(\xi) / \tilde{f}(\xi))$ is integrable in $\xi \in (1, \infty)$, then
\formula{
 \lim_{\xi \nearrow \infty} \frac{f^\ua(\xi)}{\tilde{f}^\ua(\xi)} & = \frac{f^\ua(1)}{\tilde{f}^\ua(1)} \, \exp\expr{\frac{1}{\pi} \, \int_0^\infty \frac{1}{1 + r^2} \, \log \abs{\frac{f(r)}{\tilde{f}(r)}} \D r} \times \\
 & \hspace*{3em} \times \expr{\lim_{\xi \nearrow \infty} \frac{f(\xi)}{\tilde{f}(\xi)}}^{1/2} \exp\expr{\frac{1}{\pi} \, \int_0^\infty \frac{r}{1 + r^2} \, \Arg \expr{\frac{f(r)}{\tilde{f}(r)}} \D r} , \displaybreak[0] \\
 \lim_{\xi \nearrow \infty} \frac{f^\da(\xi)}{\tilde{f}^\da(\xi)} & = \frac{f^\da(1)}{\tilde{f}^\da(1)} \, \exp\expr{\frac{1}{\pi} \, \int_0^\infty \frac{1}{1 + r^2} \, \log \abs{\frac{f(r)}{\tilde{f}(r)}} \D r} \times \\
 & \hspace*{3em} \times \expr{\lim_{\xi \nearrow \infty} \frac{f(\xi)}{\tilde{f}(\xi)}}^{1/2} \exp\expr{-\frac{1}{\pi} \, \int_0^\infty \frac{r}{1 + r^2} \, \Arg \expr{\frac{f(r)}{\tilde{f}(r)}} \D r} .
}
\end{lemma}

\begin{proof}
Let $a = \lim_{\xi \searrow 0} (f(\xi) / \tilde{f}(\xi))$. By Lemma~\ref{lem:rogers:curve} (or Corollary~\ref{cor:rogers:wh}), for $\xi > 0$,
\formula{
 \frac{\tilde{f}^\ua(1) f^\ua(\xi)}{f^\ua(1) \tilde{f}^\ua(\xi)} & = \exp \expr{-\frac{1}{\pi} \int_{-\infty}^\infty \im \expr{\expr{\frac{1}{i \xi - r} - \frac{1}{i - r}} \log \frac{f(r)}{\tilde{f}(r)}} \D r} .
}
Observe that
\formula{
 -\expr{\frac{1}{i \xi - r} - \frac{1}{i - r}} & = i \expr{\frac{\xi}{\xi^2 + r^2} - \frac{1}{1 + r^2}} + \frac{r (1 - \xi^2)}{(\xi^2 + r^2) (1 + r^2)} \, .
}
By the assumption and Proposition~\ref{prop:rogers:bound} (with $r = 1$), there is $c$ such that
\formula{
 \abs{\log \frac{f(r)}{\tilde{f}(r)}} & \le c (1 + \log(1 + r^2))
}
for $r \in \R \setminus \{0\}$. By dominated convergence in the integrals over $(-\infty, -1]$ and $[1, \infty)$, and vague convergence of $\xi (\xi^2 + r^2)^{-1} \D r$ to $\pi \delta_0(\D r)$ in the integral over $(-1, 1)$, it follows that
\formula{
 \lim_{\xi \searrow 0} \int_{-\infty}^\infty \frac{\xi}{\xi^2 + r^2} \, \log \frac{f(r)}{\tilde{f}(r)} \, \D r & = \pi a .
}
Furthermore, by the assumption and dominated convergence,
\formula{
 \lim_{\xi \searrow 0} \int_{-\infty}^\infty \frac{r (1 - \xi^2)}{(1 + r^2) (\xi^2 + r^2)} \log \frac{f(r)}{\tilde{f}(r)} \, \D r & = 2 \lim_{\xi \searrow 0} \int_0^\infty \frac{r (1 - \xi^2)}{(1 + r^2) (\xi^2 + r^2)} \Arg \expr{\frac{f(r)}{\tilde{f}(r)}} \D r \displaybreak[0] \\
 & = 2 \int_0^\infty \frac{1}{r (1 + r^2)} \Arg \expr{\frac{f(r)}{\tilde{f}(r)}} \D r .
}
This proves that $f^\ua(\xi) / \tilde{f}^\ua(\xi)$ has the desired limit as $\xi \searrow 0$. A similar argument proves that $f^\da(\xi) / \tilde{f}^\da(\xi)$ converges. The statement about the limits at infinity is proved in a similar way. (Alternatively, one can reuse the result for limits at $0$ by considering Rogers functions $1 / f(1/\xi)$ and $1 / \tilde{f}(1/\xi)$.)
\end{proof}

\subsection{Nearly balanced Rogers functions}

In the next section, most results are proved for balanced Rogers functions. The following two lemmas are used to extend some of them to slightly more general context.

\begin{lemma}
\label{lem:rogers:translation}
If $f$ is a nonconstant Rogers function, $0 \le r_0 < r_\infty \le \infty$, $|\gamma_f| \sub (r_0, r_\infty)$, $\zeta_0 = \zeta_f(r_0)$ and $\zeta_\infty = \zeta_f(r_\infty)$ (where $\zeta_f$ is the canonical representation of $\gamma_f$, $\zeta_f(0) = 0$ and $\zeta_f(\infty)$ is complex infinity), then
\formula{
 g(\xi) & = f(\xi + \zeta_0)
}
defines a Rogers function $g$ such that $|\gamma_f| \sub (0, |\zeta_\infty - \zeta_0|)$. Furthermore, the Wiener--Hopf factors of $f$ and $g$ are related to each other by
\formula[eq:rogers:wh:translation]{
\begin{gathered}
 \begin{aligned} \frac{g^\ua(\xi_1)}{g^\ua(\xi_2)} & = \frac{f^\ua(-i (i \xi_1 + \zeta_0))}{f^\ua(-i (i \xi_2 + \zeta_0))} , \qquad & \frac{g^\da(\xi_1)}{g^\da(\xi_2)} & = \frac{f^\da(i (-i \xi_1 + \zeta_0))}{f^\da(i (-i \xi_2 + \zeta_0))} , \end{aligned} \\
 \begin{aligned} g^\ua(\xi_1) g^\da(\xi_2) & = f^\ua(-i (i \xi_1 + \zeta_0)) f^\da(i (-i \xi_2 + \zeta_0)) \end{aligned}
\end{gathered}
}
for all $\xi_1, \xi_2$ for which both sides of these equalities are defined.
\end{lemma}

\begin{proof}
Note that by Theorem~\ref{th:rogers:real}, $\zeta_f((0, r_0))$ and $\zeta_f((r_\infty, \infty))$ are contained in $\dom_f$, and $\lambda_f(r) = f(\zeta_f(r))$ is nonnegative and increasing on $(0, \infty) \setminus \{r_0, r_\infty\}$.

If $r_0 = 0$, then $\zeta_0 = 0$ and $g = f$. If $r_0 > 0$, then either $\zeta_0 = i r_0$ or $\zeta_0 = -i r_0$. Consider first the case $\zeta_0 = i r_0$. The function $\Arg (g(\xi) / \xi) = \Arg f(\xi + \zeta_0) - \Arg \xi$ is a bounded harmonic function in $\C_\ra$, and with the notation of~\eqref{eq:rogers:c} in Theorem~\ref{th:rogers}\ref{it:rogers:c}, its boundary values are given by
\formula[eq:rogers:translation:boundary]{
 \lim_{t \searrow 0} \frac{\Arg g(t + i s)}{t + i s} & = \lim_{t \searrow 0} \Arg f(t + i s + i r_0) - \tfrac{\pi}{2} \sign s \\
 & = \sign(s + r_0) \ph(-s - r_0) - \tfrac{\pi}{2} \sign s
}
for almost all $s \in \R$.

Recall that $\ph$ takes values in $[0, \pi]$, and that $\ph$ vanishes almost everywhere on $(-r_0, 0)$ (because $(0, i r_0) \sub \dom_f$). Therefore, the right-hand side of~\eqref{eq:rogers:translation:boundary} is in $[-\tfrac{\pi}{2}, \tfrac{\pi}{2}]$ when $s > 0$ or $s < -r_0$, and it is equal to $\tfrac{\pi}{2}$ for almost all $s \in (-r_0, 0)$. By the maximum principle, $\Arg (g(\xi) / \xi) \in [-\tfrac{\pi}{2}, \tfrac{\pi}{2}]$ for all $\xi \in \C_\ra$. By Theorem~\ref{th:rogers}\ref{it:rogers:d}, $g$ is a Rogers function. Clearly, $\gamma_g$ is the translation of $\gamma_f$, hence $|\gamma_g| \sub (0, |\zeta_\infty - \zeta_0|)$.

Observe that $f^\ua(\xi - i \zeta_0) = f^\ua(\xi + r_0)$ is a complete Bernstein function of $\xi$. By Theorem~\ref{th:rogers:wh} (see~\eqref{eq:rogers:wh} and Remark~\ref{rem:rogers:wh:ext}), $f^\da(\xi + i \zeta_0) = f^\da(\xi - r_0)$ is holomorphic in $\xi \in \C \setminus (-\infty, 0]$ and $f^\da(\xi - r_0) \ge 0$ for $\xi \in (0, \infty)$. Furthermore, $\im f^\da(\xi - r_0) \ge 0$ when $\xi \in \C_\ua$. Hence, $f^\da(\xi - r_0)$ is a complete Bernstein function of $\xi$.

Since $f^\ua(-i (\xi + i \zeta_0)) f^\da(i (\xi + i \zeta_0)) = f(\xi + i \zeta_0) = g(\xi)$ for all $\xi \in \C \setminus i \R$, the functions $f^\ua(\xi - i \zeta_0)$ and $f^\da(\xi + i \zeta_0)$ are Wiener--Hopf factors of $g$, and therefore the are equal to $c g^\ua(\xi)$ and $c^{-1} g^\da(\xi)$, respectively, for some $c > 0$. This proves~\eqref{eq:rogers:wh:translation} for $\xi_1, \xi_2 \in \C \setminus i \R$. Extension to $\xi_1, \xi_2 \in i \R$ for which both sides make sense follows by continuity.

The proof in the case $\zeta_0 = -i r_0$ is very similar.
\end{proof}

\begin{lemma}
\label{lem:rogers:mobius}
With the assumptions and notation of Lemma~\ref{lem:rogers:translation},
\formula{
 g(\xi) & = f(u(\xi)) , & \text{where} && u(\xi) & = \zeta_0 + \frac{(\zeta_\infty - \zeta_0) \xi}{\xi + (\zeta_\infty - \zeta_0)}
}
(with the convention that $(\zeta_\infty - \zeta_0) \xi / (\xi - \zeta_\infty + \zeta_\infty) = \xi$ if $\zeta_\infty$ is complex infinity), defines a Rogers function $g$. Furthermore, the Wiener--Hopf factors of $f$ and $g$ are related to each other by
\formula[eq:rogers:wh:mobius]{
\begin{gathered}
 \begin{aligned} \frac{g^\ua(\xi_1)}{g^\ua(\xi_2)} & = \frac{f^\ua(-i u(i \xi_1))}{f^\ua(-i u(i \xi_2))} , & \frac{g^\da(\xi_1)}{g^\da(\xi_2)} & = \frac{f^\da(i u(-i \xi_1))}{f^\da(i u(-i \xi_2))} , \end{aligned} \\
 \begin{aligned} g^\ua(\xi_1) g^\da(\xi_2) & = f^\ua(-i u(i \xi_1)) f^\da(i u(-i \xi_2)) \end{aligned}
\end{gathered}
}
for all $\xi_1, \xi_2$ for which both sides of these equalities are defined.
\end{lemma}

Note that $u$ is the M\"obius transformation which maps $\zeta_0$ to $0$ and $\zeta_\infty$ to $\infty$, and preserves $\C_\ra$ and $i \R \cup \{\infty\}$ (here $\infty$ stands for the complex infinity). When $\zeta_\infty$ is complex infinity, then $u(\xi) = \xi + \zeta_0$ is a translation, and the statement of the lemma reduces to Lemma~\ref{lem:rogers:translation}. One easily finds that
\formula{
 u^{-1}(\xi) & = (\zeta_0 - \zeta_\infty) \, \frac{\xi - \zeta_0}{\xi - \zeta_\infty} \, .
}
Furthermore, $u^{-1}(0) = (\zeta_0 / \zeta_\infty) (\zeta_0 - \zeta_\infty)$ and $u^{-1}(\infty) = (\zeta_0 - \zeta_\infty)$.

\begin{proof}
By a direct calculation, $u(\xi) = u_0(u_\infty(\xi))$, where $u_0(\xi) = \xi + \zeta_0$ is a translation, and $u_\infty(\xi) = (1 / \xi + 1 / (\zeta_\infty - \zeta_0))^{-1}$ can be viewed as a translation in the $1 / \xi$ variable. By Lemma~\ref{lem:rogers:translation}, $g_0(\xi) = f(u_0(\xi))$ is a Rogers function, with $|\gamma_{g_0}| \sub (0, |\zeta_\infty - \zeta_0|)$. Hence, $g_1(\xi) = 1 / g_0(1 / \xi)$ is a Rogers function, and $|\gamma_{g_1}| \sub (|\zeta_\infty - \zeta_0|^{-1}, \infty)$. By another application of Lemma~\ref{lem:rogers:translation}, $g_\infty(\xi) = g_1(\xi + 1 / (\zeta_\infty - \zeta_0))$ is a Rogers function. Therefore, also 
\formula{
 \frac{1}{g_\infty(1 / \xi)} & = \frac{1}{g_1(1 / \xi + 1 / (\zeta_\infty - \zeta_0))} = g_0\!\expr{\frac{1}{1 / \xi + 1 / (\zeta_\infty - \zeta_0)}} = g_0(u_\infty(\xi)) = g(\xi)
}
is a Rogers function, as desired. In a similar way one shows that a double application of~\eqref{eq:rogers:wh:translation} proves~\eqref{eq:rogers:wh:mobius}.
\end{proof}

\begin{definition}
\label{def:rogers:nearly}
A nonconstant Rogers function $f$ is \emph{nearly balanced} if $|\gamma_f| = (r_0, r_\infty)$ for some $r_0$, $r_\infty$, and the Rogers function $g$ defined in Lemma~\ref{lem:rogers:mobius} is balanced.
\end{definition}

\begin{example}
\label{ex:rogers:nearly}
\begin{enumerate}[label=\rm (\alph*)]
\item The Rogers function $f(\xi) = \tfrac{1}{2} \xi^2 - i b \xi$ (the L\'evy--Khintchine exponent of the Brownian motion with drift; $b \in \R$) is a nearly balanced Rogers function.
\item The Rogers function $f(\xi) = a \xi^\alpha$ (the L\'evy--Khintchine exponent of a strictly stable L\'evy process; $\alpha \in (0, 2]$, $a$ satisfies~\eqref{eq:stable:a}) is nearly balanced if and only if it is balanced, that is, $|\Arg a| < \tfrac{\alpha \pi}{2}$ (see~\eqref{eq:stable:a}).
\item The Rogers function $f(\xi) = a \xi^\alpha - i b \xi$ (the L\'evy--Khintchine exponent of a nonstrictly stable L\'evy process; $\alpha \in (0, 1) \cup (1, 2]$, $a$ satisfies~\eqref{eq:stable:a}, $b \in \R \setminus \{0\}$) is nearly balanced if and only if $\alpha = 2$.
\item The Rogers function $f(\xi) = a \xi (1 + i \beta \tfrac{2}{\pi} \log \xi) - i b \xi$ (the L\'evy--Khintchine exponent of a nonstrictly stable L\'evy process with stability index $1$; $a > 0$, $b \in \R$, $\beta \in [-1, 1] \setminus \{0\}$) is not a nearly balanced function.
\item The Rogers function $f(\xi) = \xi / (\xi - a i) - i b \xi$ (the L\'evy--Khintchine exponent of the classical risk process; $a, b > 0$) is a nearly balanced Rogers function.
\end{enumerate}
\end{example}

%
%

\section{Extended Wiener--Hopf factorisation}
\label{sec:xwh}

The Wiener--Hopf factorisation in the $\xi$ variable of $f(\xi) + \tau$ for $\tau > 0$ and $f \in \rogers$, referred to as \emph{extended Wiener--Hopf factorisation}, plays an important role in the fluctuation theory of L\'evy processes. This topic is addressed in the following two sections, while applications are studied in Section~\ref{sec:ft}.

This section is the most technical one in the article. Its results do not depend on normalisation and they are stated for the Wiener--Hopf factors of $f(\xi) + \tau$ as defined above. In the next section the definitions are modified in order to agree with standard notions for L\'evy processes, where the limits of Wiener--Hopf factors of $f(\xi) + \tau$ are asymptotically equal as $\xi \nearrow \infty$.

In Lemmas~\ref{lem:rogers:xwh} and~\ref{lem:rogers:xwh:cbf} the ratios of extended Wiener--Hopf factors of balanced Rogers functions are proved to be complete Bernstein functions in $\tau$. This is extended to nearly balanced Rogers functions in Lemma~\ref{lem:rogers:xwh:nearly}. In Lemma~\ref{lem:rogers:xwh:boundary}, the boundary limits in $\tau$ at $(-\infty, 0)$ are calculated, which, in view of Theorem~\ref{th:rogers}\ref{it:rogers:b}, yields a formula for the measure $\mu$ in the representation of the ratios. Finally, an integral identity for the Wiener--Hopf factors is given in Lemma~\ref{lem:rogers:xwh:int}.

\subsection{Holomorphic extensions of extended Wiener--Hopf factors}

In this section by saying that a statement holds for \emph{admissible} $\xi_1, \xi_2$, it is understood that $\xi_1, \xi_2$ satisfy the appropriate assumptions of Lemma~\ref{lem:rogers:curve} for $\gamma = \gamma_f$. The following definition introduces a convenient notation for $(f+\tau)^\ua(\xi)$ and $(f+\tau)^\da(\xi)$.

\begin{definition}
\label{def:rogers:xwh}
Suppose that $f \in \rogers$, $\tau > 0$ and $g(\xi) = f(\xi) + \tau$. We define $f^\ua(\tau; \xi) = g^\ua(\xi)$, $f^\da(\tau; \xi) = g^\da(\xi)$.
\end{definition}

\begin{remark}
\label{rem:rogers:xwh}
Due to the normalisation of the Wiener--Hopf factors, $f^\ua(\tau; 1) = f^\da(\tau; 1)$ in the above definition. However, this is not the desired condition in the application to the fluctuation theory of L\'evy processes. The results of this part do not depend on the choice of the normalisation, which will become important in the next section.
\end{remark}

\begin{lemma}
\label{lem:rogers:xwh}
Let $f$ be a balanced Rogers function. Then for all admissible $\xi_1, \xi_2$, $f^\ua(\tau; \xi_1) / f^\ua(\tau; \xi_2)$, $f^\da(\tau; \xi_1) / f^\da(\tau; \xi_2)$ and $f^\ua(\tau; \xi_1) f^\da(\tau; \xi_2)$ extend to holomorphic functions of $\tau \in \C \setminus (-\infty, 0]$. If $\zeta_f$ is the canonical parametrisation of $\gamma_f \cap \C_\ra$ and $\lambda_f(r) = f(\zeta_f(r))$, then for all $\xi_1, \xi_2 > 0$ and all $\tau \in \C \setminus(-\infty, 0]$,
\formula[eq:rogers:xwh:uudd]{
 \frac{f^\ua(\tau; \xi_1)}{f^\ua(\tau; \xi_2)} & = \exp \expr{-\frac{1}{\pi} \int_0^\infty \im \expr{\frac{\zeta_f'(r)}{i \xi_1 - \zeta_f(r)} - \frac{\zeta_f'(r)}{i \xi_2 - \zeta_f(r)}} \log (\lambda_f(r) + \tau) \D r} \\
 & = \exp \expr{\frac{1}{\pi} \int_0^\infty \Arg \expr{\frac{\zeta_f(r) - i \xi_2}{\zeta_f(r) - i \xi_1}} \frac{\lambda_f'(r)}{\lambda_f(r) + \tau} \, \D r} , \\
 \frac{f^\da(\tau; \xi_1)}{f^\da(\tau; \xi_2)} & = \exp \expr{-\frac{1}{\pi} \int_0^\infty \im \expr{\frac{\zeta_f'(r)}{i \xi_1 + \zeta_f(r)} - \frac{\zeta_f'(r)}{i \xi_2 + \zeta_f(r)}} \log (\lambda_f(r) + \tau) \D r} \\
 & = \exp \expr{\frac{1}{\pi} \int_0^\infty \Arg \expr{\frac{\zeta_f(r) + i \xi_1}{\zeta_f(r) + i \xi_2}} \frac{\lambda_f'(r)}{\lambda_f(r) + \tau} \, \D r} ,
}
and
\formula[eq:rogers:xwh:ud]{
 & f^\ua(\tau; \xi_1) f^\da(\tau; \xi_2) \\
 & \qquad = \exp \expr{-\frac{1}{\pi} \int_0^\infty \im \expr{\frac{\zeta_f'(r)}{i \xi_1 - \zeta_f(r)} + \frac{\zeta_f'(r)}{i \xi_2 + \zeta_f(r)}} \log (\lambda_f(r) + \tau) \D r} \\
 & \qquad = \tau \exp \expr{\frac{1}{\pi} \int_0^\infty \Arg \expr{\frac{\zeta_f(r) + i \xi_2}{\zeta_f(r) - i \xi_1}} \frac{\lambda_f'(r)}{\lambda_f(r) + \tau} \, \D r} .
}
\end{lemma}

\begin{proof}
By Lemma~\ref{lem:rogers:curve}, if $i \xi_1, i \xi_2 \in \gamma_f^\ua$ and $\tau > 0$,
\formula[eq:rogers:xwh:uu:log]{
 \log \frac{f^\ua(\tau; \xi_1)}{f^\ua(\tau; \xi_2)} & = -\frac{1}{2 \pi i} \int_{\gamma_f} \expr{\frac{1}{i \xi_1 - z} - \frac{1}{i \xi_2 - z}} \log (f(z) + \tau) \D z .
}
Since $f$ takes positive values on $\gamma_f$, the integrand in the right-hand side extends to a holomorphic function of $\tau \in \C \setminus (-\infty, 0]$ for each $z \in \gamma_f$. We claim that by Morera's theorem, the integral is a holomorphic function of $\tau \in \C \setminus (-\infty, 0]$. The proof of this claim requires a few simple estimates.

Observe that for each $\xi_1$, $\xi_2$ such that $i \xi_1, i \xi_2 \in \gamma_f^\ua$, there is $c_1$ (depending on $f$, $\xi_1$ and $\xi_2$) such that
\formula[eq:rogers:xwh:est1]{
 \abs{\frac{1}{i \xi_1 - z} - \frac{1}{i \xi_2 - z}} & = \frac{|\xi_1 - \xi_2|}{|i \xi_1 - z| |i \xi_2 - z|} \le \frac{c_1}{1 + |z|^2}
}
for all $z \in \gamma_f$. Furthermore, for $z \in \gamma_f$ and $\tau \in \C \setminus (-\infty, 0]$,
\formula{
 \max(|\im \tau|, \re \tau) & \le |f(z) + \tau| \le f(z) + |\tau| ,
}
so that
\formula{
 -\frac{1}{\max(|\im \tau|, \re \tau)} & \le \log |f(z) + \tau| \le \log(1 + f(z)) + |\tau| ,
}
It follows that
\formula[eq:rogers:xwh:est2]{
 |\log(f(z) + \tau)| & \le \pi + |\log |f(z) + \tau|| \\
 & \le \pi + \frac{1}{\max(|\im \tau|, \re \tau)} + \log(1 + f(z)) + |\tau| .
}
By condition~\ref{it:rogers:reg:d} in Definition~\ref{def:rogers:reg} and Proposition~\ref{prop:rogers:bound} (with $r = 1$), there is $c_2$ (depending on $f$) such that
\formula[eq:rogers:xwh:est3]{
 \log(1 + f(z)) & \le c_2 (1 + \log(1 + |z|))
}
for $z \in \gamma_f$. By~\eqref{eq:rogers:xwh:est1}, \eqref{eq:rogers:xwh:est2}, \eqref{eq:rogers:xwh:est3} and condition~\ref{it:rogers:reg:e} in Definition~\ref{def:rogers:reg}, Fubini can be used to show that the integral of the right-hand side of~\eqref{eq:rogers:xwh:uu:log} along any closed simple smooth curve contained in $\C \setminus (-\infty, 0]$ is zero. By Morera's theorem, the right-hand side of~\eqref{eq:rogers:xwh:uu:log} is therefore a holomorphic function of $\tau \in \C \setminus (-\infty, 0]$, as claimed.

It follows that $f^\ua(\tau; \xi_1) / f^\ua(\tau; \xi_2)$ extends to a holomorphic function of $\tau \in \C \setminus (-\infty, 0]$, given by~\eqref{eq:rogers:xwh:uu:log}. The first part of formula~\eqref{eq:rogers:xwh:uudd} follows by symmetry of $\gamma_f$, as in~\eqref{eq:rogers:wh:real}. The properties of $f^\da(\tau; \xi_1) / f^\da(\tau; \xi_2)$ and $f^\ua(\tau; \xi_1) f^\da(\tau; \xi_2)$, as well as the third part of~\eqref{eq:rogers:xwh:uudd} and the first part of~\eqref{eq:rogers:xwh:ud}, are proved in a similar manner.

The second part of~\eqref{eq:rogers:xwh:uudd} for $\tau > 0$ is simply an application of Corollary~\ref{cor:rogers:wh:alt}. In order to extend this equality to all $\tau \in \C \setminus (-\infty, 0]$, one only needs to use Morera's theorem to show that the second line of~\eqref{eq:rogers:xwh:uudd} defines a holomorphic function of $\tau \in \C \setminus (-\infty, 0]$. This is done using similar estimates to those used in the proof of Lemma~\ref{lem:rogers:wh:alt}. Indeed, for each $\xi_1, \xi_2 > 0$ there is $c_3$ (depending on $f$, $\xi_1$ and $\xi_2$) such that
\formula{
 \abs{\Arg \expr{\frac{\zeta_f(r) - i \xi_2}{\zeta_f(r) - i \xi_1}}} & \le \abs{\log \expr{\frac{\zeta_f(r) - i \xi_2}{\zeta_f(r) - i \xi_1}}} \le \frac{c_3}{1 + r}
}
for all $r > 0$. Furthermore,
\formula{
 \abs{\frac{\lambda_f'(r)}{\lambda_f(r) + \tau}} & \le \frac{\lambda_f'(r)}{|\tau|}
}
for $r \in (0, 1)$, and, as in~\eqref{eq:rogers:wh:est:aux}, there is $c_4$ (depending on $f$) such that
\formula{
 \abs{\frac{\lambda_f'(r)}{\lambda_f(r) + \tau}} & \le \frac{c_4 |\zeta_f'(r)|}{r}
}
for $r \in (1, \infty)$. Hence, by condition~\ref{it:rogers:reg:e} in Definition~\ref{def:rogers:reg}, the integral in the second line of~\eqref{eq:rogers:xwh:uudd} is absolutely convergent and, by Fubini and Morera's theorem, it defines a holomorphic function of $\tau \in \C \setminus (-\infty, 0]$ (we omit the details). The second equality in~\eqref{eq:rogers:xwh:uudd} follows by uniqueness of holomorphic extension. The fourth part of~\eqref{eq:rogers:xwh:uudd} and the second part of~\eqref{eq:rogers:xwh:ud} are proved in the same way.
\end{proof}

\begin{lemma}
\label{lem:rogers:xwh:cbf}
With the assumptions and notation of Lemma~\ref{lem:rogers:xwh}, if $0 < \xi_1 < \xi_2$, then
\formula{
 \frac{f^\ua(\tau; \xi_1)}{f^\ua(\tau; \xi_2)} && \text{and} && \frac{f^\da(\tau; \xi_1)}{f^\da(\tau; \xi_2)}
}
are complete Bernstein functions of $\tau$. Furthermore, if $\xi_1, \xi_2 > 0$, then
\formula{
 f^\ua(\tau; \xi_1) f^\da(\tau; \xi_2)
}
is a complete Bernstein function of $\tau$. Finally, for all $\xi_1, \xi_2 > 0$, the functions $f^\ua(\tau; \xi_1) / f^\ua(\tau; \xi_2)$, $f^\da(\tau; \xi_1) / f^\da(\tau; \xi_2)$ and $f^\ua(\tau; \xi_1) f^\da(\tau; \xi_2) / \tau$ converge to $1$ as $\tau \to \infty$.
\end{lemma}

\begin{proof}
By Lemma~\ref{lem:rogers:xwh}, $f^\ua(\tau; \xi_1) / f^\ua(\tau; \xi_2)$ extends to a holomorphic function of $\tau \in \C \setminus (-\infty, 0]$ and it is nonnegative for $\tau > 0$. Suppose that $0 < \xi_1 < \xi_2$. Then
\formula[eq:rogersxwh:cbf:aux]{
 -\pi & \le \Arg \expr{\frac{\zeta_f(r) - i \xi_2}{\zeta_f(r) - i \xi_1}} \le 0 ,
}
and hence for $\tau \in \C_\ua$,
\formula{
 0 & \le \frac{1}{\pi} \int_0^\infty \Arg \expr{\frac{\zeta_f(r) - i \xi_2}{\zeta_f(r) - i \xi_1}} \im \expr{\frac{\lambda_f'(r)}{\lambda_f(r) + \tau}} \, \D r \displaybreak[0] \\
 & \le \frac{1}{\pi} \int_0^\infty (-\pi) \im \expr{\frac{\lambda_f'(r)}{\lambda_f(r) + \tau}} \, \D r \displaybreak[0] \\
 & = \lim_{r \searrow 0} \Arg(\lambda_f(r) + \tau) - \lim_{r \nearrow \infty} \Arg(\lambda_f(r) + \tau) \le \Arg \tau .
}
By Lemma~\ref{lem:rogers:xwh}, for $\tau \in \C_\ua$,
\formula{
 \Arg \frac{f^\ua(\tau; \xi_1)}{f^\ua(\tau; \xi_2)} & = \im \expr{\frac{1}{\pi} \int_0^\infty \Arg \expr{\frac{\zeta_f(r) - i \xi_2}{\zeta_f(r) - i \xi_1}} \frac{\lambda_f'(r)}{\lambda_f(r) + \tau} \, \D r} \in [0, \Arg \tau] ,
}
which implies that $f^\ua(\tau; \xi_1) / f^\ua(\tau; \xi_2)$ is a complete Bernstein function of $\tau$.

A similar argument shows that if $0 < \xi_1 < \xi_2$ and $\tau \in \C_\ua$, then
\formula{
 \Arg \frac{f^\da(\tau; \xi_1)}{f^\da(\tau; \xi_2)} & = \im \expr{\frac{1}{\pi} \int_0^\infty \Arg \expr{\frac{\zeta_f(r) + i \xi_1}{\zeta_f(r) + i \xi_2}} \frac{\lambda_f'(r)}{\lambda_f(r) + \tau} \, \D r} \in [0, \Arg \tau] ,
}
and if $\xi_1, \xi_2 > 0$ and $\tau \in \C_\ua$, then
\formula{
 \Arg (f^\ua(\tau; \xi_1) f^\da(\tau; \xi_2)) & = \Arg \tau + \im \expr{\frac{1}{\pi} \int_0^\infty \Arg \expr{\frac{\zeta_f(r) + i \xi_2}{\zeta_f(r) - i \xi_1}} \frac{\lambda_f'(r)}{\lambda_f(r) + \tau} \, \D r}
 \\ & \in [0, \Arg \tau] ,
}
which proves that also $f^\da(\tau; \xi_1) / f^\da(\tau; \xi_2)$ (when $0 < \xi_1 < \xi_2$) and $f^\ua(\tau; \xi_1) f^\da(\tau; \xi_2)$ (when $\xi_1, \xi_2 > 0$) are complete Bernstein functions of $\tau$.

Finally, the last statement follows from Lemma~\ref{lem:rogers:xwh} by~\eqref{eq:rogersxwh:cbf:aux} (and its analogues for $f^\da(\tau; \xi_1) / f^\da(\tau; \xi_2)$ and $f^\ua(\tau; \xi_1) f^\da(\tau; \xi_2)$) and dominated convergence.
\end{proof}

\begin{remark}
\label{rem:rogers:xwh:cbf:bound}
The proof of Lemma~\ref{lem:rogers:xwh:cbf} requires only formulae in Corollary~\ref{cor:rogers:wh:alt}. For balanced Rogers function, the statement can be strengthened. Indeed, by Definition~\ref{def:rogers:reg}, there is $\alpha_\ua \in (0, \pi)$ such that $\Arg \zeta_f(r) \le \alpha - \tfrac{\pi}{2}$ for all $r > 0$. Hence,
\formula{
 -\alpha_\ua & \le \Arg \expr{\frac{\zeta_f(r) - i \xi_2}{\zeta_f(r) - i \xi_1}} \le 0
}
whenever $0 < \xi_1 < \xi_2$ and $r > 0$. The proof of Lemma~\ref{lem:rogers:xwh:cbf} shows that in fact
\formula{
 \Arg \frac{f^\ua(\tau; \xi_1)}{f^\ua(\tau; \xi_2)} & \in [0, \alpha_\ua \Arg \tau]
}
for $\tau \in \C^\ua$. In a similar way there is $\alpha_\da \in (0, \pi)$ such that $\Arg \zeta_f(r) \ge \tfrac{\pi}{2} - \alpha_\da$ for all $r > 0$, and one has
\formula{
 \Arg \frac{f^\da(\tau; \xi_1)}{f^\da(\tau; \xi_2)} & \in [0, \alpha_\da \Arg \tau]
}
if $0 < \xi_1 < \xi_2$ and $\tau \in \C^\ua$. Therefore, the functions
\formula{
 \expr{\frac{f^\ua(\tau; \xi_1)}{f^\ua(\tau; \xi_2)}}^{1 / \alpha_\ua} && \text{and} && \expr{\frac{f^\da(\tau; \xi_1)}{f^\da(\tau; \xi_2)}}^{1 / \alpha_\da}
}
are complete Bernstein functions of $\tau$.
\end{remark}

\begin{lemma}
\label{lem:rogers:xwh:boundary}
With the assumptions and notation of Lemma~\ref{lem:rogers:xwh}, for all admissible $\xi_1, \xi_2$, the limits
\formula[eq:rogers:xwh:boundary:lim]{
 \lim_{t \searrow 0} \frac{f^\ua(-s + i t; \xi_1)}{f^\ua(-s + i t; \xi_2)} , && \lim_{t \searrow 0} \frac{f^\da(-s + i t; \xi_1)}{f^\da(-s + i t; \xi_2)} , && \lim_{t \searrow 0} (f^\ua(-s + i t; \xi_1) f^\da(-s + i t; \xi_2))
}
exist for every $s \in \R$. If $\xi_1, \xi_2 > 0$, then in addition
\formula[eq:rogers:xwh:boundary:weak]{
 \lim_{t \searrow 0} \expr{\im \frac{f^\ua(-s + i t; \xi_1)}{f^\ua(-s + i t; \xi_2)} \, \D s} & = \expr{\lim_{t \searrow 0} \im \frac{f^\ua(-s + i t; \xi_1)}{f^\ua(-s + i t; \xi_2)}} \D s ,
}
with the vague limit of measures in the left-hand side, and similar formulae hold for $f^\da(\tau; \xi_1) / f^\da(\tau; \xi_2)$ and $f^\ua(\tau; \xi_1) f^\da(\tau; \xi_2)$. Furthermore:
\begin{enumerate}[label=\rm (\alph*)]
\item
\label{it:rogers:xwh:a}
For $\zeta \in \gamma_f \cap \C_\ra$ and $\xi_1, \xi_2 > 0$,
\formula{
 \lim_{t \searrow 0} \frac{f^\ua(-f(\zeta) + i t; \xi_1)}{f^\ua(-f(\zeta) + i t; \xi_2)} & = \expr{\frac{f_{[\zeta]}^\ua(\xi_1)}{\xi_1 + i \zeta}}^{-1} \frac{f_{[\zeta]}^\ua(\xi_2)}{\xi_2 + i \zeta} \, , \displaybreak[0] \\
 \lim_{t \searrow 0} \frac{f^\da(-f(\zeta) + i t; \xi_1)}{f^\da(-f(\zeta) + i t; \xi_2)} & = \expr{\frac{f_{[\zeta]}^\da(\xi_1)}{\xi_1 + i \bar{\zeta}}}^{-1} \frac{f_{[\zeta]}^\da(\xi_2)}{\xi_2 + i \bar{\zeta}} \, , 
}
and
\formula{
 \lim_{t \searrow 0} (f^\ua(-f(\zeta) + i t; \xi_1) f^\da(-f(\zeta) + i t; \xi_2)) & = \expr{\frac{f_{[\zeta]}^\ua(\xi_1)}{\xi_1 + i \zeta}}^{-1} \expr{\frac{f_{[\zeta]}^\da(\xi_2)}{\xi_2 + i \bar{\zeta}}}^{-1} .
}
\item
\label{it:rogers:xwh:b}
If $s < \inf \{ f(\zeta) : \zeta \in \gamma_f \}$, then $f(\xi) - s$ is a Rogers function, $f^\ua(\tau; \xi_1) / f^\ua(\tau; \xi_1)$, $f^\da(\tau; \xi_1) / f^\da(\tau; \xi_1)$ and $f^\ua(\tau; \xi_1) f^\da(\tau; \xi_1)$ extend continuously at $\tau = -s$, and for $\xi_1, \xi_2 > 0$, the values at $\tau = -s$ are given by~\eqref{eq:rogers:xwh:uudd} and~\eqref{eq:rogers:xwh:ud}.
\item
\label{it:rogers:xwh:c}
If $s > \sup \{ f(\zeta) : \zeta \in \gamma_f \}$, then again $f^\ua(\tau; \xi_1) / f^\ua(\tau; \xi_1)$, $f^\da(\tau; \xi_1) / f^\da(\tau; \xi_1)$ and $f^\ua(\tau; \xi_1) f^\da(\tau; \xi_1)$ extend continuously at $\tau = -s$. For $\xi_1, \xi_2 > 0$, the values at $\tau = -s$ are given by
\formula{
 \hspace*{3em} \frac{f^\ua(-s; \xi_1)}{f^\ua(-s; \xi_2)} & = \exp \expr{-\frac{1}{\pi} \int_0^\infty \im \expr{\frac{\zeta_f'(r)}{i \xi_1 - \zeta_f(r)} - \frac{\zeta_f'(r)}{i \xi_2 - \zeta_f(r)}} \log (s - \lambda_f(r)) \D r} , \displaybreak[0] \\
 \frac{f^\da(-s; \xi_1)}{f^\da(-s; \xi_2)} & = \exp \expr{-\frac{1}{\pi} \int_0^\infty \im \expr{\frac{\zeta_f'(r)}{i \xi_1 + \zeta_f(r)} - \frac{\zeta_f'(r)}{i \xi_2 + \zeta_f(r)}} \log (s - \lambda_f(r)) \D r} ,
}
and
\formula[eq:rogers:xwh:ud0]{
 \hspace*{3em} & f^\ua(-s; \xi_1) f^\da(-s; \xi_2) \\
 & \hspace*{1em} = -\exp \expr{-\frac{1}{\pi} \int_0^\infty \im \expr{\frac{\zeta_f'(r)}{i \xi_1 - \zeta_f(r)} + \frac{\zeta_f'(r)}{i \xi_2 + \zeta_f(r)}} \log (s - \lambda_f(r)) \D r} ,
}
where $\zeta_f$ is the canonical parametrisation of $\gamma_f \cap \C_\ra$ and $\lambda_f(r) = f(\zeta_f(r))$.
\end{enumerate}
\end{lemma}

\begin{proof}
The proof is divided into two parts. In the first one, the existence of boundary limits is shown, while in the second one, statements~\ref{it:rogers:xwh:a}--\ref{it:rogers:xwh:c} are proved. Recall that for $\zeta \in \gamma_f \cap \C_\ra$ and $\xi \in \gamma_f \setminus \{\zeta, -\bar{\zeta}\}$,
\formula{
 f_{[\zeta]}(\xi) & = \frac{(\xi - \zeta) (\xi + \bar{\zeta})}{f(\xi) - f(\zeta)}
}
defines a Rogers function $f_{[\zeta]}$, and $f_{[\zeta]}(\zeta) = (2 \re \zeta) / f'(\zeta)$.

\emph{Part I.} Suppose that $i \xi_1, i \xi_2 \in \gamma_f^\ua$, $\tau \in \C_\ua$, $\tau = -s + i t$ for $s \in \R$ and $t > 0$. The argument is similar to the proof of Lemma~\ref{lem:rogers:xwh}, but with more careful estimates: by dominated convergence it is shown that the boundary limit can be taken in~\eqref{eq:rogers:xwh:uu:log} (and its analogues for $f^\da(\tau; \xi_1) / f^\da(\tau; \xi_2)$ and $f^\ua(\tau; \xi_1) f^\da(\tau; \xi_2)$) under the integral sign. By~\eqref{eq:rogers:xwh:est1}, there is $c_1$ (depending on $f$, $\xi_1$ and $\xi_2$), such that
\formula[eq:rogers:xwh:est4]{
 \abs{\frac{1}{i \xi_1 - \zeta_f(r)} - \frac{1}{i \xi_2 - \zeta_f(r)}} & \le \frac{c_1}{1 + r^2}
}
for all $r > 0$. If $|\lambda_f(r) - s| \ge 1$, then
\formula{
 0 & \le \log(|\lambda_f(r) - s|^2 + t^2) \le \log(|\lambda_f(r) - s|^2) + t^2;
}
otherwise,
\formula{
 t^2 & \ge \log(1 + t^2) \ge \log(|\lambda_f(r) - s|^2 + t^2) \ge \log(|\lambda_f(r) - s|^2) .
}
It follows that
\formula[eq:rogers:xwh:est5]{
 |\log(\lambda_f(r) + \tau)| & \le \pi + \tfrac{1}{2} |\log(|\lambda_f(r) - s|^2 + t^2)| \le \pi + t^2 + |\log |\lambda_f(r) - s||
}
for all $r > 0$. By Proposition~\ref{prop:rogers:bound} (with $r = 1$) and condition~\ref{it:rogers:reg:d} in Definition~\ref{def:rogers:reg}, there is $c_2 > 1$ (depending on $f$) such that
\formula[eq:rogers:xwh:est5d]{
 |\lambda_f(r) - s| & \le \lambda_f(r) + |s| \le c_2 (1 + r^2) + |s|.
}
For the lower bound on $|\lambda_f(r) - s|$, three cases are considered separately.

If $s \le \lambda_f(0^+)$, then $f(\xi) - s$ is a Rogers function of $\xi$ (by Proposition~\ref{prop:rogers:reg:lambda}). Hence, by Proposition~\ref{prop:rogers:bound} (with $r = 1$) and condition~\ref{it:rogers:reg:d} in Definition~\ref{def:rogers:reg}, there is $c_3 > 1$ (depending on $f$) such that
\formula[eq:rogers:xwh:est5a]{
 |\lambda_f(r) - s| & = |f(\zeta_f(r)) - s| \ge \frac{1}{c_3} \, \frac{r^2}{1 + r^2} \, .
}
In a similar way, if $s \ge \lambda_f(\infty^-)$, then $f$ is a bounded Rogers function (by Proposition~\ref{prop:rogers:reg:lambda}) and $s - f(1/\xi)$ is a Rogers function of $\xi$ (by Proposition~\ref{prop:rogers:bounded:prop}). Hence, by Proposition~\ref{prop:rogers:bound} (with $r = 1$) and condition~\ref{it:rogers:reg:d} in Definition~\ref{def:rogers:reg}, there is $c_4 > 1$ (depending on $f$) such that
\formula[eq:rogers:xwh:est5b]{
 |\lambda_f(r) - s| & = |s - f(\zeta_f(r))| \ge \frac{1}{c_4} \, \frac{(1/r)^2}{1 + (1/r)^2} = \frac{1}{c_4} \, \frac{1}{1 + r^2} \, .
}
Finally, consider $s = \lambda_f(r_0)$ for some $r_0 > 0$, and let $z = \zeta_f(r)$, $z_0 = \zeta_f(r_0)$. Then
\formula{
 |\lambda_f(r) - s| & = |f(z) - f(z_0)| = \frac{|z - z_0| \, |z + \bar{z}_0|}{|f_{[z_0]}(z)|} \ge \frac{|r - r_0|^2}{|f_{[z_0]}(z)|} \, .
}
By Proposition~\ref{prop:rogers:bound} (applied twice, to the Rogers function $f_{[z_0]}$) and condition~\ref{it:rogers:reg:d} in Definition~\ref{def:rogers:reg}, there are $c_5, c_6 > 1$ (depending on $f$) such that
\formula{
 |f_{[z_0]}(z)| & \le c_5 \, \frac{r_0^2 + r^2}{r_0^2} \, |f_{[z_0]}(|z_0|)| \le c_6 \, \frac{r_0^2 + r^2}{r_0^2} \, |f_{[z_0]}(z_0)| \le c_6 \, \frac{r_0^2 + r^2}{r_0^2} \, \frac{2 r_0}{|f'(z_0)|} \, .
}
Hence,
\formula[eq:rogers:xwh:est5c]{
 |\lambda_f(r) - s| & \ge \frac{1}{c_6} \, |z_0 f'(z_0)| \, \frac{|r - r_0|^2}{2 (r_0^2 + r^2)} \, .
}
Recall that as in~\eqref{eq:rogers:xwh:uu:log} (with $\tau = -s + i t$),
\formula{
 \log \frac{f^\ua(\tau; \xi_1)}{f^\ua(\tau; \xi_2)} & = -\frac{1}{2 \pi i} \int_{\gamma_f} \expr{\frac{1}{i \xi_1 - z} - \frac{1}{i \xi_2 - z}} \log (f(z) + \tau) \D z \displaybreak[0] \\
 & \hspace*{-5.3em} = -\frac{1}{2 \pi i} \int_0^\infty \expr{\frac{\zeta_f'(r)}{i \xi_1 - \zeta_f(r)} - \frac{\zeta_f'(r)}{i \xi_2 - \zeta_f(r)} + \frac{\overline{\zeta_f'(r)}}{i \xi_1 + \overline{\zeta_f(r)}} - \frac{\overline{\zeta_f'(r)}}{i \xi_2 + \overline{\zeta_f(r)}}} \log (\lambda_f(r) + \tau) \D r .
}
Denote the integrand in the right-hand side by $I(r)$. By~\eqref{eq:rogers:xwh:est4} and \eqref{eq:rogers:xwh:est5}, it follows that
\formula{
 |I(r)| & \le \frac{2 c_1 |\zeta_f'(r)|}{1 + r^2} \, (\pi + t^2 + |\log |\lambda_f(r) - s||) .
}
By~\eqref{eq:rogers:xwh:est5d} and~\eqref{eq:rogers:xwh:est5a}, there is $c_7$ (depending on $f$, $\xi_1$, $\xi_2$) such that
\formula{
 |I(r)| & \le \frac{2 c_1 |\zeta_f'(r)|}{1 + r^2} \expr{\pi + t^2 + \log(c_2(1 + r^2) + |s|) + \log \frac{1 + r^2}{c_3 r^2}} \displaybreak[0] \\
 & \le \frac{c_7 |\zeta_f'(r)| \log(r + 1/r)}{1 + r^2}
}
for all $s \le \lambda_f(0^+)$ and $t \in (0, 1]$. In a similar way, by~\eqref{eq:rogers:xwh:est5d} and~\eqref{eq:rogers:xwh:est5b}, there is $c_8$ (depending on $f$, $\xi_1$, $\xi_2$) such that
\formula{
 |I(r)| & \le \frac{c_8 |\zeta_f'(r)| \log(r + 1/r)}{1 + r^2}
}
for all $s \ge \lambda_f(\infty^-)$ and $t \in (0, 1]$. Finally, if $s = \lambda_f(r_0)$ for some $r_0 > 0$, and if $z_0 = \zeta_f(r_0)$, then by~\eqref{eq:rogers:xwh:est5d} and~\eqref{eq:rogers:xwh:est5c}, there is $c_9$ (depending on $f$, $\xi_1$, $\xi_2$) such that
\formula{
 |I(r)| & \le \frac{2 c_1 |\zeta_f'(r)|}{1 + r^2} \expr{\pi + 1 + \log(c_2(1 + r^2) + |s|) + \max\expr{\log \frac{2 c_6 (r_0^2 + r^2)}{|z_0 f'(z_0)| \, |r - r_0|^2}, 0}} \displaybreak[0] \\
 & \le \frac{c_9 |\zeta_f'(r)|}{1 + r^2} \log\expr{r + \frac{1}{r} + \frac{1}{|r - r_0|} + r_0 + \frac{1}{r_0} + \frac{1}{|z_0 f'(z_0)|}}
}
for $t \in (0, 1]$. By condition~\ref{it:rogers:reg:e} in Definition~\ref{def:rogers:reg}, the right-hand sides of the above three estimates are integrable in $r \in (0, \infty)$. In addition, these estimates prove that for $t \in (0, 1]$, $\int_0^\infty |I(r)| \D r \le c_{10}(s)$, where $c_{10}$ depends on $f$, $\xi_1$ and $\xi_2$, $c_{10}(s)$ is constant on $(-\infty, \lambda_f(0^+)]$ and on $[\lambda_f(\infty^-), \infty)$, and it is continuous on $(\lambda_f(0^+), \lambda_f(\infty^-))$. By dominated convergence, for every $s \in \R$, the finite limit
\formula{
 \lim_{t \searrow 0} \log \frac{f^\ua(-s + i t; \xi_1)}{f^\ua(-s + i t; \xi_2)} & = -\frac{1}{2 \pi i} \int_{\gamma_f} \expr{\frac{1}{i \xi_1 - z} - \frac{1}{i \xi_2 - z}} \log_\ua (f(z) - s) \D z \displaybreak[0] \\
 & \hspace*{-9.4em} = -\frac{1}{2 \pi i} \int_0^\infty \expr{\frac{\zeta_f'(r)}{i \xi_1 - \zeta_f(r)} - \frac{\zeta_f'(r)}{i \xi_2 - \zeta_f(r)} + \frac{\overline{\zeta_f'(r)}}{i \xi_1 + \overline{\zeta_f(r)}} - \frac{\overline{\zeta_f'(r)}}{i \xi_2 + \overline{\zeta_f(r)}}} \log_\ua (\lambda_f(r) - s) \D r
}
exists, and it is bounded by $c_{10}(s) / (2 \pi)$; here $\log_\ua$ denotes the continuous version of the complex logarithm in $\C \setminus (-i \infty, 0]$. The existence of the first limit in~\eqref{eq:rogers:xwh:boundary:lim} is proved. The other parts of~\eqref{eq:rogers:xwh:boundary:lim} are proved in a similar manner.

For future reference, note that if $\xi_1, \xi_2 > 0$, then the formula for the boundary limit simplifies to
\formula[eq:rogers:xwh:limit]{
 & \lim_{t \searrow 0} \frac{f^\ua(-s + i t; \xi_1)}{f^\ua(-s + i t; \xi_2)} \\
 & \hspace*{4em} = \exp \expr{-\frac{1}{\pi} \int_0^\infty \im \expr{\frac{\zeta_f'(r)}{i \xi_1 - \zeta_f(r)} - \frac{\zeta_f'(r)}{i \xi_2 - \zeta_f(r)}} \log_\ua (\lambda_f(r) - s) \D r} ,
}
and similar formulae can be given for $f^\da(\tau; \xi_1) / f^\da(\tau; \xi_2)$ and $f^\ua(\tau; \xi_1) f^\da(\tau; \xi_2)$.

When $0 < \xi_1 < \xi_2$, then, by Lemma~\ref{lem:rogers:xwh:cbf}, $g(\tau) = f^\ua(\tau; \xi_1) / f^\ua(\tau; \xi_2)$ is a complete Bernstein function of $\tau$. Recall that $|\log g(-s + i t)| \le c_{10}(s) / (2 \pi)$ for $s \in \R$ and $t \in (0, 1]$. Hence, $0 \le \im g(-s + i t) \le \exp(c_{10}(s) / (2 \pi))$ for all $t \in [0, 1]$ and $s \in \R$, and therefore the vague limit $m(\D s) = \lim_{t \searrow 0} (\im g(-s + i t)) \D s$, except possibly at $s = \lambda_f(0^+)$ and $s = \lambda_f(\infty^-)$, is absolutely continuous, with density function bounded by $\exp(c_{10}(s) / (2 \pi))$. If $m(\D s)$ had an atom at $s = \lambda_f(0^+)$ or at $s = \lambda_f(\infty^-)$, then, by the representation of nonnegative harmonic functions (Theorem~\ref{th:harmonic:positive}), the limit $\lim_{t \searrow 0} (\im g(-s + i t))$ would be infinite at this value of $s$, a contradiction. Therefore, $m$ is indeed absolutely continuous, and hence~\eqref{eq:rogers:xwh:boundary:weak} follows from Theorem~\ref{th:harmonic:positive}. The same argument works for $f^\da(\tau; \xi_1) / f^\da(\tau; \xi_2)$ and $f^\ua(\tau; \xi_1) f^\da(\tau; \xi_2)$, and the first part of the proof is complete.

\begin{figure}
\centering
\def\svgwidth{0.7\textwidth}
\begingroup%
  \makeatletter%
  \ifx\svgwidth\undefined%
    \setlength{\unitlength}{213.93bp}%
    \ifx\svgscale\undefined%
      \relax%
    \else%
      \setlength{\unitlength}{\unitlength * \real{\svgscale}}%
    \fi%
  \else%
    \setlength{\unitlength}{\svgwidth}%
  \fi%
  \global\let\svgwidth\undefined%
  \global\let\svgscale\undefined%
  \makeatother%
  \begin{picture}(1,0.65040085)%
    \put(0,0){\includegraphics[width=\unitlength]{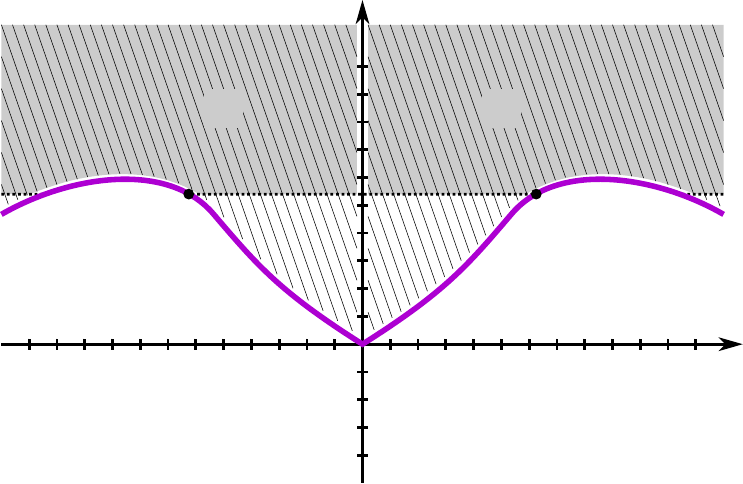}}%
    \put(0.56280092,0.16832727){\makebox(0,0)[cb]{\smash{\raisebox{-\height}{$ 1$}}}}%
    \put(0.63759174,0.16832727){\makebox(0,0)[cb]{\smash{\raisebox{-\height}{$ 2$}}}}%
    \put(0.71238256,0.16832727){\makebox(0,0)[cb]{\smash{\raisebox{-\height}{$ 3$}}}}%
    \put(0.78717337,0.16832727){\makebox(0,0)[cb]{\smash{\raisebox{-\height}{$ 4$}}}}%
    \put(0.86196419,0.16832727){\makebox(0,0)[cb]{\smash{\raisebox{-\height}{$ 5$}}}}%
    \put(0.93675501,0.16832727){\makebox(0,0)[cb]{\smash{\raisebox{-\height}{$ 6$}}}}%
    \put(0.03926518,0.16832727){\makebox(0,0)[cb]{\smash{\raisebox{-\height}{$-6$\phantom{$-$}}}}}%
    \put(0.11405600,0.16832727){\makebox(0,0)[cb]{\smash{\raisebox{-\height}{$-5$\phantom{$-$}}}}}%
    \put(0.18884682,0.16832727){\makebox(0,0)[cb]{\smash{\raisebox{-\height}{$-4$\phantom{$-$}}}}}%
    \put(0.26363764,0.16832727){\makebox(0,0)[cb]{\smash{\raisebox{-\height}{$-3$\phantom{$-$}}}}}%
    \put(0.33842846,0.16832727){\makebox(0,0)[cb]{\smash{\raisebox{-\height}{$-2$\phantom{$-$}}}}}%
    \put(0.41321928,0.16832727){\makebox(0,0)[cb]{\smash{\raisebox{-\height}{$-1$\phantom{$-$}}}}}%
    \put(0.46931239,0.26181580){\makebox(0,0)[rb]{\smash{\raisebox{-0.5\height}{$ 1$\!}}}}%
    \put(0.46931239,0.33660662){\makebox(0,0)[rb]{\smash{\raisebox{-0.5\height}{$ 2$\!}}}}%
    \put(0.46931239,0.41139742){\makebox(0,0)[rb]{\smash{\raisebox{-0.5\height}{$ 3$\!}}}}%
    \put(0.46931239,0.48618824){\makebox(0,0)[rb]{\smash{\raisebox{-0.5\height}{$ 4$\!}}}}%
    \put(0.46931239,0.56097907){\makebox(0,0)[rb]{\smash{\raisebox{-0.5\height}{$ 5$\!}}}}%
    \put(0.46931239,0.03744334){\makebox(0,0)[rb]{\smash{\raisebox{-0.5\height}{$-2$\!}}}}%
    \put(0.46931239,0.11223416){\makebox(0,0)[rb]{\smash{\raisebox{-0.5\height}{$-1$\!}}}}%
    \put(0.95545272,0.35904386){\makebox(0,0)[ct]{\smash{\raisebox{-\height}{$\gamma_f$}}}}%
    \put(0.72734072,0.38148111){\makebox(0,0)[lt]{\smash{\raisebox{-\height}{$\zeta$}}}}%
    \put(0.24493993,0.38148111){\makebox(0,0)[rt]{\smash{\raisebox{-\height}{$-\bar{\zeta}$}}}}%
    \put(0.30103305,0.50488596){\makebox(0,0)[cb]{\smash{\raisebox{-0.35\height}{(I)}}}}%
    \put(0.67498715,0.50488596){\makebox(0,0)[cb]{\smash{\raisebox{-0.35\height}{(II)}}}}%
  \end{picture}%
\endgroup%
\caption{Notation for the proof of Lemma~\ref{lem:rogers:xwh:boundary}. For $\zeta \in \gamma_f \cap \C_\ra$, the function $\log((\xi - \zeta) (\xi + \bar{\zeta}))$ is holomorphic in the complement of the segment $[-\bar{\zeta}, \zeta]$. The function $h_\zeta(\xi)$ is defined to be holomorphic in the hatched area (with the branch cut along imaginary axis), equal to $\log((\xi - \zeta) (\xi + \bar{\zeta}))$ in gray area (I), and equal to $\log((\xi - \zeta) (\xi + \bar{\zeta})) + i \pi$ in gray area (II).}
\label{fig:rogers:xwh}
\end{figure}
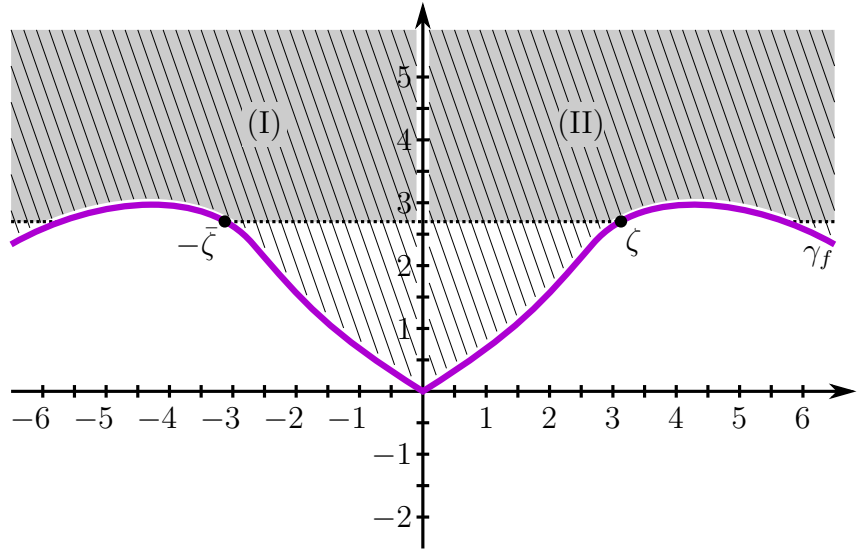

\emph{Part II.} We first prove statement~\ref{it:rogers:xwh:a}. By~\eqref{eq:rogers:wh:uu},
\formula{
 \frac{f_{[\zeta]}^\ua(\xi_1)}{f_{[\zeta]}^\ua(\xi_2)} & = \exp \expr{-\frac{1}{2 \pi i} \int_{\gamma_f} \expr{\frac{1}{i \xi_1 - z} - \frac{1}{i \xi_2 - z}} \log \frac{(z - \zeta) (z + \bar{\zeta})}{f(z) - f(\zeta)} \, \D z} .
}
Together with the first part of the proof, this shows that if $s = f(\zeta)$, then
\formula[eq:rogers:xwh:1]{
 \frac{f_{[\zeta]}^\ua(\xi_1)}{f_{[\zeta]}^\ua(\xi_2)} \, \lim_{t \searrow 0} \frac{f^\ua(-s + i t; \xi_1)}{f^\ua(-s + i t; \xi_2)} & = \exp \expr{-\frac{1}{2 \pi i} \int_{\gamma_f} \expr{\frac{1}{i \xi_1 - z} - \frac{1}{i \xi_2 - z}} h_\zeta(z) \D z} ,
}
where $h_\zeta(\xi)$ is a version of the complex logarithm of $(\xi - \zeta) (\xi + \bar{\zeta})$, which is continuous in $\xi$ on $\overline{\gamma}_f^\ua \setminus (i \R \cup \{\zeta, -\bar{\zeta}\})$ (here $\overline{\gamma}_f^\ua$ is the closure of $\gamma_f^\ua$), and such that $h_\zeta(\xi) = \log((\xi - \zeta) (\xi + \bar{\zeta}))$ when $\xi \in \gamma_f^\ua \cap \C_\ra$ and $\im \xi > \im \zeta$, and $h_\zeta(\xi) = \log((\xi - \zeta) (\xi + \bar{\zeta})) + i \pi$ when $\xi \in \gamma_f^\da \cap \C_\la$ and $\im \xi > \im \zeta$ and $\xi \in \C_\la$ (see Figure~\ref{fig:rogers:xwh}). In a similar manner,
\formula[eq:rogers:xwh:2]{
 \frac{f_{[\zeta]}^\da(\xi_1)}{f_{[\zeta]}^\da(\xi_2)} \, \lim_{t \searrow 0} \frac{f^\da(-s + i t; \xi_1)}{f^\da(-s + i t; \xi_2)} & = \exp \expr{-\frac{1}{2 \pi i} \int_{\gamma_f} \expr{\frac{1}{i \xi_1 + z} - \frac{1}{i \xi_2 + z}} h_\zeta(z) \D z}
}
for $\xi_1$, $\xi_2$ such that $-i \xi_1, -i \xi_2 \in \gamma_f^\da$, and
\formula[eq:rogers:xwh:3]{
 & f_{[\zeta]}^\ua(\xi_1) f_{[\zeta]}^\da(\xi_2) \lim_{t \searrow 0} (f^\ua(-s + i t; \xi_1) f^\da(-s + i t; \xi_2)) \\
 & \qquad\qquad = \exp \expr{-\frac{1}{2 \pi i} \int_{\gamma_f} \expr{\frac{1}{i \xi_1 - z} + \frac{1}{i \xi_2 + z}} h_\zeta(z) \D z}
}
for $\xi_1$, $\xi_2$ such that $i \xi_1 \in \gamma_f^\ua$ and $-i \xi_2 \in \gamma_f^\da$.

Recall that $h_\zeta(\xi)$ is a version of the logarithm of $(\xi - \zeta) (\xi + \bar{\zeta})$. Suppose that $i \xi_1, i \xi_2 \in \gamma_f^\ua$. By an appropriate use of the residue theorem and a limit procedure (we omit the details),
\formula{
 \exp \expr{ -\frac{1}{2 \pi i} \int_{\smash{\gamma_f}} \expr{\frac{1}{i \xi_1 - z} - \frac{1}{i \xi_2 - z}} \, h_\zeta(\xi) \D z } & = \frac{\xi_1 + i \zeta}{\xi_2 + i \zeta} \, .
}
In a similar manner,
\formula{
 \exp \expr{-\frac{1}{2 \pi i} \int_{\gamma_f} \expr{\frac{1}{i \xi_1 + z} - \frac{1}{i \xi_2 + z}} \, h_\zeta(\xi) \D z} & = \frac{\xi_1 + i \bar{\zeta}}{\xi_2 + i \bar{\zeta}}
}
for all $\xi_1$, $\xi_2$ such that $-i \xi_1, -i \xi_2 \in \gamma_f^\da$. Finally,
\formula{
 \exp \expr{-\frac{1}{2 \pi i} \int_{\gamma_f} \expr{\frac{1}{i \xi_1 - z} + \frac{1}{i \xi_2 + z}} \, h_\zeta(\xi) \D z} & = (\xi_1 + i \zeta) (\xi_2 + i \bar{\zeta})
}
for all $\xi_1$, $\xi_2$ such that $i \xi_1 \in \gamma_f^\ua$ and $-i \xi_2 \in \gamma_f^\da$. Part~\ref{it:rogers:xwh:a} follows.

Part~\ref{it:rogers:xwh:b} follows immediately from Lemma~\ref{lem:rogers:xwh}. Finally, part~\ref{it:rogers:xwh:c} is proved using~\eqref{eq:rogers:xwh:limit} and its analogues for the other two expressions: in this case one has $\log^\ua(\lambda_f(r) - s) = \log(s - \lambda_f(r)) + i \pi$ for all $r > 0$, and the result follows by the residue theorem (we omit the details; note the minus sign in front of the exponential in~\eqref{eq:rogers:xwh:ud0}).
\end{proof}

\begin{remark}
\label{rem:rogers:duality}
If $f$ is a bounded balanced Rogers function, $s \ge \lambda_f(\infty^-)$, then $g(\xi) = 1 / (s - f(\xi))$ and $h(\xi) = 1 / g(1 / \xi) = s - f(1 / \xi)$, then, by Lemma~\ref{lem:rogers:xwh:boundary}\ref{it:rogers:xwh:c}, Proposition~\ref{prop:rogers:wh:prop} and a short calculation,
\formula{
 \frac{f^\ua(-s; \xi_1)}{f^\ua(-s; \xi_2)} & = \frac{g^\ua(\xi_2)}{g^\ua(\xi_1)} = \frac{h^\da(1 / \xi_1)}{h^\da(1 / \xi_2)} \, , &
 \frac{f^\da(-s; \xi_1)}{f^\da(-s; \xi_2)} & = \frac{g^\da(\xi_2)}{g^\da(\xi_1)} = \frac{h^\ua(1 / \xi_1)}{h^\ua(1 / \xi_2)} \, ,
}
and
\formula{
 f^\ua(-s; \xi_1) f^\da(-s; \xi_2) & = -\frac{1}{g^\ua(\xi_1) g^\da(\xi_2)} = -h^\da(1 / \xi_1) h^\ua(1 / \xi_2)
}
for all admissible $\xi_1, \xi_2$.
\end{remark}

\subsection{Extension to nearly balanced Rogers functions}

With some effort, Lemma~\ref{lem:rogers:xwh:cbf} can be extended to nearly balanced Rogers functions. Below we only sketch the argument, leaving some details to the reader.

\begin{lemma}
\label{lem:rogers:xwh:translation}
Let $f$ be a nonconstant Rogers function, $r_0 > 0$, $|\gamma_f| \sub (r_0, \infty)$, $\zeta_0 = \zeta_f(r_0)$ (where $\zeta_f$ is the canonical parametrisation of $\gamma_f \cap \C_\ra$), $g(\xi) = f(\xi - \zeta_0)$ as in Lemma~\ref{lem:rogers:translation} and $h(\xi) = g(\xi) - \lambda_f(r_0^+)$ (where $\lambda_f(r) = f(\zeta_f(r))$). If the assertion of Lemma~\ref{lem:rogers:xwh:cbf} holds for $h$, then it also holds for $f$.
\end{lemma}

\begin{proof}
The argument, although relatively simple, is quite lengthy. Recall that either $\zeta_0 = i r_0$ or $\zeta_0 = -i r_0$. Suppose first that $\zeta_0 = i r_0$. By~\eqref{eq:rogers:wh:translation},
\formula{
 \frac{f^\da(\tau; \xi_1)}{f^\da(\tau; \xi_2)} & = \frac{g^\da(\tau; \xi_1 + r_0)}{g^\da(\tau; \xi_2 + r_0)} = \frac{h^\da(\tau + \lambda_f(r_0^+); \xi_1 + r_0)}{h^\da(\tau + \lambda_f(r_0^+); \xi_2 + r_0)} \, .
}
If $0 < \xi_1 < \xi_2$, then, by the assertion of Lemma~\ref{lem:rogers:xwh:cbf} for $h$, $f^\da(\tau; \xi_1) / f^\da(\tau; \xi_2)$ is a complete Bernstein function of $\tau$, as desired. In a similar manner, by~\eqref{eq:rogers:wh:translation},
\formula{
 f^\ua(\tau; \xi_1) f^\da(\tau; \xi_2) & = h^\ua(\tau + \lambda_f(r_0^+); \xi_1 - r_0) h^\da(\tau + \lambda_f(r_0^+); \xi_2 + r_0) .
}
If $\xi_1 > r_0$ and $\xi_2 > 0$, then, by the assertion of Lemma~\ref{lem:rogers:xwh:cbf} for $h$, $f^\ua(\tau; \xi_1) f^\da(\tau; \xi_2)$ is a complete Bernstein function of $\tau$. If $0 < \xi_1 < r_0$ and $\xi_2 > 0$, then, by the definition of the Wiener--Hopf factors and the identity $h(i \xi_1 - i r_0) + \lambda_f(r_0^+) = f(i \xi_1) = \lambda_f(\xi_1)$,
\formula{
 f^\ua(\tau; \xi_1) f^\da(\tau; \xi_2) & = \frac{h(i \xi_1 - i r_0) + \tau + \lambda_f(r_0^+)}{h^\da(\tau + \lambda_f(r_0^+); r_0 - \xi_1)} \, h^\da(\tau + \lambda_f(r_0^+); \xi_2 + r_0) \displaybreak[0] \\
 & = (\tau + \lambda_f(\xi_1)) \, \frac{h^\da(\tau + \lambda_f(r_0^+); \xi_2 + r_0)}{h^\da(\tau + \lambda_f(r_0^+); r_0 - \xi_1)} \, .
}
By the assertion of Lemma~\ref{lem:rogers:xwh:cbf} for $h$, the function $h^\da(\tau; r_0 - \xi_1) / h^\da(\tau; \xi_2 + r_0)$ is a complete Bernstein function of $\tau$. By substituting $\tau + \lambda_f(r_0^+) - \lambda_f(\xi_1)$ for $\tau$ (note that $\lambda_f(r_0^+) - \lambda_f(\xi_1) > 0$), so are
\formula{
 \frac{h^\da(\tau + \lambda_f(r_0^+) - \lambda_f(\xi_1); r_0 - \xi_1)}{h^\da(\tau + \lambda_f(r_0^+) - \lambda_f(\xi_1); \xi_2 + r_0)} && \text{and} && \tau \, \frac{h^\da(\tau + \lambda_f(r_0^+) - \lambda_f(\xi_1); \xi_2 + r_0)}{h^\da(\tau + \lambda_f(r_0^+) - \lambda_f(\xi_1); r_0 - \xi_1)} \, .
}
Substituting $\tau + \lambda_f(\xi_1)$ for $\tau$ in the right-hand side expression proves that $f^\ua(\tau; \xi_1) f^\da(\tau; \xi_2)$ is a complete Bernstein function of $\tau$. Finally, if $\xi_1 = r_0$ and $\xi_2 > 0$, then, by continuity of the Wiener--Hopf factors and Proposition~\ref{prop:cbf:limit},
\formula{
 f^\ua(\tau; \xi_1) f^\da(\tau; \xi_2) & = \lim_{\eps \searrow 0^+} f^\ua(\tau; \xi_1 + \eps) f^\da(\tau; \xi_2)
}
is a complete Bernstein function of $\tau$ as a pointwise limit of complete Bernstein functions. Therefore, $f^\ua(\tau; \xi_1) f^\da(\tau; \xi_2)$ is a complete Bernstein function of $\tau$ whenever $\xi_1, \xi_2 > 0$.

The remaining case is very similar. Once again by by~\eqref{eq:rogers:wh:translation},
\formula{
 \frac{f^\ua(\tau; \xi_1)}{f^\ua(\tau; \xi_2)} & = \frac{g^\ua(\tau; \xi_1 - r_0)}{g^\ua(\tau; \xi_2 - r_0)} = \frac{h^\ua(\tau + \lambda_f(r_0^+); \xi_1 - r_0)}{h^\ua(\tau + \lambda_f(r_0^+); \xi_2 - r_0)} \, .
}
As before, if $r_0 < \xi_1 < \xi_2$, then $f^\ua(\tau; \xi_1) / f^\ua(\tau; \xi_2)$ is a complete Bernstein function of $\tau$ directly by the assertion of Lemma~\ref{lem:rogers:xwh:cbf} for $h$. If $\xi_1 < r_0 < \xi_2$, then
\formula{
 \frac{f^\ua(\tau; \xi_1)}{f^\ua(\tau; \xi_2)} & = \frac{\tau + \lambda_f(\xi_1)}{h^\ua(\tau + \lambda_f(r_0^+); \xi_2 - r_0) h^\da(\tau + \lambda_f(r_0^+); r_0 - \xi_1)} \, .
}
By the assertion of Lemma~\ref{lem:rogers:xwh:cbf} for $h$, $h^\ua(\tau; \xi_2 - r_0) h^\da(\tau; r_0 - \xi_1)$ is a complete Bernstein function of $\tau$. Substituting as in the former part of the proof, we obtain that
\formula{
 \frac{\tau}{h^\ua(\tau + \lambda_f(r_0^+) - \lambda_f(\xi_1); \xi_2 - r_0) h^\da(\tau + \lambda_f(r_0^+) - \lambda_f(\xi_1); r_0 - \xi_1)} && \text{and} && \frac{f^\ua(\tau; \xi_1)}{f^\ua(\tau; \xi_2)}
}
are complete Bernstein functions of $\tau$. If $\xi_1 < \xi_2 < r_0$, then, in a similar manner,
\formula{
 \frac{f^\ua(\tau; \xi_1)}{f^\ua(\tau; \xi_2)} & = \frac{\tau + \lambda_f(\xi_1)}{\tau + \lambda_f(\xi_2)} \, \frac{h^\da(\tau + \lambda_f(r_0^+); r_0 - \xi_2)}{h^\da(\tau + \lambda_f(r_0^+); r_0 - \xi_1)} \, .
}
By the assertion of Lemma~\ref{lem:rogers:xwh:cbf} for $h$, $h^\da(\tau; r_0 - \xi_2) / h^\da(\tau; r_0 - \xi_1)$ is a complete Bernstein function of $\tau$. Substituting $\tau + \lambda_f(r_0^+) - \lambda_f(\xi_2)$ for $\tau$, then $\tau + \lambda_f(\xi_2) - \lambda_f(\xi_1)$ for $\tau$, and finally $\tau + \lambda_f(\xi_1)$ for $\tau$, we obtain that
\formula{
\begin{gathered}
 \tau \, \frac{h^\da(\tau + \lambda_f(r_0^+) - \lambda_f(\xi_2); r_0 - \xi_1)}{h^\da(\tau + \lambda_f(r_0^+) - \lambda_f(\xi_2); r_0 - \xi_2)} \, , \\
 \frac{\tau}{\tau + \lambda_f(\xi_2) - \lambda_f(\xi_1)} \, \frac{h^\da(\tau + \lambda_f(r_0^+) - \lambda_f(\xi_1); r_0 - \xi_2)}{h^\da(\tau + \lambda_f(r_0^+) - \lambda_f(\xi_1); r_0 - \xi_1)} \, ,
\end{gathered}
}
and $f^\ua(\tau; \xi_1) / f^\da(\tau; \xi_2)$ are all complete Bernstein functions of $\tau$. The remaining cases $r_0 = \xi_1 < \xi_2$ and $0 < \xi_1 < \xi_2 = r_0$ are resolved by continuity of the Wiener--Hopf factors and Proposition~\ref{prop:cbf:limit}, as in the former part of the proof.

The proof is complete if $\zeta_0 = i r_0$. In the remaining case $\zeta_0 = -i r_0$, the argument is very similar.
\end{proof}

\begin{lemma}
\label{lem:rogers:xwh:nearly}
Lemma~\ref{lem:rogers:xwh:cbf} extends to nearly balanced Rogers functions.
\end{lemma}

\begin{proof}
In this proof, the notation of Lemmas~\ref{lem:rogers:translation} and~\ref{lem:rogers:mobius} are used. In particular, $g(\xi) = f(u(\xi))$, or $f(\xi) = g(u^{-1}(\xi))$, where $u$ is the M\"obius transformation defined in~\eqref{eq:rogers:wh:mobius}, and $g$ is a balanced Rogers function. As in the proof of Lemma~\ref{lem:rogers:mobius}, one has $u(\xi) = u_0(u_\infty(\xi))$, where $u_0(\xi) = \xi + \zeta_0$ and $u_\infty(\xi) = (1 / \xi + 1 / (\zeta_\infty - \zeta_0))^{-1}$.

The argument is somewhat similar to the proof of Lemma~\ref{lem:rogers:xwh:translation}. It suffices to show that the assertion of Lemma~\ref{lem:rogers:xwh:cbf} holds for $f_\infty(\xi) - \lambda_f(r_0^+)$, where $f_\infty(\xi) = g(u_\infty^{-1}(\xi)) = f(u_0(\xi))$. Indeed, then the result for $f(\xi) = f_\infty(u_0^{-1}(\xi))$ follows by Lemma~\ref{lem:rogers:xwh:translation}. By substituting $g(\xi) - \lambda_f(r_0^+)$ for $g$ and $f_\infty(\xi) - \lambda_f(r_0^+)$ for $f$, one sees that it suffices to consider the case $r_0 = 0$ and $\lambda_f(0^+) = 0$.

We therefore assume that $f(\xi) = g((1 / \xi - 1 / \zeta_\infty)^{-1})$ for a balanced Rogers function $g$. If $r_\infty = \infty$, then $f(\xi) = g(\xi)$ and the result reduces to Lemma~\ref{lem:rogers:xwh:cbf}. Otherwise, either $\zeta_\infty = i r_\infty$ or $\zeta_\infty = -i r_\infty$. Suppose that $\zeta_\infty = i r_\infty$. By~\eqref{eq:rogers:wh:mobius},
\formula{
 \frac{f^\da(\tau; \xi_1)}{f^\da(\tau; \xi_2)} & = \frac{g^\da(\tau; (1 / \xi_1 + 1 / r_0)^{-1})}{g^\da(\tau; (1 / \xi_2 + 1 / r_0)^{-1})} \, ,
}
which is a complete Bernstein function of $\tau$ if $0 < \xi_1 < \xi_2$ by Lemma~\ref{lem:rogers:xwh:cbf}. In a similar manner, by~\eqref{eq:rogers:wh:mobius},
\formula{
 f^\ua(\tau; \xi_1) f^\da(\tau; \xi_2) & = g^\ua(\tau; (1 / \xi_1 - 1 / r_0)^{-1}) g^\da(\tau; (1 / \xi_2 + 1 / r_0)^{-1})
}
is a complete Bernstein function of $\tau$ if $0 < \xi_1 < r_\infty$ and $\xi_2 > 0$. If $\xi_1 > r_\infty$ and $\xi_2 > 0$, then, as in the proof of Lemma~\ref{lem:rogers:xwh:translation},
\formula{
 f^\ua(\tau; \xi_1) f^\da(\tau; \xi_2) & = \frac{g(i \xi_1 - i r_0) + \tau}{g^\da(\tau; (1 / r_0 - 1 / \xi_1)^{-1})} \, g^\da(\tau; (1 / \xi_2 + 1 / r_0)^{-1}) \displaybreak[0] \\
 & = (\tau + \lambda_f(\xi_1)) \, \frac{g^\da(\tau; (1 / \xi_2 + 1 / r_0)^{-1})}{g^\da(\tau; (1 / r_0 - 1 / \xi_1)^{-1})} \, .
}
By Lemma~\ref{lem:rogers:xwh:cbf}, the function $g^\da(\tau; (1 / \xi_2 + 1 / r_0)^{-1}) / g^\da(\tau; (1 / r_0 - 1 / \xi_1)^{-1})$ is a complete Bernstein function of $\tau$, and by Lemma~\ref{lem:rogers:xwh:boundary}\ref{it:rogers:xwh:c}, it extends to a holomorphic function of $\tau \in \C \setminus [-\lambda_f(r_\infty^-), 0]$, with positive values on $(-\infty, -\lambda_f(r_\infty^-))$. By Proposition~\ref{prop:cbf:quot}, $f^\ua(\tau; \xi_1) f^\da(\tau; \xi_2)$ is a complete Bernstein function of $\tau$. Finally, if $\xi_1 = r_\infty$ and $\xi_2 > 0$, one uses continuity of Wiener--Hopf factors and Proposition~\ref{prop:cbf:limit}.

For the remaining case, once again by by~\eqref{eq:rogers:wh:mobius},
\formula{
 \frac{f^\ua(\tau; \xi_1)}{f^\ua(\tau; \xi_2)} & = \frac{g^\ua(\tau; (1 / \xi_1 - 1 / r_0)^{-1})}{g^\ua(\tau; (1 / \xi_2 - 1 / r_0)^{-1})} \, .
}
As before, if $0 < \xi_1 < \xi_2 < r_\infty$, then $f^\ua(\tau; \xi_1) / f^\ua(\tau; \xi_2)$ is a complete Bernstein function of $\tau$ by Lemma~\ref{lem:rogers:xwh:cbf}. If $0 < \xi_1 < r_\infty < \xi_2$, then
\formula{
 \frac{f^\ua(\tau; \xi_1)}{f^\ua(\tau; \xi_2)} & = \frac{g^\ua(\tau; (1 / \xi_1 - 1 / r_0)^{-1}) g^\da(\tau; (1 / r_0 - 1 / \xi_2)^{-1})}{\tau + \lambda_f(\xi_2)} \, .
}
By Lemma~\ref{lem:rogers:xwh:cbf}, the numerator is a complete Bernstein function of $\tau$, and by Lemma~\ref{lem:rogers:xwh:boundary}\ref{it:rogers:xwh:c}, it extends to a holomorphic function of $\tau \in \C \setminus [-\lambda_f(r_\infty^-), 0]$, with negative values on $(-\infty, -\lambda_f(r_\infty^-))$. By Proposition~\ref{prop:cbf:quot}, $f^\ua(\tau; \xi_1) / f^\da(\tau; \xi_2)$ is a complete Bernstein function of $\tau$. If $r_\infty < \xi_1 < \xi_2$, then
\formula{
 \frac{f^\ua(\tau; \xi_1)}{f^\ua(\tau; \xi_2)} & = \frac{(\tau + \lambda_f(\xi_1)) g^\da(\tau; (1 / r_0 - 1 / \xi_2)^{-1})}{(\tau + \lambda_f(\xi_2)) g^\da(\tau; (1 / r_0 - 1 / \xi_1)^{-1})} \, ,
}
and again, by Lemmas~\ref{lem:rogers:xwh:cbf} and~\ref{lem:rogers:xwh:boundary}\ref{it:rogers:xwh:c}, and by a double application of Proposition~\ref{prop:cbf:quot}, $f^\ua(\tau; \xi_1) / f^\da(\tau; \xi_2)$ is a complete Bernstein function of $\tau$. Finally, if $0 < \xi_1 < \xi_2 = r_\infty$ or $r_\infty = \xi_1 < \xi_2$, one uses continuity of Wiener--Hopf factors and Proposition~\ref{prop:cbf:limit}.
\end{proof}

\begin{remark}
\label{rem:rogers:xwh:nearly}
Using similar methods, Lemma~\ref{lem:rogers:xwh:boundary}, and so also Theorem~\ref{th:kappa:inv}, apparently can be extended to nearly balanced Rogers functions. However, the proof requires rather tedious calculations.
\end{remark}

\subsection{Integral identity}

The following result was proved in the symmetric case in~\cite[Lemma~3.1]{bib:kmr12} and used in~\cite{bib:k11, bib:kmr12} to construct eigenfunction expansion of the generator of the process killed upon leaving half-line. For further discussion, see Section~\ref{subsec:ee}.

\begin{lemma}
\label{lem:rogers:xwh:int}
With the assumptions and notation of Lemma~\ref{lem:rogers:xwh}, for all $\xi_1, \xi_2 > 0$,
\formula[eq:rogers:xwh:int]{
 \frac{\xi_1 + \xi_2}{\pi} \int_0^\infty \frac{f_{[\zeta_f(r)]}^\ua(\xi_1)}{(\xi_1 + i \zeta_f(r)) (\xi_1 - i \overline{\zeta_f(r)})} \, \frac{f_{[\zeta_f(r)]}^\da(\xi_2)}{(\xi_2 - i \zeta_f(r)) (\xi_2 + i \overline{\zeta_f(r)})} \, \lambda_f'(r) \re \zeta_f(r) \D r = 1 ,
}
where $\zeta_f$ is the canonical parametrisation of $\gamma_f \cap \C_\ra$ and $\lambda_f(r) = f(\zeta_f(r))$.
\end{lemma}

\begin{proof}
By Lemma~\ref{lem:rogers:xwh:cbf}, the function $h(\xi_1, \xi_2; \tau) = \tau / (f^\ua(\tau; \xi_1) f^\da(\tau; \xi_2))$ is a complete Bernstein function of $\tau$, and $\lim_{\tau \nearrow \infty} h(\xi_1, \xi_2, \tau) = 1$. Furthermore, by Proposition~\ref{prop:rogers:wh:est},
\formula{
 h(\xi_1, \xi_2; \tau) & \le 2 \, \frac{1 + \xi_1}{\xi_1} \, \frac{1 + \xi_2}{\xi_2} \, \frac{\tau}{\sqrt{|f(1)| + \tau} \, \sqrt{|f(1)| + \tau}} \, ,
}
so that $\lim_{\tau \searrow 0} h(\xi_1, \xi_2; \tau) = 0$. In particular, the constants $c_0$ and $c_1$ are equal to zero in the Stieltjes representation~\eqref{eq:cbf:b} (Theorem~\ref{th:cbf}\ref{it:cbf:b}) of the complete Bernstein function $h(\xi_1, \xi_2; \tau)$.

By Theorem~\ref{th:cbf}\ref{it:cbf:b}, monotone convergence and~\eqref{eq:rogers:xwh:boundary:weak} in Lemma~\ref{lem:rogers:xwh:boundary},
\formula{
 1 & = \lim_{\tau \nearrow \infty} h(\xi_1, \xi_2; \tau) = \frac{1}{\pi} \int_0^\infty \lim_{t \searrow 0} \expr{-\im \frac{h(\xi_1, \xi_2; -s + i t)}{-s + i t}} \D s .
}
By Lemma~\ref{lem:rogers:xwh:boundary}, the integrand vanishes if $0 < s < \inf\{f(\zeta) : \zeta \in \gamma_f\}$ or $s > \sup\{f(\zeta) : \zeta \in \gamma_f\}$. Hence, substituting $s = \lambda_f(r) = f(\zeta_f(r))$,
\formula{
 \frac{1}{\pi} \int_0^\infty \lim_{t \searrow 0} \expr{-\im \frac{h(\xi_1, \xi_2; -f(\zeta_f(r)) + i t)}{-f(\zeta_f(r)) + i t}} \lambda_f'(r) \D r & = 1 .
}
Again by Lemma~\ref{lem:rogers:xwh:boundary} and the identity
\formula{
 -\im \frac{1}{(\xi_1 + i \zeta) (\xi_2 + i \bar{\zeta})} & = \frac{(\xi_1 + \xi_2) \re \zeta}{(\xi_1 + i \zeta) (\xi_1 - i \bar{\zeta}) (\xi_2 - i \zeta) (\xi_2 + i \bar{\zeta})} ,
}
one obtains
\formula{
 \frac{1}{\pi} \int_0^\infty \frac{(\xi_1 + \xi_2) f_{[\zeta_f(r)]}^\ua(\xi_1) f_{[\zeta_f(r)]}^\da(\xi_2) \re \zeta_f(r)}{(\xi_1 + i \zeta_f(r)) (\xi_1 - i \overline{\zeta_f(r)}) (\xi_2 - i \zeta_f(r)) (\xi_2 + i \overline{\zeta_f(r)})} \, \lambda_f'(r) \D r & = 1 ,
}
as desired.
\end{proof}

%
%

\section{Normalisation of extended Wiener--Hopf factors}
\label{sec:norm}

In order to apply the results of the preceding section in the fluctuation theory of L\'evy processes, the extended Wiener--Hopf factors need to be renormalised. In this short section we introduce the modified Wiener--Hopf factors $\kappa_f^\ua(\tau; \xi)$, $\kappa_f^\da(\tau; \xi)$ and $\kappa_f^\bullet(\tau)$, and prove that they satisfy the assertion of Theorem~\ref{th:kappa:cbf}, in particular, they are complete Bernstein functions in both $\xi$ and $\tau$. In the next section it will be argued that \smash{$\kappa_f^\ua(\tau; \xi)$} and \smash{$\kappa_f^\da(\tau; \xi)$} agree with the usual definition of the Wiener--Hopf factors for a L\'evy process with L\'evy--Khintchine exponent $f$.

Recall that $f(\xi) + \tau = f^\ua(\tau; -i \xi) f^\da(\tau; i \xi)$ is the extended Wiener--Hopf factorisation, considered in the previous section.

\subsection{Normalisation}

The main definition requires the following technical result.

\begin{lemma}
\label{lem:rogers:xwh:limit}
For every nonzero Rogers function $f$ and $\tau_1, \tau_2 > 0$, the functions $f^\ua(\tau_1; \xi) / f^\ua(\tau_2; \xi)$ and $f^\da(\tau_1; \xi) / f^\da(\tau_2; \xi)$ converge as $\xi \nearrow \infty$ to positive numbers.
\end{lemma}

\begin{proof}
If $f$ is bounded, then, by Proposition~\ref{prop:rogers:bounded}, $a = f(\infty^-)$ is nonnegative. Hence,
\formula{
 \lim_{\xi \nearrow \infty} \frac{f(\xi) + \tau_1}{f(\xi) + \tau_2} & = \frac{\tau_1 + a}{\tau_2 + a} > 0 .
}
Using the inequality $|\Arg z| \le |\im z| / \re z$ for $z = f(\xi) + \tau \in \C_\ra$, and then again Proposition~\ref{prop:rogers:bounded}, for $\tau, \xi > 0$ one obtains
\formula{
 |\Arg (f(\xi) + \tau)| & \le \frac{|\im f(\xi)|}{\tau + \re f(\xi)} \le \frac{1}{\tau + a} \, \frac{1}{\pi} \int_{-\infty}^\infty \frac{\xi |s|}{\xi^2 + s^2} \, \frac{\mu(\D s)}{|s|} \, ,
}
so that
\formula{
 \int_0^\infty \frac{|\Arg (f(\xi) + \tau)|}{\xi} \, \D \xi & \le \frac{1}{\tau + a} \, \frac{1}{2} \int_{-\infty}^\infty \frac{\mu(\D s)}{|s|} < \infty .
}
Hence the functions $f(\xi) + \tau_1$ and $f(\xi) + \tau_2$ satisfy the assumptions of Lemma~\ref{lem:rogers:wh:limit}, and the result follows.

Suppose now that $f$ is not bounded. Then, by Proposition~\ref{prop:rogers:bounded}, $\lim_{\xi \nearrow \infty} |f(\xi)| = \infty$, so that
\formula{
 \lim_{\xi \nearrow \infty} \frac{f(\xi) + \tau_1}{f(\xi) + \tau_2} & = 1 .
}
Assume with no loss of generality that $\tau_1 \ge \tau_2$. Then $z = 1 + (\tau_1 - \tau_2) / (f(\xi) + \tau_2) \in \C_\ra$ for $\xi > 0$, so that $|\Arg z| \le |\im z| / \re z$. Hence,
\formula{
 \abs{\Arg \frac{f(\xi) + \tau_1}{f(\xi) + \tau_2}} & = \abs{\Arg \expr{1 + \frac{\tau_1 - \tau_2}{f(\xi) + \tau_2}}} \displaybreak[0] \\
 & \le \abs{\im \frac{\tau_1 - \tau_2}{f(\xi) + \tau_2}} \le \frac{(\tau_1 - \tau_2) |\im f(\xi)|}{|f(\xi)|^2} = (\tau_1 - \tau_2) \abs{\im \frac{1}{f(\xi)}} .
}
Therefore,
\formula{
 \int_1^\infty \abs{\Arg \frac{f(\xi) + \tau_1}{f(\xi) + \tau_2}} \frac{\D \xi}{\xi} & \le (\tau_1 - \tau_2) \int_1^\infty \abs{\im \frac{\xi^2}{f(\xi)}} \frac{\D \xi}{\xi^3} \, ,
}
which is finite by Proposition~\ref{prop:rogers:imag} applied to the Rogers function $\xi^2 / f(\xi)$ (see Proposition~\ref{prop:rogers:prop}\ref{it:rogers:prop:b}). Again $f(\xi) + \tau_1$ and $f(\xi) + \tau_2$ satisfy the assumptions of Lemma~\ref{lem:rogers:wh:limit}, and the proof is complete.
\end{proof}

By Lemma~\ref{lem:rogers:xwh:limit}, the following definition is well-formed. It provides a modified version of the extended Wiener--Hopf factorisation: $f(\xi) + \tau = \kappa_f^\bullet(\tau) \kappa_f^\ua(\tau; -i \xi) \kappa_f^\da(\tau; i \xi)$.

\begin{definition}
\label{def:rogers:kappa}
For a nonzero Rogers function $f$, we define
\formula[eq:rogers:kappa]{
 \kappa_f^\ua(\tau; \xi) & = \expr{\lim_{\eta \nearrow \infty} \frac{f^\ua(1; \eta)}{f^\ua(\tau; \eta)}} f^\ua(\tau; \xi) , \\
 \kappa_f^\da(\tau; \xi) & = \expr{\lim_{\eta \nearrow \infty} \frac{f^\da(1; \eta)}{f^\da(\tau; \eta)}} f^\da(\tau; \xi) , \\
 \kappa_f^\bullet(\tau) & = \expr{\lim_{\eta \nearrow \infty} \frac{f^\ua(\tau; \eta)}{f^\ua(1; \eta)}} \expr{\lim_{\eta \nearrow \infty} \frac{f^\da(\tau; \eta)}{f^\da(1; \eta)}} = \frac{f(\xi) + \tau}{\kappa_f^\ua(\tau; -i \xi) \kappa_f^\da(\tau; i \xi)}
}
for all $\xi \in \C \setminus (-\infty, 0]$ and $\tau > 0$.
\end{definition}

\begin{lemma}
\label{lem:rogers:dot}
If $f$ is a bounded nonzero Rogers function, then $\kappa_f^\bullet(\tau) = (\tau + a) / (1 + a)$ for all $\tau > 0$, where $a = f(\infty^-) > 0$. Otherwise, $\kappa_f^\bullet(\tau) = 1$ for all $\tau > 0$.
\end{lemma}

\begin{proof}
Let $a = f(\infty^-)$ if $f$ is bounded, $a = \infty$ otherwise. By Corollary~\ref{cor:rogers:wh}, for $\eta, \tau > 0$,
\formula{
 \frac{f^\ua(\tau; \eta) f^\da(\tau; \eta)}{f^\ua(1; \eta) f^\da(1; \eta)} & = \exp \expr{\frac{1}{2 \pi} \int_{-\infty}^\infty \expr{\frac{1}{\eta + i r} + \frac{1}{\eta - i r}} \log \frac{f(r) + \tau}{f(r) + 1} \, \D r} \displaybreak[0] \\
 & = \exp \expr{\frac{1}{\pi} \int_{-\infty}^\infty \frac{\eta}{\eta^2 + r^2} \, \log \frac{f(r) + \tau}{f(r) + 1} \, \D r} \displaybreak[0] \\
 & = \exp \expr{\frac{1}{\pi} \int_{-\infty}^\infty \frac{1/\eta}{1/\eta^2 + s^2} \, \log \frac{f(1/s) + \tau}{f(1/s) + 1} \, \D s} .
}
As $\eta \nearrow \infty$, the measures $(1/\eta) / (1/\eta^2 + s^2) \D s$ converge to $\pi \delta_0(\D s)$, and hence
\formula{
 \lim_{\eta \nearrow \infty} \frac{f^\ua(\tau; \eta) f^\da(\tau; \eta)}{f^\ua(1; \eta) f^\da(1; \eta)} & = \exp \expr{\log \frac{a + \tau}{a + 1}} ,
}
where if $a = \infty$, it is understood that $(a + \tau) / (a + 1) = 1$.
\end{proof}

\begin{corollary}
\label{cor:rogers:xwh:cbf}
If $f$ is a nearly balanced Rogers function, then
\formula{
 \kappa_f^\ua(\tau; \xi), && \kappa_f^\da(\tau; \xi), && \kappa_f^\bullet(\tau) \kappa_f^\ua(\tau; \xi), && \kappa_f^\bullet(\tau) \kappa_f^\da(\tau; \xi), && \kappa_f^\bullet(\tau)
}
are complete Bernstein functions of both $\xi$ and $\tau$ for all $\xi, \tau > 0$.
\end{corollary}

\begin{proof}
For a fixed $\tau > 0$, $\kappa_f^\ua(\tau; \xi) = c f^\ua(\tau; \xi)$ for some $c > 0$, so that $\kappa_f^\ua(\tau; \xi)$ and $\kappa_f^\bullet(\tau) \kappa_f^\ua(\tau; \xi)$ are complete Bernstein functions of $\xi$. Furthermore, for a fixed $\xi > 0$, by Lemma~\ref{lem:rogers:xwh:nearly}, $\kappa_f^\ua(\tau; \xi)$ and $\kappa_f^\bullet(\tau) \kappa_f^\ua(\tau; \xi)$ are pointwise limits of complete Bernstein functions of $\tau$, namely
\formula{
 \kappa_f^\ua(\tau; \xi) & = \lim_{\eta \nearrow \infty} \expr{f^\ua(1; \eta) \, \frac{f^\ua(\tau; \xi)}{f^\ua(\tau; \eta)}} , & \kappa_f^\bullet(\tau) \kappa_f^\ua(\tau; \xi) & = \lim_{\eta \nearrow \infty} \frac{f^\ua(\tau; \xi) f^\da(\tau; \eta)}{f^\da(1; \eta)} \, .
}
By Proposition~\ref{prop:cbf:limit}, $\kappa_f^\ua(\tau; \xi)$ and $\kappa_f^\bullet(\tau) \kappa_f^\ua(\tau; \xi)$ are complete Bernstein functions of $\tau$. Statements related to $\kappa_f^\da(\tau; \xi)$ are proved in a similar manner. Finally, $\kappa_f^\bullet(\tau)$ is either constant or equal to $(\tau + a) / (1 + a)$ for some $a > 0$, and hence a complete Bernstein function of $\tau$.
\end{proof}

\begin{corollary}
\label{cor:rogers:xwh:xcbf}
If $f$ is a nearly balanced Rogers function, $0 < \tau_1 < \tau_2$ and $0 < \xi_1 < \xi_2$, then
\formula{
 \frac{\kappa_f^\ua(\tau_1; \xi)}{\kappa_f^\ua(\tau_2; \xi)} \, , && \frac{\kappa_f^\da(\tau_1; \xi)}{\kappa_f^\da(\tau_2; \xi)} \, , && \frac{\kappa_f^\ua(\tau; \xi_1)}{\kappa_f^\ua(\tau; \xi_2)} \, , && \frac{\kappa_f^\da(\tau; \xi_1)}{\kappa_f^\da(\tau; \xi_2)}
}
are complete Bernstein functions of $\xi$ and $\tau$, respectively.
\end{corollary}

\begin{proof}
Observe that for fixed $\tau_1, \tau_2$, one has $\kappa_f^\ua(\tau_1; \xi) / \kappa_f^\ua(\tau_2; \xi) = c f^\ua(\tau_1; \xi) / f^\ua(\tau_2; \xi)$ for some $c > 0$. Furthermore, by Corollary~\ref{cor:rogers:wh},
\formula{
 \frac{f^\ua(\tau_1; \xi)}{f^\ua(\tau_2; \xi)} = \frac{f^\ua(\tau_1; 1)}{f^\ua(\tau_2; 1)} \, \frac{g^\ua(\xi)}{g^\ua(1)} \, ,
}
where
\formula{
 g(\xi) & = \frac{\tau_1 + f(\xi)}{\tau_2 + f(\xi)} = \xi^2 \expr{\xi^2 + (\tau_2 - \tau_1) \, \frac{\xi^2}{\tau_1 + f(\xi)}}^{-1}
}
is a Rogers function by Proposition~\ref{prop:rogers:prop}\ref{it:rogers:prop:a}. In particular, $\kappa_f^\ua(\tau_1; \xi) / \kappa_f^\ua(\tau_2; \xi)$ is a complete Bernstein function of $\xi$. By Lemma~\ref{lem:rogers:xwh:nearly}, also $\kappa_f^\ua(\tau; \xi_1) / \kappa_f^\ua(\tau; \xi_2) = f^\ua(\tau; \xi_1) / f^\ua(\tau; \xi_2)$ is a complete Bernstein function of $\tau$. Statements related to $\kappa_f^\da(\tau; \xi)$ are proved in a similar manner.
\end{proof}

\begin{remark}
The exponential representation (Theorem~\ref{th:rogers}\ref{it:rogers:c}) of the functions $\kappa_f^\ua(\tau; \xi_1) / \kappa_f^\ua(\tau; \xi_2) = f^\ua(\tau; \xi_1) / f^\ua(\tau; \xi_2)$ and $\kappa_f^\da(\tau; \xi_1) / \kappa_f^\da(\tau; \xi_2) = f^\da(\tau; \xi_1) / f^\da(\tau; \xi_2)$, where $0 < \xi_1 < \xi_2$, is described by Lemma~\ref{lem:rogers:xwh:boundary}. Similar statement for $\kappa_f^\ua(\tau; \xi)$ and $\kappa_f^\da(\tau; \xi)$, where $\xi > 0$ is fixed, can apparently be obtained by taking an appropriate limit (see Proposition~\ref{prop:cbf:limit}). If this procedure can be justified, then one has
\formula{
 \lim_{t \searrow 0} \kappa_f^\ua(-f(\zeta) + i t; \xi) & = \expr{\frac{f_{[\zeta]}^\ua(\xi)}{\xi + i \zeta}}^{-1} \lim_{\eta \nearrow \infty} \frac{f^\ua(1; \eta) f_{[\zeta]}^\ua(\eta)}{\eta} \, , \displaybreak[0] \\
 \lim_{t \searrow 0} \kappa_f^\da(-f(\zeta) + i t; \xi) & = \expr{\frac{f_{[\zeta]}^\da(\xi)}{\xi + i \bar{\zeta}}}^{-1} \lim_{\eta \nearrow \infty} \frac{f^\da(1; \eta) f_{[\zeta]}^\da(\eta)}{\eta}
}
for $\zeta \in \gamma_f \cap \C_\ra$ and $\xi > 0$. In particular, the measures $\mu_{\xi\ua}(\D s)$, $\mu_{\xi\da}(\D s)$ in the Stieltjes representation~\eqref{eq:cbf:b} (Theorem~\ref{th:cbf}\ref{it:cbf:b}) for $\kappa_f^\ua(\tau; \xi)$ and $\kappa_f^\da(\tau; \xi)$ (as functions of $\tau$) would have density functions $c_\ua(s) / f_{[\zeta(s)]}^\ua(\xi)$, $c_\da(s) / f_{[\zeta(s)]}^\ua(\xi)$ for some $c_\ua(s), c_\da(s)$ and $\zeta(s) \in \gamma_f \cap \C_\ra$ satisfying $f(\zeta(s)) = s$. Noteworthy, these density functions are \emph{Stieltjes functions} of $\xi$.
\end{remark}

\subsection{Formulae and examples}

After proving three different expressions for the normalised extended Wiener--Hopf factors, four examples are discussed. The first result is reminiscent of~\cite{bib:ku10a}.

\begin{corollary}
\label{cor:rogers:kappa:a}
If $f$ is a nonzero Rogers function, then for all $\tau, \xi > 0$
\formula[eq:rogers:kappa:a]{
 \sqrt{\kappa_f^\bullet(\tau)} \, \frac{\kappa_f^\ua(\tau; \xi)}{f^\ua(1; 1)} & = \exp \expr{-\frac{1}{\pi} \int_0^\infty \im \expr{\frac{\log (f(r) + \tau)}{i \xi - r} - \frac{\log (f(r) + 1)}{i - r}} \D r} , \\
 \sqrt{\kappa_f^\bullet(\tau)} \, \frac{\kappa_f^\da(\tau; \xi)}{f^\da(1; 1)} & = \exp \expr{-\frac{1}{\pi} \int_0^\infty \im \expr{\frac{\log (f(r) + \tau)}{i \xi + r} - \frac{\log (f(r) + 1)}{i + r}} \D r} .
}
\end{corollary}

Recall that if $f$ is unbounded, then $\kappa_f^\bullet(\tau) = 1$.

\begin{proof}
By Lemma~\ref{lem:rogers:curve} (see~\eqref{eq:rogers:wh:real}),
\formula{
 \frac{\kappa_f^\ua(\tau; \xi)}{f^\ua(1; 1)} & = \lim_{\eta \nearrow \infty} \exp \biggl( -\frac{1}{\pi} \int_0^\infty \im \biggl( \frac{\log (f(r) + \tau)}{i \xi - r} - \frac{\log (f(r) + 1)}{i - r} \\
 & \hspace*{15em} - \frac{1}{i \eta - r} \, \log \frac{f(r) + \tau}{f(r) + 1} \biggr) \D r \biggr) .
}
Observe that for $\eta \in (0, \infty)$, $r \in (0, \infty)$,
\formula{
 \im \expr{\frac{1}{i \eta - r} \, \log \frac{f(r) + \tau}{f(r) + 1}} & = - \frac{\eta}{\eta^2 + r^2} \log \abs{\frac{f(r) + \tau}{f(r) + 1}} - \frac{r}{\eta^2 + r^2} \im \log \frac{f(r) + \tau}{f(r) + 1} \, .
}
By a substitution $r = \eta s$ and dominated convergence,
\formula{
 \lim_{\eta \nearrow \infty} \int_0^\infty \frac{\eta}{\eta^2 + r^2} \log \abs{\frac{f(r) + \tau}{f(r) + 1}} \D r & = \lim_{\eta \nearrow \infty} \int_0^\infty \frac{1}{1 + s^2} \log \abs{\frac{f(\eta r) + \tau}{f(\eta r) + 1}} \D s = 0
}
if $f$ is unbounded (note that $\lim_{r \nearrow \infty} |f(r)| = \infty$ by Proposition~\ref{prop:rogers:bounded}). For bounded $f$, in a similar manner,
\formula{
 \lim_{\eta \nearrow \infty} \int_0^\infty \frac{\eta}{\eta^2 + r^2} \log \abs{\frac{f(r) + \tau}{f(r) + 1}} \D r & = \frac{\pi}{2} \log \frac{f(\infty^-) + \tau}{f(\infty^-) + 1} \, .
}
Furthermore, by monotone convergence, provided that the integral in the left-hand side of the following formula is finite for some $\eta \in (0, \infty)$,
\formula{
 \lim_{\eta \nearrow \infty} \int_0^\infty \frac{r}{\eta^2 + r^2} \abs{\im \log \frac{f(r) + \tau}{f(r) + 1}} \D r & = 0 ,
}
and the first part of~\eqref{eq:rogers:kappa:a} follows. The other one is proved by a similar argument. Hence, it remains to prove finiteness of the integral in the left-hand side.

Suppose that $\tau > 1$. Observe that the following are complete Bernstein functions of $\xi$: $\log(1 + (\tau - 1) \xi)$, $1 / \log(1 + (\tau - 1) / \xi)$, and finally
\formula{
 g(\xi) & = \expr{\log \frac{\xi + \tau}{\xi + 1}}^{-1} = \expr{\log \expr{1 + \frac{\tau - 1}{\xi + 1}}}^{-1} .
}
By Proposition~\ref{prop:rogers:conv}, $g(f(r))$ is a Rogers function. Hence, so is
\formula{
 h(r) & = \frac{r^2}{g(f(r))} = r^2 \log \frac{f(r) + \tau}{f(r) + 1} \, .
}
By Proposition~\ref{prop:rogers:imag},
\formula{
 \int_0^\infty \frac{r}{\eta^2 + r^2} \abs{\im \log \frac{f(r) + \tau}{f(r) + 1}} \D r & = \int_0^\infty \frac{r}{\eta^2 + r^2} \frac{\im |h(r)|}{r^2} \, \D r
}
is finite for $\eta = 1$.

If $\tau < 1$, the argument is very similar (we omit the details). The case $\tau = 1$ reduces to Lemma~\ref{lem:rogers:curve} (see~\eqref{eq:rogers:wh:real}).
\end{proof}

By Lemma~\ref{lem:rogers:xwh} and dominated convergence, one obtains easily the following two results (we omit the details).

\begin{corollary}
\label{cor:rogers:kappa:b}
If $f$ is a balanced Rogers function and $\int_1^\infty 1 / (r \lambda_f(r)) \D r < \infty$, then for all $\tau, \xi > 0$
\formula[eq:rogers:kappa:b]{
 \frac{\kappa_f^\ua(\tau; \xi)}{f^\ua(1; 1)} & = \exp \expr{-\frac{1}{\pi} \int_0^\infty \im \expr{\frac{\zeta_f'(r) \log (\lambda_f(r) + \tau)}{i \xi - \zeta_f(r)} - \frac{\zeta_f'(r) \log (\lambda_f(r) + 1)}{i - \zeta_f(r)}} \D r} , \\
 \frac{\kappa_f^\da(\tau; \xi)}{f^\da(1; 1)} & = \exp \expr{-\frac{1}{\pi} \int_0^\infty \im \expr{\frac{\zeta_f'(r) \log (\lambda_f(r) + \tau)}{i \xi + \zeta_f(r)} - \frac{\zeta_f'(r) \log (\lambda_f(r) + 1)}{i + \zeta_f(r)}} \D r} .
}
Here $\lambda_f(r) = f(\zeta_f(r))$, where $\zeta_f(r)$ is the canonical parametrisation of $\gamma_f \cap \C_\ra$.
\end{corollary}

\begin{corollary}
\label{cor:rogers:kappa:c}
If $f$ is a balanced Rogers function, then for all $\tau, \xi > 0$
\formula[eq:rogers:kappa:c]{
 \frac{\kappa_f^\ua(\tau; \xi)}{f^\ua(1; 1)} & = \exp \expr{-\frac{1}{\pi} \int_0^\infty \expr{\frac{\lambda_f'(r) \Arg(\zeta_f(r) - i \xi)}{\lambda_f(r) + \tau} - \frac{\lambda_f'(r) \Arg(\zeta_f(r) - i)}{\lambda_f(r) + 1}} \D r} , \\
 \frac{\kappa_f^\da(\tau; \xi)}{f^\da(1; 1)} & = \exp \expr{-\frac{1}{\pi} \int_0^\infty \expr{\frac{\lambda_f'(r) \Arg(\zeta_f(r) + i \xi)}{\lambda_f(r) + \tau} - \frac{\lambda_f'(r) \Arg(\zeta_f(r) + i)}{\lambda_f(r) + 1}} \D r} .
}
Here $\lambda_f(r) = f(\zeta_f(r))$, where $\zeta_f(r)$ is the canonical parametrisation of $\gamma_f \cap \C_\ra$.
\end{corollary}

\begin{example}
\label{ex:rogers:xwh:bm}
Let $f(\xi) = \tfrac{1}{2} \xi^2 - i b \xi$ with $b \in \R$ be the L\'evy--Khintchine exponent of the Brownian motion with drift. By a simple calculation,
\formula{
 f(\xi) + \tau & = \tfrac{1}{2} (-i \xi + \sqrt{b^2 + 2 \tau} - b) (i \xi + \sqrt{b^2 + 2 \tau} + b) ,
}
so that
\formula{
 \frac{f^\ua(\tau; \xi_1)}{f^\ua(\tau; \xi_2)} & = \frac{\xi_1 + \sqrt{b^2 + 2 \tau} - b}{\xi_2 + \sqrt{b^2 + 2 \tau} - b} \, , & \frac{f^\da(\tau; \xi_1)}{f^\da(\tau; \xi_2)} & = \frac{\xi_1 + \sqrt{b^2 + 2 \tau} + b}{\xi_2 + \sqrt{b^2 + 2 \tau} + b}
}
and
\formula{
 \kappa_f^\ua(\tau; \xi) & = \sqrt{\frac{1 + \sqrt{b^2 + 2} + b}{2 (1 + \sqrt{b^2 + 2} - b)}} \, (\xi + \sqrt{b^2 + 2 \tau} - b) , \displaybreak[0] \\
 \kappa_f^\da(\tau; \xi) & = \sqrt{\frac{1 + \sqrt{b^2 + 2} - b}{2 (1 + \sqrt{b^2 + 2} + b)}} \, (\xi + \sqrt{b^2 + 2 \tau} + b) .
}
These formulae are well-known, see e.g.~\cite{bib:k06}.
\end{example}

\begin{example}
\label{ex:rogers:xwh:stable}
If $f$ is the L\'evy--Khintchine exponent of a strictly stable L\'evy process (see Section~\ref{sec:ex}), the extended Wiener--Hopf factors typically are not given by a closed-form expression; see~\cite{bib:ku13} for the description of known cases and series representation. Here we consider case of stability index $1$, that is, $f(\xi) = a \xi$ with $\re a \ge 0$. Then $\zeta_f(r) = r e^{i \thet}$ and $\lambda_f(r) = c r$, where $\thet = -\Arg a$ and $c = |a|$ (see Example~\ref{ex:rogers:real}). Denote $\delta = e^{i \thet}$. By~\eqref{eq:rogers:wh:uu},
\formula{
 \frac{f^\ua(\tau; \xi_1)}{f^\ua(\tau; \xi_2)} & = \exp \expr{-\frac{1}{\pi} \int_0^\infty \im \expr{\frac{\delta}{i \xi_1 - \delta r} - \frac{\delta}{i \xi_2 - \delta r}} \log(c r + \tau) \D r}
}
for $\tau, \xi_1, \xi_2 > 0$. Recall that the dilogarithm $\Li x$ is the antiderivative of $-\log(1 - x) / x$, holomorphic in $\C \setminus [1, \infty)$. By a short calculation,
\formula{
 -\int \frac{\delta}{i \xi - \delta r} \, \log(c r + \tau) \D r & = \Li \frac{\delta (c r + \tau)}{\delta \tau + i c \xi} + \log (c r + \tau) \log \expr{1 - \frac{\delta (c r + \tau)}{\delta \tau + i c \xi}} \, .
}
Since $\im (\delta / (\delta \tau + i c \xi)) < 0$ when $\xi > 0$, and since $\Li x + \tfrac{1}{2} (\log x)^2 + i \pi \log x$ converges to $\pi^2/3$ as $|x| \nearrow \infty$ and $\im x < 0$, one obtains, after a short calculation,
\formula{
 -\int_0^\infty \expr{\frac{\delta}{i \xi_1 - \delta r} - \frac{\delta}{i \xi_2 - \delta r}} \log(c r + \tau) \D r
 & = \Li \frac{\delta \tau}{\delta \tau + i c \xi_2} - \Li \frac{\delta \tau}{\delta \tau + i c \xi_1} \\
 & \hspace*{-17em} + \log \tau \log \frac{\xi_2}{\xi_1} + \expr{i \pi + \frac{1}{2} \log \frac{\delta \tau}{\delta \tau + i c \xi_1} + \frac{1}{2} \log \frac{\delta \tau}{\delta \tau + i c \xi_2}} \log \frac{\delta \tau + i c \xi_1}{\delta \tau + i c \xi_2} .
}
Hence, for $\tau, \xi_1, \xi_2 > 0$,
\formula{
 \frac{f^\ua(\tau; \xi_1)}{f^\ua(\tau; \xi_2)} & = \exp \biggl( \frac{1}{\pi} \im \biggl( \Li \frac{\delta \tau}{\delta \tau + i c \xi_2} - \Li \frac{\delta \tau}{\delta \tau + i c \xi_1} \\
 & + \expr{i \pi + \frac{1}{2} \log \frac{\delta \tau}{\delta \tau + i c \xi_1} + \frac{1}{2} \log \frac{\delta \tau}{\delta \tau + i c \xi_2}} \log \frac{\delta \tau + i c \xi_1}{\delta \tau + i c \xi_2} \biggr) \biggr) ,
}
and, in a similar manner, for $\tau, \xi > 0$,
\formula{
 \frac{\kappa_f^\ua(\tau; \xi)}{f^\ua(1; 1)} & = \exp \biggl( \frac{1}{\pi} \im \biggl(- \Li \frac{\delta \tau}{\delta \tau + i c \xi} - \frac{i \pi}{2} \log \tau + \expr{i \pi + \frac{1}{2} \log \frac{\delta^2 \tau^2}{\delta \tau + i c \xi}} \log(\delta \tau + i c \xi) \\
 & \hspace*{13.5em} + \Li \frac{\delta}{\delta + i c} - \expr{i \pi + \log \frac{\delta^2}{\delta + i c}} \log(\delta + i c) \biggr) \biggr) .
}
It follows that
\formula{
 \kappa_f^\ua(\tau; \xi) & = \frac{c}{\sqrt{\tau}} \, \exp \expr{\re \expr{\expr{1 + \frac{1}{2 \pi i} \log \frac{\delta^2 \tau^2}{\delta \tau + i c \xi}} \log(\delta \tau + i c \xi)} - \frac{1}{\pi} \im \Li \frac{\delta \tau}{\delta \tau + i c \xi} }
}
for some constant $c > 0$. By a similar argument,
\formula{
 \kappa_f^\da(\tau; \xi) & = \frac{c}{\sqrt{\tau}} \, \exp \expr{\re \expr{\expr{1 + \frac{1}{2 \pi i} \log \frac{\bar{\delta}^2 \tau^2}{\bar{\delta} \tau + i c \xi}} \log(\bar{\delta} \tau + i c \xi)} - \frac{1}{\pi} \im \Li \frac{\bar{\delta} \tau}{\bar{\delta} \tau + i c \xi} }
}
for a (possibly different) constant $c > 0$. For a related calculation in the symmetric case, see~\cite{bib:kkms10}.
\end{example}

The following explicit example illustrates the extended Wiener--Hopf factorisation for nearly balanced Rogers functions.

\begin{example}
\label{ex:rogers:xwh:nearly:1}
Let $f(\xi) = \xi / (\xi - 4 i) - i \xi = (-i \xi) (i \xi + 3) (i \xi + 4)^{-1}$ be the L\'evy--Khintchine exponent of the classical risk process (see Example~\ref{ex:rogers}\itref{ex:rogers:e}). Then $f^\ua(\xi) = (2 / \sqrt{5}) \xi$ and $f^\da(\xi) = (\sqrt{5} / 2) (\xi + 3) / (\xi + 4)$. Furthermore, by a simple calculation, $\im f(x + i y) = 0$ if and only if $x = 0$ or $x^2 + (y - 4)^2 = 4$, so $\gamma_f$ is the circle $|\xi - 4 i| = 2$ and $|\gamma_f| = (2, 6)$. It can be checked that
\formula{
 f(\xi) + \tau & = \frac{\xi^2 + (\tau - 3) i \xi + 4 \tau}{i \xi + 4} \displaybreak[0] \\
 & = \frac{(i \xi + \tfrac{1}{2}(\sqrt{\tau + 1} \sqrt{\tau + 9} - \tau + 3))(-i \xi + \tfrac{1}{2}(\sqrt{\tau + 1} \sqrt{\tau + 9} + \tau - 3))}{i \xi + 4}
}
It follows that
\formula{
 \frac{f^\ua(\tau; \xi_1)}{f^\ua(\tau; \xi_2)} & = \frac{\xi_1 + \tfrac{1}{2}(\sqrt{\tau + 1} \sqrt{\tau + 9} + \tau - 3)}{\xi_2 + \tfrac{1}{2}(\sqrt{\tau + 1} \sqrt{\tau + 9} + \tau - 3)} \, , \displaybreak[0] \\
 \frac{f^\da(\tau; \xi_1)}{f^\da(\tau; \xi_2)} & = \frac{(\xi_1 + \tfrac{1}{2}(\sqrt{\tau + 1} \sqrt{\tau + 9} - \tau + 3)) (\xi_2 + 4)}{(\xi_2 + \tfrac{1}{2}(\sqrt{\tau + 1} \sqrt{\tau + 9} - \tau + 3)) (\xi_1 + 4)}
}
and
\formula{
 \kappa_f^\ua(\tau; \xi) & = \frac{1}{\sqrt{5}} \expr{\xi + \tfrac{1}{2}(\sqrt{\tau + 1} \sqrt{\tau + 9} + \tau - 3)} , \displaybreak[0] \\
 \kappa_f^\da(\tau; \xi) & = \sqrt{5} \, \frac{\xi + \tfrac{1}{2}(\sqrt{\tau + 1} \sqrt{\tau + 9} - \tau + 3)}{\xi + 4} \, .
}
\end{example}

%
%

\section{Applications to fluctuation theory}
\label{sec:ft}

The results of the previous section apply directly to the Wiener--Hopf factorisation in fluctuation theory of L\'evy processes. In the following section the main results of the article are proved. First, a short introduction to fluctuation theory is provided. Next it is shown that the Laplace exponents of ladder processes are the Wiener--Hopf factors $\kappa_f^\ua(\tau; \xi)$ and $\kappa_f^\da(\tau; \xi)$, studied in the previous section. This already proves Theorem~\ref{th:kappa:cbf}. Then, a formula for the space-only Laplace exponent of the supremum functional is given (Theorem~\ref{th:kappa:inv}), and it is specialised for strictly stable L\'evy processes (Theorem~\ref{th:kappa:stable}). Finally, inversion of this Laplace transform (Conjecture~\ref{con:kappa:sup}) and connection to generalised eigenfunction expansions are discussed.

\subsection{Elements of fluctuation theory}

Let $X_t$ be a L\'evy process and $\Psi$ its L\'evy--Khintchine exponent. Define the supremum and the infimum functionals
\formula{
 X^\ua_t & = \sup \tset{X_s : s \in [0, t]} , & X^\da_t & = \inf \tset{X_s : s \in [0, t]}
}
Furthermore, let
\formula{
 T_t^\ua & = \inf \tset{s \in [0, t] : X_s = X^\ua_t} , & T_t^\da & = \inf \tset{s \in [0, t] : X_s = X^\da_t}
}
be the (unique almost surely if $X_t$ is not a compound Poisson process) times at which the process $X_t$ attains its supremum and infimum over $[0, t]$, respectively. Fluctuation theory provides a description of the above objects and studies their properties. This is typically done using Laplace exponents of increasing and decreasing \emph{ladder processes}, which are given by the Fristedt--Pecherski--Rogozin formulae: for $\tau \ge 0$ and $\xi \in \overline{\C}_\ra$,
\formula[eq:ft:f]{
 \kappa^\ua(\tau; \xi) & = \exp\expr{\int_0^\infty \int_{(0, \infty)} \frac{e^{-t} - e^{-\tau t - \xi x}}{t} \, \pr(X_t \in \D x) \D t} , \\
 \kappa^\da(\tau; \xi) & = \exp\expr{\int_0^\infty \int_{(-\infty, 0)} \frac{e^{-t} - e^{-\tau t + \xi x}}{t} \, \pr(X_t \in \D x) \D t} , \\
 \kappa^\bullet(\tau) & = \exp\expr{\int_0^\infty \frac{e^{-t} - e^{-\tau t}}{t} \, \pr(X_t = 0) \D t} ,
}
and for $\tau, \sigma \ge 0$ and $\xi \in \overline{\C}_\ra$,
\formula[eq:ft:pr]{
 \int_0^\infty \ex \exp(-\xi X^\ua_t - \tau T_t^\ua) \sigma e^{-\sigma t} \D t & = \frac{\kappa^\ua(\sigma; 0)}{\kappa^\ua(\tau + \sigma; \xi)} \, , \\
 \int_0^\infty \ex \exp(\xi X^\da_t - \tau T_t^\da) \sigma e^{-\sigma t} \D t & = \frac{\kappa^\da(\sigma; 0)}{\kappa^\da(\tau + \sigma; \xi)} .
}
For the derivation of these formulae and various extensions, see~\cite{bib:b96, bib:d07, bib:k06}.

As was observed e.g.\ in~\cite{bib:b96}, for $\tau > 0$ and $\xi \in \R$,
\formula{
 \kappa^\bullet(\tau) \kappa^\ua(\tau; -i \xi) \kappa^\da(\tau; i \xi) & = \exp\expr{\int_0^\infty \int_{\R \setminus \{0\}} \frac{e^{-t} - e^{-\tau t + i \xi x}}{t} \, \pr(X_t \in \D x) \D t} , \displaybreak[0] \\
 & = \exp\expr{\int_0^\infty \frac{e^{-t} - e^{-t (\tau + \Psi(\xi))}}{t} \, \D t} \displaybreak[0] \\
 & = \exp(\log(\tau + \Psi(\xi))) = \tau + \Psi(\xi)
}
(the second equality is an application of Frullani integral). This provides the Wiener--Hopf factorisation of $\tau + \Psi(\xi)$.

\begin{remark}
Unless $X_t$ is a compound Poisson process (that is, $\Psi$ is bounded on $\R$), $\kappa^\bullet(\tau) = 1$ and therefore it can be ignored in definitions. Some authors decide to exclude compound Poisson processes from their results, others incorporate $\kappa^\bullet(\tau)$ into $\kappa^\ua(\tau; \xi)$ or $\kappa^\da(\tau; \xi)$. However, this either breaks the perfect duality between the Wiener--Hopf factors, or leads to less general results. For this reason, below $\kappa^\bullet(\tau)$ is kept in the notation, following, e.g.,~\cite{bib:tt02, bib:tt08}.
\end{remark}

\subsection{Connection with Wiener--Hopf factorisation for Rogers functions}

If $\Psi$ is equal to a Rogers function $f$, then, by uniqueness of the Wiener--Hopf factorisation, the two factorisations of $\tau + \Psi(\xi) = \tau + f(\xi)$: Fristedt--Pecherski--Rogozin one using~\eqref{eq:ft:f} and Baxter--Donsker-type one based on~\eqref{eq:rogers:kappa}, must agree up to a constant. That is, $\kappa^\ua(\tau; \xi) = c_1(\tau) \kappa_f^\ua(\tau; \xi)$ and $\kappa^\da(\tau; \xi) = c_2(\tau) \kappa_f^\da(\tau; \xi)$ for $\xi \in \C_\ra$ for some $c_1(\tau), c_2(\tau) > 0$. Furthermore, by Definition~\ref{def:rogers:kappa},
\formula{
 \lim_{\xi \nearrow \infty} \frac{\kappa_f^\ua(\tau; \xi)}{\kappa_f^\ua(1; \xi)} & = 1 ,
}
and by the definition~\eqref{eq:ft:f} and monotone convergence,
\formula{
 \lim_{\xi \nearrow \infty} \frac{\kappa^\ua(\tau; \xi)}{\kappa^\ua(1; \xi)} & = 1 .
}
It follows that $c_1(\tau)$ does not depend on $\tau$. In a similar manner, $c_2(\tau)$ is a constant. As a corollary, $\kappa^\bullet(\tau) = c_3 \kappa_f^\bullet(\tau)$ for some constant $c_3$. If $0 < \tau_1 < \tau_2$ and $0 < \xi_1 < \xi_2$, Theorem~\ref{th:kappa:cbf} is therefore an immediate consequence of Corollaries~\ref{cor:rogers:xwh:cbf} and~\ref{cor:rogers:xwh:xcbf}. The cases $\tau_1 = 0$ and $\xi_1 = 0$ follow by continuity and Proposition~\ref{prop:cbf:limit}; the cases $\tau_1 = \tau_2$ and $\xi_1 = \xi_2$ are trivial.

As an almost immediate consequence of Theorem~\ref{th:kappa:cbf}, one obtains complete monotonicity of the following functionals of $X_t$.

\begin{theorem}
\label{th:kappa:cm}
If $X_t$ is a L\'evy process whose L\'evy--Khintchine exponent is a nearly balanced Rogers function, then for $\sigma > 0$ and $\xi > 0$,
\formula{
 \frac{\D}{\D x} \int_0^\infty \sigma e^{-\sigma t} \pr(X^\ua_t < x) \D t , && \frac{\D}{\D s} \int_0^\infty \sigma e^{-\sigma t} \pr(T^\ua_t < s) \D t && \text{and} && \ex \exp(-\xi X^\ua_t)
}
are completely monotone functions on $(0, \infty)$ of $x$, $s$ and $t$, respectively. Similar statement holds for $X^\da_t$ and $T^\da_t$.
\end{theorem}

\begin{proof}
The Laplace transforms of the above expressions in the corresponding variables are given by
\formula{
 \frac{\kappa^\ua(\sigma; 0)}{\kappa^\ua(\sigma; \xi)} \, , && \frac{\kappa^\ua(\sigma; 0)}{\kappa^\ua(\tau + \sigma; 0)} \, , && \frac{1}{\sigma} \, \frac{\kappa^\ua(\sigma; 0)}{\kappa^\ua(\sigma; \xi)} \, ,
}
as functions of $\xi$, $\tau$ and $\sigma$, respectively. All these functions are \emph{Stieltjes functions}: the first two are reciprocals of complete Bernstein functions (of $\xi$ an and $\tau$, respectively), and the last one is a complete Bernstein function of $\sigma$ divided by $\sigma$ (see~\cite[Chapter~2 and Theorem~7.3]{bib:ssv10}). To complete the proof, recall that a Stieltjes function is the Laplace transform of a sum of a completely monotone function on $(0, \infty)$ and a multiple of the Dirac measure $\delta_0$ (see~\cite[Theorem~2.2]{bib:ssv10}).
\end{proof}

\subsection{Space-only Laplace transform of the supremum}

Inversion of the Laplace transform in the time variable $\sigma$ is possible, and it leads to a technically complicated, but useful expression for the Laplace transform of $X^\ua_t$ and $X^\da_t$, contained in Theorem~\ref{th:kappa:inv}.

\begin{proof}[Proof of Theorem~\ref{th:kappa:inv}]
For consistency with the main part of the article, denote the L\'evy--Khintchine exponent of $X_t$ by $f$. Note that with the present notation, $\zeta(r) = \zeta_f(r)$, $\lambda(r) = \lambda_f(r)$, $\Psi_{r\ua}(\xi) = f_{[\zeta_f(r)]}^\ua(\xi) / f_{[\zeta_f(r)]}^\ua(0^+)$ and $\Psi_{r\da}(\xi) = f_{[\zeta_f(r)]}^\da(\xi) / f_{[\zeta_f(r)]}^\da(0^+)$ in~\eqref{eq:kappa:inv} and~\eqref{eq:kappa:inv:ratio}. Formula~\eqref{eq:kappa:inv:ratio} follows directly from Lemma~\ref{lem:rogers:curve}. It remains to prove~\eqref{eq:kappa:inv}.

Let $\xi > 0$. As in the proof of Theorem~\ref{th:kappa:cm}, the Laplace transform of $\ex \exp(-\xi X^\ua_t)$, where $t$ is transformed into $\sigma$, is given by $\sigma^{-1} \kappa_f^\ua(\sigma; 0) / \kappa_f^\ua(\sigma; \xi)$. By Lemma~\ref{lem:rogers:xwh:cbf}, for $\eps \in (0, \xi)$,
\formula{
 \frac{\kappa_f^\ua(\sigma; \eps)}{\kappa_f^\ua(\sigma; \xi)} & = \frac{f^\ua(\sigma; \eps)}{f^\ua(\sigma; \xi)}
}
is a complete Bernstein function of $\sigma$, whose limit as $\sigma \nearrow \infty$ is equal to $1$. Hence, in the Stieltjes representation~\eqref{eq:cbf:b} (Theorem~\ref{th:cbf}\ref{it:cbf:b}) of the above function, the constant $c_1$ is equal to zero. By Theorem~\ref{th:cbf} and Lemma~\ref{lem:rogers:xwh:boundary},
\formula[eq:kappa:inv:aux]{
 \frac{\kappa_f^\ua(\sigma; \eps)}{\kappa_f^\ua(\sigma; \xi)} & = \frac{\kappa_f^\ua(0; \eps)}{\kappa_f^\ua(0; \xi)} + \frac{1}{\pi} \int_0^\infty \frac{f_{[\zeta(r)]}^\ua(\xi)}{f_{[\zeta(r)]}^\ua(\eps)} \, \im \expr{\frac{\eps + i \zeta(r)}{\xi + i \zeta(r)}} \frac{\sigma}{\sigma + \lambda(r)} \, \frac{\lambda'(r)}{\lambda(r)} \, \D r \\
 & = \lim_{\tau \searrow 0} \frac{f^\ua(\tau; \eps)}{f^\ua(\tau; \xi)} + \frac{1}{\pi} \int_0^\infty \frac{f_{[\zeta(r)]}^\ua(\xi)}{f_{[\zeta(r)]}^\ua(\eps)} \, \frac{(\xi + \eps) \re \zeta(r)}{|\xi + i \zeta(r)|^2} \, \frac{\sigma}{\sigma + \lambda(r)} \, \frac{\lambda'(r)}{\lambda(r)} \, \D r ,
}
with $\zeta = \zeta_f$ and $\lambda = \lambda_f$.

By Lemma~\ref{lem:rogers:xwh} and monotone convergence (note that the integrand is negative),
\formula{
 \lim_{\tau \searrow 0} \frac{f^\ua(\tau; \eps)}{f^\ua(\tau; \xi)} & = \exp\expr{\frac{1}{\pi} \int_0^\infty \Arg\expr{\frac{\zeta(r) - i \xi}{\zeta(r) - i \eps}} \frac{\lambda'(r)}{\lambda(r)} \, \D r} .
}
By Fatou's lemma,
\formula{
 \lim_{\eps \searrow 0} \lim_{\tau \searrow 0} \frac{f^\ua(\tau; \eps)}{f^\ua(\tau; \xi)} & \le \exp\expr{\frac{1}{\pi} \int_0^\infty \Arg\expr{\frac{\zeta(r) - i \xi}{\zeta(r)}} \frac{\lambda'(r)}{\lambda(r)} \, \D r} .
}
By Definition~\ref{def:rogers:reg}, there is $\alpha \in (0, \tfrac{\pi}{2})$ such that $|\Arg \zeta(r)| \le \alpha$ for $r \in (0, \infty)$. This and the inequality $\arctan x \ge x / (1 + x)$ for $x > 0$ imply that
\formula{
 \Arg\expr{\frac{\zeta(r) - i \xi}{\zeta(r)}} & \le \Arg (1 - i e^{i \alpha} (\xi / r)) = -\arctan \frac{(\xi / r) \sin \alpha}{1 + (\xi / r) \cos \alpha} \displaybreak[0] \\
 & \le -\frac{(\xi / r) \sin \alpha}{1 + (\xi / r) \cos \alpha + (\xi / r) \sin \alpha} \le -\frac{\sin \alpha}{2} \, \frac{\xi/r}{1 + \xi/r} = -\frac{\sin \alpha}{2} \, \frac{\xi}{r + \xi} \, .
}
Therefore,
\formula{
 \lim_{\eps \searrow 0} \lim_{\sigma \searrow 0} \frac{f^\ua(\sigma; \eps)}{f^\ua(\sigma; \xi)} & \le \exp\expr{-\frac{\sin \alpha}{2 \pi} \int_0^\infty \frac{\xi}{r + \xi} \, \frac{\lambda'(r)}{\lambda(r)} \, \D r} \le \exp\expr{-\frac{\sin \alpha}{4 \pi} \int_0^\xi \frac{\lambda'(r)}{\lambda(r)} \, \D r} = 0 ,
}
because $\lim_{r \searrow 0} \log \lambda(r) = -\infty$. Hence, the first term in the right-hand side of~\eqref{eq:kappa:inv:aux} is equal to zero, that is,
\formula{
 \frac{\kappa_f^\ua(\sigma; 0)}{\kappa_f^\ua(\sigma; \xi)} & = \frac{1}{\pi} \, \lim_{\eps \searrow 0} \int_0^\infty \frac{f_{[\zeta(r)]}^\ua(\xi)}{f_{[\zeta(r)]}^\ua(\eps)} \, \frac{(\xi + \eps) \re \zeta(r)}{|\xi + i \zeta(r)|^2} \, \frac{\sigma}{\sigma + \lambda(r)} \, \frac{\lambda'(r)}{\lambda(r)} \, \D r .
}
Note that $(\xi + \eps) \le 2 \xi$ and $f_{[\zeta(r)]}^\ua(\eps) \ge f_{[\zeta(r)]}^\ua(0^+)$ for $\eps \in (0, \xi)$. By dominated convergence,
\formula{
 \frac{1}{\sigma} \, \frac{\kappa_f^\ua(\sigma; 0)}{\kappa_f^\ua(\sigma; \xi)} & = \frac{1}{\pi} \int_0^\infty \frac{f_{[\zeta(r)]}^\ua(\xi)}{f_{[\zeta(r)]}^\ua(0^+)} \, \frac{\xi \re \zeta(r)}{|\xi + i \zeta(r)|^2} \, \frac{1}{\sigma + \lambda(r)} \, \frac{\lambda'(r)}{\lambda(r)} \, \D r .
}
It remains to observe that $1 / (\sigma + \lambda(r)) = \int_0^\infty e^{-\sigma t - \lambda(r) t} \D t$, and use Fubini and uniqueness of the Laplace transform.
\end{proof}

\begin{remark}
The second paragraph of the proof shows that for a L\'evy process whose L\'evy--Khintchine exponent is a balanced Rogers function, one has $\kappa^\ua(0; 0) = \kappa^\da(0; 0) = 0$. This property is equivalent to each of the following two statements: the ladder processes are proper subordinators (with no killing); $X_t$ oscillates as $t \nearrow \infty$.
\end{remark}

\begin{proof}[Proof of Theorem~\ref{th:kappa:stable}]
Recall that $X_t$ is a strictly stable L\'evy process with L\'evy--Khintchine exponent $f(\xi) = a \xi^\alpha$ for $\xi > 0$ (see Section~\ref{sec:ex}), and the Rogers function $f$ is nondegenerate, that is, $a \ne 0$ and $|\Arg a| < \tfrac{\alpha \pi}{2}$ (equality may hold in the other condition in~\eqref{eq:stable:a}, $|\Arg a| \le (2 - \alpha) \tfrac{\alpha \pi}{2}$). In this case $\zeta_f(r) = r e^{i \thet}$ and $\lambda_f(r) = c r^\alpha$ for $r > 0$, where $c = |a|$ and $\thet = -\tfrac{1}{\alpha} \Arg a$ (see Example~\ref{ex:rogers:real}).

By~\eqref{eq:kappa:inv} in Theorem~\ref{th:kappa:inv}, and a substitution $r = \xi u$, with the present notation,
\formula{
 \ex \exp(-\xi X^\ua_t) & = \frac{1}{\pi} \int_0^\infty \frac{f_{[r e^{i \thet}]}^\ua(\xi)}{f_{[r e^{i \thet}]}^\ua(0^+)} \, \frac{\xi r \cos \thet}{\xi^2 - 2 \xi r \sin \thet + r^2} \, \frac{c \alpha r^{\alpha - 1}}{c r^\alpha} \, e^{-c t r^\alpha} \D r \displaybreak[0] \\
 & = \frac{\alpha}{\pi} \int_0^\infty \frac{f_{[\xi u e^{i \thet}]}^\ua(\xi)}{f_{[\xi u e^{i \thet}]}^\ua(0^+)} \, \frac{\cos \thet}{1 - 2 u \sin \thet + u^2} \, e^{-c t \xi^\alpha u^\alpha} \D u .
}
By~\eqref{eq:kappa:inv:ratio} in Theorem~\ref{th:kappa:inv},
\formula{
 \frac{f_{[r e^{i \thet}]}^\ua(\xi)}{f_{[r e^{i \thet}]}^\ua(0^+)} & = \lim_{\eps \searrow 0} \exp \Biggl( \frac{1}{\pi} \int_0^\infty \re \Biggl( \expr{\frac{e^{i \thet}}{\xi + i s e^{i \thet}} - \frac{e^{i \thet}}{\eps + i s e^{i \thet}}} \times \\
 & \hspace*{9em} \times \log \frac{(s e^{i \thet} - r e^{i \thet}) (s e^{i \thet} + r e^{-i \thet})}{c s^\alpha - c r^\alpha} \Biggr) \D s \Biggr) \displaybreak[0] \\
 & = \lim_{\eps \searrow 0} \exp \Biggl( \frac{1}{\pi} \int_0^\infty \expr{\frac{\xi \cos \thet}{|\xi + i s e^{i \thet}|^2} - \frac{\eps \cos \thet}{|\eps + i s e^{i \thet}|^2}} \log \frac{(s - r) |s e^{2 i \thet} + r|}{c (s^\alpha - r^\alpha)} \, \D s \\
 & \hspace*{7em} - \frac{1}{\pi} \int_0^\infty \expr{\frac{\xi \sin \thet - s}{|\xi + i s e^{i \thet}|^2} - \frac{\eps \sin \thet - s}{|\eps + i s e^{i \thet}|^2}} \Arg(s e^{2 i \thet} + r) \D s \Biggr) .
}
By a direct calculation,
\formula[eq:kappa:stable:aux]{
 \frac{1}{\pi} \int_0^\infty \frac{\eps \cos \thet}{|\eps + i s e^{i \thet}|^2} \, \D s & = \frac{1}{2} + \frac{\thet}{\pi} = \ro
}
(see~\eqref{eq:stable:a:ro}), and hence
\formula{
 \lim_{\eps \searrow 0} \expr{\frac{1}{\pi} \, \frac{\eps \cos \thet}{|\eps + i s e^{i \thet}|^2} \, \ind_{(0, \infty)}(s) \D s} & = \ro \delta_0(\D s) .
}
Using this for the former integral and dominated convergence for the latter one (we omit the details), one obtains
\formula{
 \frac{f_{[r e^{i \thet}]}^\ua(\xi)}{f_{[r e^{i \thet}]}^\ua(0^+)} & = \exp \Biggl( \frac{1}{\pi} \int_0^\infty \frac{\xi \cos \thet}{|\xi + i s e^{i \thet}|^2} \, \log \frac{(s - r) |s e^{2 i \thet} + r|}{c (s^\alpha - r^\alpha)} \, \D s - \ro \log \frac{r^{2-\alpha}}{c} \\
 & \hspace*{12em} - \frac{1}{\pi} \int_0^\infty \expr{\frac{\xi \sin \thet - s}{|\xi + i s e^{i \thet}|^2} + \frac{1}{s}} \Arg(s e^{2 i \thet} + r) \D s \Biggr) .
}
When $r = \xi u$, a substitution $s = \xi v$ gives
\formula{
 \frac{f_{[\xi u e^{i \thet}]}^\ua(\xi)}{f_{[\xi u e^{i \thet}]}^\ua(0^+)} & = \exp \Biggl( \frac{1}{\pi} \int_0^\infty \frac{\cos \thet}{|1 + i v e^{i \thet}|^2} \, \log \frac{(v - u) |v e^{2 i \thet} + u|}{c \xi^{\alpha - 2} (v^\alpha - u^\alpha)} \, \D v \\
 & \hspace*{-1em} - \ro \log \frac{(\xi u)^{2-\alpha}}{c} - \frac{1}{\pi} \int_0^\infty \expr{\frac{\sin \thet - v}{|1 + i v e^{i \thet}|^2} + \frac{1}{v}} \frac{\Arg(v e^{2 i \thet} + u)}{v} \, \D v \Biggr) .
}
Finally, by~\eqref{eq:kappa:stable:aux},
\formula{
 \frac{f_{[\xi u e^{i \thet}]}^\ua(\xi)}{f_{[\xi u e^{i \thet}]}^\ua(0^+)} & = \exp \Biggl( \frac{1}{\pi} \int_0^\infty \frac{\cos \thet}{1 - 2 v \sin \thet + v^2} \, \log \frac{(v - u) |v e^{2 i \thet} + u|}{v^\alpha - u^\alpha} \D v \\
 & \hspace*{-1em} - \ro \log u^{2-\alpha} - \frac{1}{\pi} \int_0^\infty \frac{1 - v \sin \thet}{1 - 2 v \sin \thet + v^2} \, \frac{\Arg(v e^{2 i \thet} + u)}{v} \, \D v \Biggr) .
}
The above calculations prove that
\formula[eq:kappa:stable:inv:alt]{
 \ex \exp(-\xi X^\ua_t) & = \frac{\alpha}{\pi} \int_0^\infty \exp \Biggl( \frac{1}{\pi} \int_0^\infty \frac{\cos \thet}{1 - 2 v \sin \thet + v^2} \, \log \frac{(v - u) |v e^{2 i \thet} + u|}{v^\alpha - u^\alpha} \D v \\
 & \hspace*{-4em} - \frac{1}{\pi} \int_0^\infty \frac{1 - v \sin \thet}{1 - 2 v \sin \thet + v^2} \, \frac{\Arg(v e^{2 i \thet} + u)}{v} \, \D v \Biggr) \frac{u^{-(2 - \alpha) \ro} \cos \thet}{1 - 2 u \sin \thet + u^2} \, e^{-c t \xi^\alpha u^\alpha} \D u .
}
It remains to use $\thet = \pi \ro - \tfrac{\pi}{2}$.
\end{proof}

\subsection{Toward the formula for the distribution function of the supremum and generalised eigenfunction expansion}
\label{subsec:ee}

In the remaining part of this section, inversion of the Laplace transform in Theorem~\ref{th:kappa:inv} is discussed. This requires the following definition, which for symmetric Rogers functions reduces to generalised eigenfunctions studied in~\cite{bib:k11, bib:kmr12}.

\begin{definition}
\label{def:kappa:f}
Let $f$ be a balanced Rogers function and $r > 0$. The functions $F_{f \ua}(r; x)$ and $F_{f \da}(r; x)$ of $x \in (0, \infty)$ are defined by means of Laplace transform,
\formula{
 \laplace F_{f \ua}(r; \xi) & = \frac{\re \zeta_f(r)}{|f_{[\zeta_f(r)]}^\ua(i \overline{\zeta_f(r)})|} \, \frac{f_{[\zeta_f(r)]}^\ua(\xi)}{(\xi + i \zeta_f(r)) (\xi - i \overline{\zeta_f(r)})}
}
for $\xi \in \C_\ra$ such that $\re \xi > \im \zeta_f(r)$, and
\formula{
 \laplace F_{f \da}(r; \xi) & = \frac{\re \zeta_f(r)}{|f_{[\zeta_f(r)]}^\da(i \zeta_f(r))|} \, \frac{f_{[\zeta_f(r)]}^\da(\xi)}{(\xi - i \zeta_f(r)) (\xi + i \overline{\zeta_f(r)})}
}
for $\xi \in \C_\ra$ such that $\re \xi > -\im \zeta_f(r)$.
\end{definition}

\begin{lemma}
\label{lem:kappa:f}
Let $f$ be a balanced Rogers function, $r > 0$ and $\zeta_f(r) = a_f(r) + i b_f(r)$. Then for $x > 0$,
\formula{
 F_{f \ua}(r; x) & = e^{b_f(r) x} \sin(a_f(r) x + \thet_{f \ua}(r)) - G_{f \ua}(r; x) , \\
 F_{f \da}(r; x) & = e^{-b_f(r) x} \sin(a_f(r) x + \thet_{f \da}(r)) - G_{f \da}(r; x) ,
}
where $G_{f \ua}(r; x)$, $G_{f \da}(r; x)$ are completely monotone functions of $x \in (0, \infty)$, with Laplace transforms
\formula{
 \laplace G_{f \ua}(r; \xi) & = \frac{1}{2 i |f_{[\zeta]}^\ua(i \bar{\zeta})|} \expr{\frac{f_{[\zeta]}^\ua(i \bar{\zeta})}{\xi - i \bar{\zeta}} - \frac{f_{[\zeta]}^\ua(-i \zeta)}{\xi + i \zeta}} - \frac{\re \zeta}{|f_{[\zeta]}^\ua(i \bar{\zeta})|} \, \frac{f_{[\zeta]}^\ua(\xi)}{(\xi - i \bar{\zeta}) (\xi + i \zeta)} , \displaybreak[0] \\
 \laplace G_{f \da}(r; \xi) & = \frac{1}{2 i |f_{[\zeta]}^\da(i \zeta)|} \expr{\frac{f_{[\zeta]}^\da(i \zeta)}{\xi - i \zeta} - \frac{f_{[\zeta]}^\da(-i \bar{\zeta})}{\xi + i \bar{\zeta}}} - \frac{\re \zeta}{|f_{[\zeta]}^\da(i \zeta)|} \, \frac{f_{[\zeta]}^\da(\xi)}{(\xi - i \zeta) (\xi + i \bar{\zeta})}
}
(here $\zeta = \zeta_f(r)$), and $\thet_{f \ua}(r), \thet_{f \da}(r) \in [0, \pi)$ are given by
\formula[eq:kappa:thet]{
 \thet_{f \ua}(r) & = \Arg f_{[\zeta_f(r)]}^\ua(i \overline{\zeta_f(r)}) = -\Arg f_{[\zeta_f(r)]}^\ua(-i \zeta_f(r)) , \\
 \thet_{f \da}(r) & = \Arg f_{[\zeta_f(r)]}^\da(i \zeta_f(r)) = -\Arg f_{[\zeta_f(r)]}^\da(-i \overline{\zeta_f(r)}) .
}
Furthermore, $\thet_{f \ua}(r) - \thet_{f \da}(r) = \Arg f'(\zeta_f(r)) = -\Arg \zeta_f'(r)$.
\end{lemma}

\begin{proof}
Let $\zeta = \zeta_f(r)$. By Lemma~\ref{lem:rogers:quot} and Theorem~\ref{th:rogers:wh}, $f_{[\zeta]}^\ua$ is a complete Bernstein function. By Lemma~\ref{lem:cbf:quot} (applied to $f_{[\zeta]}^\ua$ and $i \bar{\zeta}$),
\formula{
 \frac{f_{[\zeta]}^\ua(\xi)}{(\xi - i \bar{\zeta}) (\xi + i \zeta)} & = \frac{1}{2 i \re \zeta} \expr{\frac{f_{[\zeta]}^\ua(i \bar{\zeta})}{\xi - i \bar{\zeta}} - \frac{f_{[\zeta]}^\ua(-i \zeta)}{\xi + i \zeta}} - \frac{g_\ua(r; \xi)}{\xi}
}
for $\xi \in \C \setminus (-\infty, 0]$, where $g_\ua(r; \xi)$ is a complete Bernstein function of $\xi$.

Recall that for $g \in \cbf$, $g(\xi) / \xi$ is a \emph{Stieltjes function}, that is, the Laplace transform of a completely monotone function plus a multiple of the Dirac measure $\delta_0$ (see~\cite[Theorem~2.2]{bib:ssv10}). Furthermore, $g_\ua(r; \xi) / \xi$ converges to $0$ as $\xi \nearrow \infty$ (see Lemma~\ref{lem:cbf:quot}). Therefore, $g_\ua(r; \xi) / \xi$ is the Laplace transform of a completely monotone function.

Observe that since $f_{[\zeta]}^\ua$ and $f_{[\zeta]}^\da$ are complete Bernstein functions, one has
\formula{
 f_{[\zeta]}^\ua(i \bar{\zeta}) & = \overline{f_{[\zeta]}^\ua(-i \zeta)} , & f_{[\zeta]}^\da(-i \bar{\zeta}) & = \overline{f_{[\zeta]}^\da(i \zeta)} .
}
This proves that~\eqref{eq:kappa:thet} properly defines $\thet_{f \ua}(r)$ and $\thet_{f \da}(r)$. Furthermore, by Lemma~\ref{lem:rogers:quot} and Theorem~\ref{th:rogers:wh},
\formula[eq:kappa:f:wh]{
 f_{[\zeta]}^\ua(-i \zeta) f_{[\zeta]}^\da(i \zeta) = f_{[\zeta]}(\zeta) = \frac{2 \re \zeta}{f'(\zeta)} ,
}
and therefore $\thet_{f \da}(r) - \thet_{f \ua}(r) = -\Arg f'(\zeta)$. 

Finally, for $z \in \C$, the function $1 / (\xi - z)$ is (the holomorphic extension of) the Laplace transform of $e^{z x} \ind_{\hl}(x)$. It follows that
\formula{
 F_{f \ua}(r; x) & = \frac{\exp(i \bar{\zeta} x + i \thet_{f \ua}(\zeta)) - \exp(-i \zeta x - i \thet_{f \da}(\zeta))}{2 i} - G_{f \ua}(\zeta; x) ,
}
as desired. The proof for $F_{f \da}(r; x)$ is very similar.
\end{proof}

It follows that the first part of~\eqref{eq:kappa:inv} can be rewritten as
\formula{
 \int_0^\infty e^{-\xi x} \pr(X^\ua_t < x) \D x & = \int_0^\infty \frac{|f_{[\zeta_f(r)]}^\ua(i \overline{\zeta_f(r)})|}{f_{[\zeta_f(r)]}^\ua(0^+)} \, \laplace F_{f\ua}(r; x) e^{-t \lambda_f(r)} \lambda_f'(r) \D r .
}
Provided that Fubini can be used to change the order of integration, one can invert the Laplace transform, thus obtaining a semi-explicit expression for the distribution of $X^\ua_t$. This leads to the following conjectured result, which possibly could be proved by using the methods developed in~\cite{bib:k11, bib:kmr12} for the symmetric case. The natural question whether the assumption $\sup \{ \im \zeta_f(r) : r \in (0, \infty)\} < \infty$ can be relaxed seems challenging, even for strictly stable L\'evy processes.

\begin{conjecture}
\label{con:kappa:sup}
If $X_t$ is a L\'evy process whose L\'evy--Khintchine exponent $f$ is a balanced Rogers function, $\sup \{ \im \zeta_f(r) : r \in (0, \infty)\} < \infty$, $\sup \{ \thet_{f\ua}(r) : r \in (0, \infty) \} < \tfrac{\pi}{2}$, and
\formula{
 \int_1^\infty \frac{|f_{[\zeta_f(r)]}^\ua(i \overline{\zeta_f(r)})|}{f_{[\zeta_f(r)]}^\ua(0^+)} \, e^{-t_0 \lambda_f(r)} \lambda_f'(r) \D r & < \infty
}
for some $t_0 \in [0, \infty)$, then
\formula[eq:kappa:sup]{
 \pr(X^\ua_t < x) & = \int_0^\infty \frac{|f_{[\zeta_f(r)]}^\ua(i \overline{\zeta_f(r)})|}{f_{[\zeta_f(r)]}^\ua(0^+)} \, F_{f\ua}(r; x) e^{-t \lambda_f(r)} \lambda_f'(r) \D r
}
for $t > t_0$ and $x > 0$. Here $\zeta_f$ is the canonical parametrisation of $\gamma_f$ (see Definition~\ref{def:rogers:curve}), $\lambda_f(r) = f(\zeta_f(r))$, $f_{[\zeta]}$ is defined in Lemma~\ref{lem:rogers:quot}, Wiener--Hopf factor $f_{[\zeta]}^\ua$ is defined in Theorem~\ref{th:rogers:wh} and $F_{f \ua}(r; x)$ is described by Lemma~\ref{lem:kappa:f}.
\end{conjecture}

\begin{remark}
If $\{ \im \zeta_f(r) : r \in (0, \infty) \}$ is a bounded subset of $\R$, then the expression for $\pr(X_t^\ua < x)$ in the above conjecture corresponds to the eigenfunction expansion of the constant $1$ in the system of generalised eigenfunctions $F_{f\ua}(r; -x)$ and coeigenfunctions $F_{f\da}(r; -x)$ of the generator of the process $X_t$, killed upon leaving $(-\infty, 0)$. With this interpretation, Lemma~\ref{lem:rogers:xwh:int} can be rewritten as
\formula[eq:ee:completeness]{
 \frac{2}{\pi} \int_0^\infty \laplace F_{f\ua}(r; \xi_1) \laplace F_{f_\da}(r; \xi_2) |\zeta_f'(r)| \D r & = \frac{1}{\xi_1 + \xi_2} \, .
}
The numbers $\laplace F_{f\da}(r; \xi_2)$ are the coefficients of the eigenfunction expansion of $f_2(x) = e^{\xi_2 x} \ind_{(-\infty, 0)}(x)$. Similarly, $\laplace F_{f\ua}(r; \xi_1)$ are the coefficients of the coeigenfunction expansion of $f_1(x) = e^{\xi_1 x} \ind_{(-\infty, 0)}(x)$. Since $\tscalar{f_1, f_2} = 1 / (\xi_1 + \xi_2)$, and the functions of the form $e^{\xi x} \ind_{(-\infty, 0)}(x)$ form a linearly dense set in $L^2((-\infty, 0))$, formula~\eqref{eq:ee:completeness} asserts that $F_{f\ua}(r; x)$ and $F_{f\da}(r; \xi)$ form a complete system of generalised eigenfunctions and coeigenfunctions.

When $\{ \im \zeta_f(r) : r \in (0, \infty) \}$ is not bounded, a formal statement of the notion of generalised eigenfunction expansion is already a challenging problem. With an appropriate definition, Lemma~\ref{lem:rogers:xwh:int} can be hoped to be applicable in a similar way as described above.
\end{remark}

\subsection{Generalised eigenfunctions for strictly stable L\'evy processes (or power-type Rogers functions)}

Suppose that the assumptions of Theorem~\ref{th:kappa:stable} are satisfied, that is, the L\'evy--Khintchine exponent of $X_t$ is the Rogers function $f(\xi) = a \xi^\alpha$ for $\xi \in \C_\ra$, where $\alpha \in (0, 2]$, $a$ satisfies~\eqref{eq:stable:a} and $|\Arg a| < \tfrac{\alpha \pi}{2}$ (see Section~\ref{sec:ex}). Below formulae for $\thet_{f\ua}(r)$, $\thet_{f\da}(r)$, $G_{f\ua}(r; x)$ and $G_{f\da}(r; x)$ are found.

\begin{lemma}
\label{lem:ee:stable}
Let $X_t$ be a strictly stable L\'evy process with stability index $\alpha \in (0, 2]$, positivity parameter $\ro \in (0, 1)$ and scale parameter $k > 0$. If $f$ denotes the corresponding Rogers function, then for $r, x > 0$,
\formula{
 F_{f\ua}(r; x) & = e^{-r x \cos(\ro \pi)} \sin(r x \sin(\ro \pi) + (1 - \ro) (1 - \alpha \ro) \tfrac{\pi}{2}) - G_{f\ua}(r x) ,
}
where
\formula{
 G_{f\ua}(x) & = \frac{1}{\pi} \sqrt{\frac{\alpha \cos \thet}{2}} \int_0^\infty \frac{\sin(\alpha \pi \ro) u^\alpha}{1 + 2 u^\alpha \cos(\alpha \pi \ro) + u^{2 \alpha}} \, e^{I(u) + J(u) - x u} \D u ,
}
with
\formula{
 I(u) & = -\frac{1}{2 \pi} \int_0^\infty \expr{\frac{2 u \sin(\ro \pi)}{u^2 - 2 u v \cos(\ro \pi) + v^2} - \frac{\sin(2 \ro \pi)}{1 - 2 v \cos(2 \ro \pi) + v^2}} \times \\
 & \hspace*{17em} \times \log \frac{(v - 1) \sqrt{1 - 2 v \cos(2 \ro \pi) + v^2}}{v^\alpha - 1} \, \D v \\
 J(u) & = \frac{1}{2 \pi} \, \pvint_0^\infty \expr{\frac{2 v - 2 u \cos(\ro \pi)}{u^2 - 2 u v \cos(\ro \pi) + v^2} + \frac{1}{1 - v} - \frac{v - \cos(2 \ro \pi)}{1 - 2 v \cos(2 \ro \pi) + v^2}} \times \\
 & \hspace*{17em} \times \Arg(1 - v \cos(2 \ro \pi) + i v \sin(2 \ro \pi)) \D v ,
}
and $\pvint$ denoting the Cauchy principal value integral, with a singularity at $v = 1$. A similar formula for $F_{f\da}(r; x)$ is obtained by replacing $\ro$ by $1 - \ro$.
\end{lemma}

\begin{proof}
Recall that $f(\xi) = a \xi^\alpha$ for $\xi \in \C_\ra$. Let $c = |a|$ and $\thet = -\tfrac{1}{\alpha} \Arg a$, so that $\zeta_f(r) = r e^{i \thet}$ and $\lambda_f(r) = c r^\alpha$. As in the proof of Theorem~\ref{th:kappa:stable},
\formula{
 \log f_{[r e^{i \thet}]}(s e^{i \thet}) & = \log \frac{(s - r) |s e^{2 i \thet} + r|}{c (s^\alpha - r^\alpha)} + i \Arg (s e^{2 i \thet} + r) ,
}
and $\log f_{[r e^{i \thet}]}(-s e^{-i \thet})$ is the conjugate number. Furthermore, by Lemma~\ref{lem:rogers:curve} with $\xi_2 = \xi$ and $\xi_1 = -i \zeta_f(r) + \eps e^{i \thet} = -i (r + \eps i) e^{i \thet}$ as $\eps \searrow 0$,
\formula[eq:ee:stable:aux]{
 \log \frac{f_{[r e^{i \thet}]}^\ua(-i r e^{i \thet})}{f_{[r e^{i \thet}]}^\ua(\xi)} & = -\frac{1}{2 \pi i} \, \pvint_0^\infty \expr{\frac{e^{i \thet}}{r e^{i \thet} - s e^{i \thet}} - \frac{e^{i \thet}}{i \xi - s e^{i \thet}}} \log f_{[r e^{i \thet}]}(s e^{i \thet}) \D s \\
 & \hspace*{-7.8em} - \frac{1}{2 \pi i} \int_0^\infty \expr{\frac{e^{-i \thet}}{r e^{i \thet} + s e^{-i \thet}} - \frac{e^{-i \thet}}{i \xi + s e^{-i \thet}}} \log f_{[r e^{i \thet}]}(-s e^{-i \thet}) \D s + \frac{\log f_{[r e^{i \thet}]}(r e^{i \thet})}{2} \, ; 
}
the last term in the right-hand side corresponds to $-\tfrac{1}{2} \delta_r(\D s)$ being the limit of measures $\tfrac{1}{2 \pi} \im (e^{i \thet} / ((r + i \eps) e^{i \thet} - s e^{i \thet})) \D s$ as $\eps \searrow 0$ (we omit the details).

Recall that $\thet_{f\ua}(r) = -\im (\log f_{[r e^{i \thet}]}^\ua(-i r e^{i \thet}))$, and that $f_{[r e^{i \thet}]}^\ua(1) > 0$. By taking $\xi = 1$ in~\eqref{eq:ee:stable:aux} and comparing the imaginary parts of both sides,
\formula{
 \thet_{f\ua}(r) & = -\frac{1}{2 \pi} \, \pvint_0^\infty \re \expr{\frac{1}{r - s} \, \log f_{[r e^{i \thet}]}(s e^{i \thet}) + \frac{1}{r e^{2 i \thet} + s} \, \log f_{[r e^{i \thet}]}(-s e^{-i \thet})} \D s \\
 & \hspace*{26em} - \frac{\Arg f_{[r e^{i \thet}]}(r e^{i \thet})}{2} \, .
}
Since $f_{[r e^{i \thet}]}(r e^{i \thet}) = (2 r \cos \thet) / f'(r e^{i \thet}) = \tfrac{2}{\alpha c} r^{2 - \alpha} e^{i \thet} \cos \thet$, one has
\formula{
 \thet_{f\ua}(r) & = -\frac{1}{2 \pi} \, \pvint_0^\infty \biggl(\expr{\frac{1}{r - s} + \frac{r \cos(2 \thet) + s}{|r e^{2 i \thet} + s|^2}} \log \frac{(s - r) |s e^{2 i \thet} + r|}{c (s^\alpha - r^\alpha)} \\
 & \hspace*{15em} - \frac{r \sin(2 \thet)}{|r e^{2 i \thet} + s|^2} \, \Arg(s e^{2 i \thet} + r) \biggr) \D s - \frac{\thet}{2} \, .
}
Substituting $s = r u$ in the integral over $(0, r - \eps)$ and $s = r / u$ in the integral over $(r + \eps, \infty)$, and taking the limit as $\eps \searrow 0$, one obtains
\formula{
 \thet_{f\ua}(r) & = -\frac{1}{2 \pi} \int_0^1 \biggl(\expr{\frac{1}{1 - u} + \frac{\cos(2 \thet) + u}{|e^{2 i \thet} + u|^2}} \log \frac{(u - 1) |u e^{2 i \thet} + 1|}{c r^{\alpha - 2} (u^\alpha - 1)} \\
 & \hspace*{18em} - \frac{\sin(2 \thet)}{|e^{2 i \thet} + u|^2} \, \Arg(u e^{2 i \thet} + 1) \\
 & \hspace*{7em} + \frac{1}{u} \expr{\frac{1}{u - 1} + \frac{\cos(2 \thet) u + 1}{|e^{2 i \thet} u + 1|^2}} \log \frac{(1 - u) |e^{2 i \thet} + u|}{c r^{\alpha - 2} u^{2 - \alpha} (1 - u^\alpha)} \\
 & \hspace*{18em} - \frac{\sin(2 \thet)}{|e^{2 i \thet} u + 1|^2} \, \Arg(e^{2 i \thet} + u) \biggr) \D u - \frac{\thet}{2} \, .
}
Elementary simplification yields
\formula{
 \thet_{f\ua}(r) & = \frac{1}{4 \pi} \int_0^1 \expr{\frac{2}{1 - u} + \frac{1}{u + e^{2 i \thet}} + \frac{1}{u + e^{-2 i \thet}}} \log \frac{1}{u^{2 - \alpha}} \, \D u \\
 & \hspace*{15em} + \frac{1}{\pi} \int_0^1 \frac{\thet \sin(2 \thet)}{u^2 + 2 u \cos(2 \thet) + 1} \, \D u - \frac{\thet}{2} \, .
}
By a direct calculation,
\formula{
 \int_0^1 \frac{\thet \sin(2 \thet)}{u^2 + 2 u \cos(2 \thet) + 1} \, \D u & = \thet^2 .
}
As in~\cite[Example~6.1]{bib:k11}, by the series expansion, Fubini and integration by parts,
\formula{
 \hspace*{3em} & \hspace*{-3em} \int_0^1 \expr{\frac{2}{1 - u} + \frac{1}{u + e^{2 i \thet}} + \frac{1}{u + e^{-2 i \thet}}} \log \frac{1}{u^{2 - \alpha}} \, \D u \\
 & = -(2 - \alpha) \sum_{n = 0}^\infty (2 - (-e^{2 i \thet})^{n+1} - (-e^{-2 i \thet})^{n+1}) \int_0^1 u^n \log u \, \D u \\
 & = 2 (2 - \alpha) \sum_{n = 0}^\infty \frac{1 - \cos(2 (n+1) (\pi + 2 \thet))}{(n + 1)^2} = \frac{(2 - \alpha) (\pi^2 - 4 \thet^2)}{2} \, ;
}
the last equality follows by the Fourier series expansion of the right-hand side. Therefore,
\formula[eq:ee:stable:theta]{
 \thet_{f\ua}(r) & = \frac{(2 - \alpha) (\pi^2 - 4 \thet^2)}{8 \pi} + \frac{\thet^2}{\pi} - \frac{\thet}{2} = \frac{(1 - \ro) (1 - \alpha \ro) \pi}{2} \, ,
}
because $\ro = \tfrac{1}{2} + \tfrac{\thet}{\pi}$.

An explicit formula for $G_{f\ua}(r; x)$ is obtained by inverting the Laplace transform. Let $r > 0$ and $\zeta = \zeta_f(r) = r e^{i \thet}$. By Lemma~\ref{lem:kappa:f}, the Laplace transform of $G_{f\ua}(r; x)$ is a Stieltjes function, so that $G_{f\ua}(r; x)$ is the Laplace transform of a measure, which is described by the boundary values of the holomorphic extension of $\laplace G_{f\ua}(r; \xi)$ along $(-\infty, 0]$. By the definition of the Wiener--Hopf factors (Theorem~\ref{th:rogers:wh}), the holomorphic extension of $\laplace G_{f\ua}(r; \xi)$ satisfies (we omit the details)
\formula{
 \lim_{t \nearrow 0} \im(\laplace G_{f\ua}(r; -s + i t)) & = -\frac{\re \zeta}{|f_{[\zeta]}^\ua(-i \zeta)|} \, \frac{1}{f_{[\zeta]}^\da(s) (-s - i \bar{\zeta}) (-s + i \zeta)} \lim_{t \nearrow 0} \im(f_{[\zeta]}(-i s - t)) \\
 & = \frac{r \cos \thet}{c \, |f_{[\zeta]}^\ua(-i \zeta) f_{[\zeta]}^\da(s)|} \, \frac{\sin(\alpha \pi \ro) s^\alpha}{r^{2 \alpha} + 2 r^\alpha s^\alpha \cos(\alpha \pi \ro) + s^{2 \alpha}} \, .
}
Therefore (see~\cite[]{bib:ssv10}),
\formula{
 G_{f\ua}(r; x) & = \frac{1}{\pi} \int_0^\infty \frac{r \cos \thet}{c \, |f_{[r e^{i \thet}]}^\ua(-i r e^{i \thet}) f_{[r e^{i \thet}]}^\da(s)|} \, \frac{\sin(\alpha \pi \ro) s^\alpha}{r^{2 \alpha} + 2 r^\alpha s^\alpha \cos(\alpha \pi \ro) + s^{2 \alpha}} \, e^{-x s} \D s \\
 & = \frac{1}{\pi} \int_0^\infty \frac{r^{2-\alpha} \cos \thet}{c \, |f_{[r e^{i \thet}]}^\ua(-i r e^{i \thet}) f_{[r e^{i \thet}]}^\da(r u)|} \, \frac{\sin(\alpha \pi \ro) u^\alpha}{1 + 2 u^\alpha \cos(\alpha \pi \ro) + u^{2 \alpha}} \, e^{-r x u} \D u .
}
As in~\eqref{eq:ee:stable:aux}, for $\xi > 0$,
\formula{
 \log (f_{[r e^{i \thet}]}^\ua(-i r e^{i \thet}) f_{[r e^{i \thet}]}^\da(\xi)) & = -\frac{1}{2 \pi i} \, \pvint_0^\infty \expr{\frac{e^{i \thet}}{r e^{i \thet} - s e^{i \thet}} + \frac{e^{i \thet}}{i \xi + s e^{i \thet}}} \log f_{[r e^{i \thet}]}(s e^{i \thet}) \D s \\
 & \hspace*{-7.8em} - \frac{1}{2 \pi i} \int_0^\infty \expr{\frac{e^{-i \thet}}{r e^{i \thet} + s e^{-i \thet}} + \frac{e^{-i \thet}}{i \xi - s e^{-i \thet}}} \log f_{[r e^{i \thet}]}(-s e^{-i \thet}) \D s + \frac{\log f_{[r e^{i \thet}]}(r e^{i \thet})}{2} \, .
}
Hence, taking $\xi = r u$ and comparing real parts of both sides, in a similar way as in the calculation of $\thet_{f\ua}(r)$,
\formula{
 \log \frac{c |f_{[r e^{i \thet}]}^\ua(-i r e^{i \thet}) f_{[r e^{i \thet}]}^\da(r u)|}{r^{2 - \alpha}} & = \log \frac{c}{r^{2 - \alpha}} + \frac{\log |f_{[r e^{i \thet}]}(r e^{i \thet})|}{2} \\
 & \hspace*{-11em} - \frac{1}{2 \pi} \, \pvint_0^\infty \biggl(\expr{-\frac{r \sin(2 \thet)}{|r e^{2 i \thet} + s|^2} - 2 \, \frac{r u \cos \thet}{|i r u e^{-i \thet} + s|^2}} \log \frac{(s - r) |s e^{2 i \thet} + r|}{c (s^\alpha - r^\alpha)} \\
 & \hspace*{-5em} + \expr{\frac{1}{r - s} - \frac{r \cos(2 \thet) + s}{|r e^{2 i \thet} + s|^2} + 2 \, \frac{r u \sin \thet + s}{|i r u e^{-i \thet} + s|^2}} \Arg (s e^{2 i \thet} + r)\biggr) \D s \, .
}
By a substitution $s = r u$, simplification and direct calculation (we omit the details),
\formula{
 \log \frac{c |f_{[r e^{i \thet}]}^\ua(-i r e^{i \thet})f_{[r e^{i \thet}]}^\da(r u)|}{r^{2 - \alpha}} & = \frac{1}{2} \log \frac{2 \cos(\thet)}{\alpha} + \\
 & \hspace*{-9em} + \frac{1}{2 \pi} \int_0^\infty \expr{\frac{\sin(2 \thet)}{|e^{2 i \thet} + v|^2} + 2 \, \frac{u \cos \thet}{|i u e^{-i \thet} + v|^2}} \log \frac{(v - 1) |v e^{2 i \thet} + 1|}{v^\alpha - 1} \, \D v \\
 & \hspace*{-7em} + \frac{1}{2 \pi} \, \pvint_0^\infty \expr{\frac{1}{1 - v} - \frac{\cos(2 \thet) + v}{|e^{2 i \thet} + v|^2} + 2 \, \frac{u \sin \thet + v}{|i u e^{-i \thet} + v|^2}} \Arg(v e^{2 i \thet} + 1) \D v .
}
This completes the proof.
\end{proof}

\addtocontents{toc}{\SkipTocEntry}
\section*{}
\subsection*{Acknowledgements}

I thank Panki Kim, Alexey Kuznetsov, Pierre Patie, Ren\'e Schilling and Zoran Vondra\v{c}ek for inspiring discussions on the subject of the article.

%
%

%
%


\begin{thebibliography}{00}

\bibitem{bib:bnmr01}
O.~E.~Barndorff-Nielsen, T.~Mikosch, S.~I.~Resnick (Eds.),
\emph{L{\'e}vy Processes: Theory and Applications}.
Birkh{\"a}user, Boston, 2001.

\bibitem{bib:bd57}
G.~Baxter, M.~D.~Donsker,
\emph{On the distribution of the supremum functional for processes with stationary independent increments}.
Trans. Amer. Math. Soc. 85 (1957) 73--87.

\bibitem{bib:bdp08}
V.~Bernyk, R.~C.~Dalang, G.~Peskir,
\emph{The law of the supremum of a stable L\'{e}vy process with no negative jumps}.
Ann. Probab. 36(5) (2008): 1777--1789.

\bibitem{bib:b96}
J.~Bertoin,
\emph{L\'{e}vy Processes}.
Cambridge Univ. Press, Melbourne, New York, 1996.

\bibitem{bib:b73}
N.~H.~Bingham,
\emph{Maxima of sums of random variables and suprema of stable processes}.
Z.~Wahrscheinlichkeitstheorie Verw. Gebiete 26 (1973): 273--296.

\bibitem{bib:bbkrsv09}
K.~Bogdan, T.~Byczkowski, T.~Kulczycki, M.~Ryznar, R.~Song, Z.~Vondra\v{c}ek,
\emph{Potential Analysis of Stable Processes and its Extensions}.
Lecture Notes in Mathematics 1980, Springer, 2009.

\bibitem{bib:bgr13}
K.~Bogdan, T.~Grzywny, M.~Ryznar,
\emph{Density and tails of unimodal convolution semigroups}.
Preprint, 2013, arXiv:1305.0976v1.

\bibitem{bib:bgr14}
K.~Bogdan, T.~Grzywny, M.~Ryznar,
\emph{Barriers, exit time and survival probability for unimodal L\'evy processes}.
Preprint, 2013, arXiv:1307.0270v1.

\bibitem{bib:bkkk14}
J.~Burridge, A.~Kuznetsov, A.~E.~Kyprianou, M.~Kwa\'snicki,
\emph{New families of subordinators with explicit transition probability semigroup}.
In preparation.

\bibitem{bib:d56}
D.~A.~Darling,
\emph{The maximum of sums of stable random variables}.
Trans. Amer. Math. Soc. 83 (1956) 164--169.

\bibitem{bib:d87}
R.~A.~Doney,
\emph{On Wiener-Hopf factorisation and the distribution of extrema for certain stable processes}.
Ann. Probab. 15(4) (1987) 1352--1362.

\bibitem{bib:d07}
R.~A.~Doney,
\emph{Fluctuation Theory for L{\'e}vy Processes}.
Lecture Notes in Math. 1897, Springer, Berlin, 2007.

\bibitem{bib:dr11}
R.~A.~Doney, V.~Rivero,
\emph{Asymptotic behaviour of first passage time distributions for L\'evy processes}.
Probab. Theory Related Fields 157(1) (2013): 1--45.

\bibitem{bib:ds10}
R.~A.~Doney, M.~S.~Savov,
\emph{The asymptotic behavior of densities related to the supremum of a stable process}.
Ann. Probab. 38(1) (2010): 316--326.

\bibitem{bib:f74}
B.~E.~Fristedt,
\emph{Sample functions of stochastic processes with stationary, independent increments}.
In: \emph{Advances in Probability and Related Topics}, vol. 3, Dekker, New York, 1974, 241--396.

\bibitem{bib:gj10}
P.~Graczyk, T.~Jakubowski,
\emph{On Wiener-Hopf factors of stable processes}.
Ann. Inst. Henri Poincar{\'e} (B) 47(1) (2010): 9--19.

\bibitem{bib:gj12}
P.~Graczyk, T.~Jakubowski,
\emph{On exit time of stable processes}.
Stoch. Proc. Appl. 122(1) (2012): 31--41.

\bibitem{bib:g13}
T.~Grzywny,
\emph{On Harnack inequality and H\"older regularity for isotropic unimodal L\'evy processes}.
Preprint, 2013, arXiv:1301.2441v2.

\bibitem{bib:h69}
C.~C.~Heyde,
\emph{On the maximum of sums of random variables and the supremum functional for stable processes}.
J.~Appl. Probab. 6 (1969): 419--429.

\bibitem{bib:hk11}
F.~Hubalek, A.~Kuznetsov,
\emph{A convergent series representation for the density of the supremum of a stable process}.
Elect. Comm. Probab. 16 (2011): 84--95.

\bibitem{bib:jf12}
W.~Jedidi, S.~Fourati,
work in progress.

\bibitem{bib:ksv12}
P.~Kim, R.~Song, Z.~Vondra\v{c}ek,
\emph{Potential theory of subordinate Brownian motions revisited}.
In: T.~Zhang, X.~Zhou (Eds.), \emph{Stochastic Analysis and Applications to Finance--Essays in Honour of Jia-an Yan}, World Scientific, 2012, 243--290.

\bibitem{bib:ksv13}
P.~Kim, R.~Song, Z.~Vondra\v{c}ek,
\emph{Potential theory of subordinate Brownian motions with Gaussian components}.
Stoch. Proc. Appl. 123 (2013): 764--795.

\bibitem{bib:ksv14}
P.~Kim, R.~Song, Z.~Vondra\v{c}ek,
\emph{Global uniform boundary Harnack principle with explicit decay rate and its application}.
Stoch. Proc. Appl. 124 (2014): 235--267

\bibitem{bib:kkms10}
T.~Kulczycki, M.~Kwa{\'s}nicki, J.~Ma{\l}ecki, A.~St{\'o}s,
\emph{Spectral properties of the Cauchy process on half-line and interval}.
Proc. London Math. Soc. 101(2) (2010): 589--622.

\bibitem{bib:ku10a}
A.~Kuznetsov,
\emph{Analytic Proof of Pecherski\u{\i}--Rogozin Identity and Wiener--Hopf Factorization}.
Theory Probab. Appl. 55(3) (2010): 432--443.

\bibitem{bib:ku10}
A.~Kuznetsov,
\emph{Wiener-Hopf factorization and distribution of extrema for a family of L\'{e}vy processes}.
Ann. Appl. Prob. 20(5) (2010): 1801--1830.

\bibitem{bib:ku11}
A.~Kuznetsov,
\emph{On extrema of stable processes}.
Ann. Probab. 39(3) (2011): 1027--1060.

\bibitem{bib:ku13}
A.~Kuznetsov,
\emph{On the density of the supremum of a stable process}.
Stoch. Proc. Appl. 123(3) (2013): 986--1003.

\bibitem{bib:kkp13}
A.~Kuznetsov, A.~Kyprianou, J.~C.~Pardo,
\emph{Meromorphic L\'{e}vy processes and their fluctuation identities}.
Ann. Appl. Probab. 22(3) (2012): 1101--1135.

\bibitem{bib:k06}
A.~E.~Kyprianou,
\emph{Introductory Lectures on Fluctuations of L\'{e}vy Processes with Applications}.
Universitext, Springer-Verlag, Berlin, 2006.

\bibitem{bib:k11}
M.~Kwa{\'s}nicki,
\emph{Spectral analysis of subordinate Brownian motions on the half-line}.
Studia Math. 206(3) (2011): 211--271.

\bibitem{bib:k12a}
M.~Kwa{\'s}nicki,
\emph{Eigenvalues of the fractional Laplace operator in the interval}.
J.~Funct. Anal. 262(5) (2012): 2379--2402.

\bibitem{bib:k12}
M.~Kwa\'snicki,
\emph{Spectral theory for one-dimensional symmetric L\'evy processes killed upon hitting the origin}.
Electron. J. Probab. 17 (2012), 83:1--29.

\bibitem{bib:kmr12}
M.~Kwa{\'s}nicki, J.~Ma{\l}ecki, M.~Ryznar,
\emph{First passage times for subordinate Brownian motions}.
Stoch. Proc. Appl 123 (2013): 1820--1850.

\bibitem{bib:kmr13}
M.~Kwa{\'s}nicki, J.~Ma{\l}ecki, M.~Ryznar,
\emph{Suprema of L\'evy processes}.
Ann. Probab. 41(3B) (2013): 2047--2065.

\bibitem{bib:m13}
Z.~Michna,
\emph{Explicit formula for the supremum distribution of a spectrally negative stable process}.
Electron. Commun. Probab. 18 (2013), 10:1--6.

\bibitem{bib:p14}
H.~Pant\'{\i},
\emph{On L\'evy processes conditioned to avoid zero}.
Preprint, 2013, arXiv:1304.3191v1.

\bibitem{bib:pr69}
E.A.~Pecherski, B.A.~Rogozin,
\emph{The joint distributions of random variables associated to fluctuations of a process with independent increments}.
Teor. Veroyatnost. Primenen. 14(3) (1969): 431--444; English transl. in Theory Probab. Appl. 14(3) (1969): 410--423.

\bibitem{bib:r83}
L.C.G.~Rogers,
\emph{Wiener--Hopf factorization of diffusions and L\'evy processes}.
Proc. London Math. Soc. 47(3) (1983): 177--191.

\bibitem{bib:r66}
B.A.~Rogozin
\emph{Distribution of certain functionals related to boundary problems for processes with independent increments}.
Teor. Veroyatnost. i Primenen. 11(4) (1966): 656--670.

\bibitem{bib:s99}
K.~Sato,
\emph{L{\'e}vy Processes and Infinitely Divisible Distributions}.
Cambridge Univ. Press, Cambridge, 1999.

\bibitem{bib:ssv10}
R.~Schilling, R.~Song, Z.~Vondra\v{c}ek,
\emph{Bernstein Functions: Theory and Applications}.
De Gruyter, Studies in Math. 37, Berlin, 2012.

\bibitem{bib:tt02}
Y.~Tamura, H.~Tanaka,
\emph{On a fluctuation identity for multidimensional L\'evy processes}.
Tokyo J.~Math. 25 (2002): 363--380.

\bibitem{bib:tt08}
Y.~Tamura, H.~Tanaka,
\emph{On a formula on the potential operators of absorbing L\'evy processes in the half space}.
Stoch. Proc. Appl 118 (2008): 199--212.

\bibitem{bib:y12}
K.~Yano,
\emph{On harmonic function for the killed process upon hitting zero of asymmetric L\'evy processes}.
Preprint, 2012, arXiv:1211.5955v1.

\end{thebibliography}
\end{document}